\documentclass[preprint,12pt]{elsarticle}

\usepackage{amssymb}
\usepackage{amsmath}

\usepackage{graphicx} 
\usepackage[utf8]{inputenc}
\usepackage[english]{babel}
\usepackage{tabularx}
\usepackage{amsthm}
\usepackage{tikz}
\usetikzlibrary{hobby,shapes.misc,decorations.pathreplacing} 
\usetikzlibrary{hobby}
\usepackage{graphicx}
\usepackage{wrapfig}
\usepackage[labelfont=bf]{caption}
\usepackage{extarrows}
\usepackage{algpseudocode}
\usepackage{algorithm}
\usepackage[most]{tcolorbox}
\usepackage{xcolor}
\usepackage{hyperref}
\usepackage{cleveref}
\usepackage{enumitem} 
\usepackage{mathrsfs}
\usepackage{multicol}
\usepackage{dblfloatfix}
\usetikzlibrary{patterns}
\usepackage[figure]{hypcap}
\usepackage[a4paper, left=3cm, right=3cm, top=3cm]{geometry}

\usetikzlibrary{calc}
\usetikzlibrary{hobby,shapes.misc,decorations.pathreplacing} 
\usetikzlibrary{hobby}

\newcommand{\norm}[1]{\left\lVert#1\right\rVert}
\newcommand{\enorm}[1]{{\left\vert\kern-0.25ex\left\vert\kern-0.25ex\left\vert #1\right\vert\kern-0.25ex\right\vert\kern-0.25ex\right\vert}}
\newtheorem{theorem}{Theorem}[section]
\newtheorem{Cor}[theorem]{Corollary}
\newtheorem{lemma}[theorem]{Lemma}
\newtheorem{definition}[theorem]{Definition}
\newtheorem{Bsp}[theorem]{Example}
\newtheorem{Rem}[theorem]{Remark}

\newtheorem{Interpretation}[theorem]{Interpretation}

\newtheorem{dar}[theorem]{Definition and Remark}

\newtheorem{Prop}[theorem]{Proposition}
\newtheorem{innercustomgeneric}{\customgenericname}
\providecommand{\customgenericname}{}

\newcommand{\einschraenkung}{\,\rule[-5pt]{0.4pt}{12pt}\,{}}
\DeclareMathOperator{\essinf}{ess\,inf}
\DeclareMathOperator{\esssup}{ess\,sup}

\journal{Journal of Nonlinear Analysis}

\begin{document}

\begin{frontmatter}


 \title{Nonlocal Problems Governed by Symmetric Nonlocal Operators \tnoteref{label1}}
 \tnotetext[label1]{This work has been supported by the German Research Foundation (DFG) within the Research Training Group 2126: 
“Algorithmic Optimization”.}

\author[1]{Leonhard Frerick}
\ead{frerick@uni-trier.de}

\author[1]{Julia Huschens\corref{cor1}}
\ead{huschens@uni-trier.de}

\author[3]{Michael Vu}
\ead{michael.vu@tu-dresden.de}
 \cortext[cor1]{Corresponding author.}
\affiliation[1]{organization={Universität Trier},
             addressline={FB IV, Mathematik},
             postcode={54286 Trier},
             country={Germany}}
\affiliation[3]{organization={Technische Universität Dresden},
             addressline={Institut für Geometrie},
             postcode={01062 Dresden},
             country={Germany}}




\begin{abstract}
Nonlocal boundary value problems with Dirichlet or Neumann boundary are well-studied for nonlocal operators of the type $\mathcal{L}_\gamma u = \operatorname{PV} \int_{\mathbb{R}^d} \big(u(\cdot)-u(y)\big) \gamma(\cdot,y) \, \mathrm{d}y$ where the underlying kernel function $\gamma: \mathbb{R}^d \times \mathbb{R}^d \rightarrow [0,\infty)$ is assumed to be measurable and symmetric. In this paper, a theory is introduced for problems whose governing operator is of the more general type 
\[\mathcal{L}u:= \operatorname{PV} \int_{\mathbb{R}^d}\big(u(\cdot)-u(y)\big) \, K(\cdot, \mathrm{d}y)\]
where ${K: \mathbb{R}^d \times \mathcal{B}(\mathbb{R}^d) \rightarrow [0,\infty]}$  is a \textit{symmetric} transition kernel. Our main focus is on nonlocal Dirichlet and Neumann problems and a classical Hilbert space approach is developed for solving designated weak formulations. As an example, the discrete Poisson problem on $\Omega=(0,1)^d$ is discussed.
\end{abstract}


\end{frontmatter}



\section{Introduction}

When natural phenomena and data-driven coherences are driven by dynamics which are not purely local, they cannot be described satisfactorily by partial differential equations (PDEs). As a consequence, mathematical models governed by nonlocal operators are of interest and within the last decades, such models have progressively appeared in a wide range of contextes (peridyanmics: \cite{du2011mathematical}, \cite{silling2000reformulation}, \cite{MR2733097}; image processing and denoising: \cite{MR2608636}, \cite{MR2480109}, \cite{lou2010image}; medicine: \cite{MR4769849}, \cite{cusimano2015order}; finance and stochastic processes: \cite{MR2465826}, \cite{MR3709057}, \cite{MR1809268}). While local operators $\mathcal{L}$ are distinctively characterized by the fact that it suffices to know a function $u:\mathbb{R}^d \rightarrow \mathbb{R}$ in an arbitrarily small neighborhood of $x \in \mathbb{R}^d$ in order to determine $\mathcal{L}u(x)$, nonlocal operators are those operators where this is not the case. Typically, nonlocal operators arise in the form of integral operators whereas local operators are given by (partial) differentiation operators. \\

\noindent
Let $\mathcal{B}(\mathbb{R}^d)$ denote the Borel-$\sigma$-algebra on $\mathbb{R}^d$, let $\emptyset \neq \Omega \subseteq \mathbb{R}^d$ be a bounded and open set, and let $u: \mathbb{R}^d \rightarrow \mathbb{R}$ be Borel measurable. In this article, we are concerned with steady-state nonlocal equations of the form 
\begin{equation}\label{nonlocal_equation}
    \mathcal{L}u = f \text{ on } \Omega
\end{equation}
which are equipped with either a Dirichlet- or Neumann-type boundary and whose governing nonlocal operators
\begin{equation}\label{Nonlocal_operator}
    \mathcal{L} u(x):=\operatorname{PV} \int_{\mathbb{R}^d} \big(u(x)-u(y)\big) \, K(x, \mathrm{d}y), \ x \in \mathbb{R}^d,
\end{equation}
are determined by a \textit{transition kernel} $K:\mathbb{R}^d \times \mathcal{B}(\mathbb{R}^d) \rightarrow [0,\infty]$, i.\,e. a mapping satisfying
 \begin{enumerate}[label=(P\arabic*)]
     \item For all $x \in \mathbb{R}^d$ it holds that $K(x,\cdot)$ is a $\sigma$-finite measure on $\big(\mathbb{R}^d, \mathcal{B}(\mathbb{R}^d)\big)$; \label{P1}
     \item For all $A \in \mathcal{B}(\mathbb{R}^d)$, the mapping $K(\cdot,A): \mathbb{R}^d \rightarrow [0,\infty]$ is Borel measurable. \label{P2}
 \end{enumerate}
Hereinafter, we study for a specific class of kernels $K$ \textendash{} namely those satisfying a designated symmetry condition, see \Cref{Sec_Symmetry} \textendash{} under which particular conditions variational solutions to such equations exist, when they are unique, and when they depend continuously on the input data (well-posedness in the sense of Hadamard). \\

\noindent
Let us outline why we think that the above setting is interesting to consider: In the last decades, the well-posedness of boundary value problems governed by nonlocal operators have been subject to a vast amount of studies (e.g. \cite{foghem2022general}, \cite{Frerick2022TheNN}; see below for more references). Quite remarkably though \textendash{} and this is what factually motivates the current article \textendash{} mostly one particular realization of $\mathcal{L}$ has been taken into account so far: the nonlocal diffusion operator $\mathcal{L}_\gamma$,
\begin{equation}\label{nonlocal_difusion_operator}
    \begin{aligned}
        \mathcal{L}_\gamma u(x) &:= \operatorname{PV} \int_{\mathbb{R}^d} \big(u(x)-u(y)\big)\gamma(x,y) \, \mathrm{d}y \\
        &= \operatorname{PV} \int_{\mathbb{R}^d} \big(u(x)-u(y)\big) \, K_\gamma(x, \mathrm{d}y), \ x \in \mathbb{R}^d, 
     \end{aligned}
\end{equation}
where for a Borel measurable kernel and symmetric function $\gamma: \mathbb{R}^d \times \mathbb{R}^d \rightarrow [0,\infty)$, the (\textit{symmetric}) transition kernel $K$ is given by $K_\gamma:\mathbb{R}^d \times \mathcal{B}(\mathbb{R}^d) \rightarrow [0,\infty]$,
\begin{equation}\label{K_gamma}
        K_\gamma(x,A) := \int_A \gamma(x,y) \, \mathrm{d}x.
\end{equation}
However, since also many other choices of nonlocal operators are conceivable (and of interest), it is the object of this article to replace the mapping $K_\gamma$ by an arbitrary transition kernel $K$ in order to augment the scope of nonlocal operators considered in nonlocal models substantially. Among others, our approach is designed to encompass the consideration of the following operators $\mathcal{L}$, see \Cref{Sec_Symmetry} for a definition of the corresponding mappings $K$:
\begin{enumerate}
    \item \textbf{The Nonlocal Diffusion Operator $\mathbf{\mathcal{L}_\gamma}$:} In this case, our framework and our well-posedness results seek to be consistent with the ones given by \cite{foghem2022general} and \cite{Frerick2022TheNN}. Note that for $s \in (0,1)$, $\delta>0$, and $C_{d,s}$ denoting the normalization constant from \cite[(3.2)]{di2012hitchhikerʼs}, important realizations of (\ref{nonlocal_difusion_operator}) are given by
    \begin{align*}
        (-\Delta)^s u(x) &:= \operatorname{PV} \int_{\mathbb{R}^d} C_{d,s} \frac{\big(u(x)-u(y)\big)}{\norm{x-y}^{d+2s}} \, \mathrm{d}y, \tag{\textbf{Fractional Laplacian}}\\
        \mathcal{L}_\delta u(x) &:= \int_{\mathbb{R}^d} \big(u(x)-u(y)\big) \chi_{\{\norm{x-y} < \delta\}} \, \mathrm{d}y. \tag{\textbf{Truncation Operator}}
    \end{align*}
    In what follows, the class of nonlocal diffusion operators $\mathcal{L}_\gamma$ will be denoted by $\mathscr{D}$. Note that if not mentioned explicitly otherwise, $\gamma: \mathbb{R}^d \times \mathbb{R}^d \rightarrow [0,\infty)$ is always assumed to be a symmetric kernel function.
    \item \textbf{The $d$-Dimensional Stencil Operator $\mathbf{\mathcal{L}_{d,h}}$:}  Let $e_i:=(0,0,...,0,1,0,...,0)^\top$ be the $i$-th standard unit vector in $\mathbb{R}^d$ and $h>0$. Then, the operator of interest is given by
    \begin{equation}\label{defi_stencil_operator}\index{Operator ! $\mathcal{L}_{d,h}$}
        \mathcal{L}_{d,h} u(x):= \frac{1}{h^2}\Big( 2d \, u(x) - \sum_{i=1}^d \big(u(x+he_i)+u(x-he_i)\big)\Big).
    \end{equation}
    Especially, in the numerics of partial differential equations, this operator is of fundamental interest as it provides the backbone for solving the Poisson equation $-\Delta u = f$ on a computer, see \Cref{Sec_Applications}. 
    \item \textbf{The $\epsilon$-Sphere Operator $\mathbf{\mathcal{L}_{\epsilon}}$: } Let $\epsilon>0$ be given and let $\norm{\cdot}$ denote any norm on $\mathbb{R}^d$. In this example, we are concerned with the integral operator
    \begin{equation}\label{defi_sphere_operator}\index{Operator ! $\mathcal{L}_\epsilon$}
        \mathcal{L}_{\epsilon}u(x) := \int_{\partial B_\epsilon(x)} \big(u(x)-u(y)\big) \, \mathrm{d}H^{d-1}(y),
    \end{equation}
    for which the integration over the $\epsilon$-sphere $\partial B_\epsilon(x) := \{y \in \mathbb{R}^d: \norm{x-y}=\epsilon\}$ takes place with respect to the Hausdorff measure $H^{d-1}$. We refer to \cite[Section 8.2]{Huschens} for a study of nonlocal boundary value problems in the context of this operator. 
\end{enumerate}
 Even though the operators in (\ref{defi_stencil_operator}) and (\ref{defi_sphere_operator}) are clearly nonlocal, they can notably not be expressed as a nonlocal diffusion operator $\mathcal{L}_\gamma$ for
 any $\gamma: \mathbb{R}^d \times \mathbb{R}^d \rightarrow [0,\infty)$ Borel measurable because the set of points needed to evaluate $\mathcal{L}_{d,h}u(x)$ and $\mathcal{L}_\epsilon u(x)$, $x \in \mathbb{R}^d$, is always a set of zero Lebesgue measure.\\

\noindent
Assume that $\emptyset \neq \Omega \subseteq \mathbb{R}^d$ is an open set and that $f: \Omega \rightarrow \mathbb{R}$ and $g: \mathbb{R}^d \setminus \Omega \rightarrow \mathbb{R}$ are Borel measurable functions. For a \textit{symmetric} transition kernel $K$, we are interested in finding solutions $u: \mathbb{R}^d \rightarrow \mathbb{R}$ of the problems \\
\begin{minipage}{0.2\textwidth}
    \[ \hspace{5.3cm} \mathcal{L}u=f \text{ in } \Omega,\]
\end{minipage}
\begin{minipage}{0.8\textwidth}
    \begin{align*}
    u&=g \text{ in } \Gamma, \tag{DP}\\
    \mathcal{N}u &=g \text{ in } \Gamma, \tag{NP} 
    \end{align*}
\end{minipage} \\
\\
where $\mathcal{N}= \mathcal{N}_{\Omega,K}$ is the \textit{nonlocal Neumann operator} defined on the \textit{nonlocal boundary} \linebreak
\newpage
\noindent
$\Gamma = \Gamma(\Omega,K)$ of $\Omega$,
\begin{align}
    \mathcal{N}u(y)&:= \int_\Omega \big(u(y)-u(x)\big)\, K(y, \mathrm{d}x), \quad y \in \Gamma, \label{Neumann_operator} \\
    \Gamma &:=\left\{y \in \mathbb{R}^d\setminus \Omega: K(y,\Omega)>0\right\}.
\end{align}
Note that because in general the pointwise evaluation of $\mathcal{L}u(x)$, $x \in \mathbb{R}^d$, may fail even for a smooth function $u: \mathbb{R}^d \rightarrow \mathbb{R}$, the two problems above are evaluated in a weak/variational sense only which is facilitated by a nonlocal integration by parts formula in which $\mathcal{N}$ takes the place of the normal derivative $\partial_\eta$ while $\Gamma$ substitutes the topological boundary $\partial \Omega$ of $\Omega$. In the following, problem (\ref{DP}) will be referred to as the \textit{nonlocal Dirichlet problem} whereas \textendash{} motivated by the correspondence of $\mathcal{N}$ and $\partial_\eta$ \textendash{} problem (\ref{NP}) will be called the \textit{nonlocal Neumann problem}. \\

\noindent
Let us comment on related work in literature centering around the nonlocal diffusion operator $\mathcal{L}_\gamma \in \mathscr{D}$: To the best of our knowledge, \cite{gunzburger2010nonlocal} were the first to formulate and consider the nonlocal equation (\ref{nonlocal_equation}) for $\mathcal{L}= \mathcal{L}_\gamma \in \mathscr{D}$. By now, a plenty of literature exists in which the well-posedness of the latter equation is addressed under different choices of boundary conditions. In particular the corresponding nonlocal Dirichlet problem has attracted a lot of attention and was studied for instance in \cite{du2012analysis} and \cite{felsingerDirichlet} (see \cite{ros2016nonlocal} for a survey); also for not necessarily symmetric kernels $\gamma$ (e.g. \cite{felsingerDirichlet}).  In contrast, contributions to the nonlocal Neumann problem were made by various authors and for different concepts of a ``nonlocal normal derivative" $\mathcal{N}$. While deviant approaches of how to set up $\mathcal{N}$ can be found in \cite{barles2014neumann}, \cite{cortazar2007boundary}, or \cite{du2012analysis}, our formulation is inspired by the work of \cite{dipierro2017nonlocal} who studied problem (\ref{NP}) for $\mathcal{L} = (-\Delta)^s$, $s \in (0,1)$ and $\mathcal{N}= \mathcal{N}_{\Omega,\gamma_s}$, $\gamma_s(x,y) := C_{d,s} \frac{1}{\norm{x-y}^{d+2s}}$. In 2024, a modified approach was utilized by \cite{foghem2022general} to tackle the complement value problem $\mathcal{L}_\gamma u =f$ in $\Omega$, $\mathcal{N}_{\Omega,K_\gamma} u =g$ in $\mathbb{R}^d \setminus \Omega$ for all measurable $\gamma: \mathbb{R}^d \times \mathbb{R}^d \setminus \operatorname{diag} \rightarrow [0,\infty)$ satisfying, for some $\Lambda \geq 1$,
\begin{equation}
    \Lambda^{-1} \, \nu(y-x) \leq \gamma(x,y) \leq \Lambda \, \nu(y-x), \quad x,y \in \mathbb{R}^d,
\end{equation}
where $\nu: \mathbb{R}^d \setminus \{0\}\rightarrow [0,\infty)$ is assumed to be a the density of a symmetric Lévy measure, i.\,e.,
\[\nu(x)= \nu(-x) \text{ for } x\in \mathbb{R}^d \setminus \{0\} \quad \text{ and } \quad \int_{\mathbb{R}^d \setminus \{0\}} \min\left\{1, \norm{x}^2\right\} \nu(x) \, \mathrm{d}x < \infty.\] 
Finally and in parallel, for arbitrary symmetric and measurable kernels \linebreak $\gamma: \mathbb{R}^d \times \mathbb{R}^d \rightarrow [0,\infty)$, a universal solution theory for problem (\ref{NP}) was provided in \cite{Frerick2022TheNN} where the Neumann boundary condition was imposed on the nonlocal boundary $\Gamma = \left\{y \in \mathbb{R}^d: \int_\Omega \gamma(x,y) \, \mathrm{d}y > 0 \right\}$ only. As announced above, we emphasize that our framework is designed to be a generalization of their work.\\

\noindent
This article is organized as follows:
\begin{itemize}
    \item Let $\lambda$ be an arbitrary but fixed Borel measure $\lambda$ on $\mathbb{R}^d$. Then, in \Cref{Sec_Symmetry}, a notion of symmetry for the transition kernel $K:\mathbb{R}^d \times \mathcal{B}(\mathbb{R}^d) \rightarrow [0,\infty]$ is established via the Fubini-like integration identity
    \begin{equation}\label{Fubini_technique}
        \int_{\mathbb{R}^d} \int_{\mathbb{R}^d} f(x,y) \, K(x, \mathrm{d}y) \, \mathrm{d}\lambda(x) = \int_{\mathbb{R}^d} \int_{\mathbb{R}^d} f(x,y) \, K(y, \mathrm{d}x) \, \mathrm{d}\lambda(y),
    \end{equation}
    which is assumed to hold for all suitably integrable functions $f: \mathbb{R}^d \times \mathbb{R}^d \rightarrow \mathbb{R}$. As it will become evident in \Cref{subsection_integration_by_parts}, this identity is the key for providing a variational formulation of the boundary value problems (\ref{DP}) and (\ref{NP}) which is facilitated by a nonlocal integration by parts formula. Note that our notion of symmetry will be consistent with \cite[Identity 1-2]{dyda2020regularity} and collapse to the symmetry of $\gamma$ in case that $K=K_\gamma$ and $\lambda$ is the Lebesgue Borel measure on $\mathbb{R}^d$, cf. \Cref{dar_Symmetry} and \Cref{Bsp_symmetry_diffusion_operator}.
    \item In \Cref{Sec_Variational_Analysis}, a weak respectively variational formulation of the problems (\ref{DP}) and (\ref{NP}) is derived next which depends on the Borel measure $\lambda$ occurring in (\ref{Fubini_technique}). The associated test function spaces $V_0=V_0(\Omega,K,\lambda)$ and $V=V(\Omega,K,\lambda)$ are introduced in \Cref{subsec_Function_spaces} as the nonlocal counterparts of the Hilbert spaces $H^1_0(\Omega)$ and $H^1(\Omega)$, respectively, while the underlying integration by parts formula is proven in \Cref{subsection_integration_by_parts}.
    \item In \Cref{Sec_Dirichlet_Problem} and \Cref{Sec_Neumann_Problem}, the well-posedness of the weak formulations of the problems (\ref{DP}) and (\ref{NP}) is addressed in closer detail. More specifically, under suitable conditions on the input data $f: \Omega \rightarrow \mathbb{R}$ and $g: \Gamma \rightarrow \mathbb{R}$, a standard gadget from Hilbert space theory \textendash{} the Lax Milgram theorem \textendash{} is made applicable in order to ensure the existence of a (unique) solution of (\ref{DP}) or (\ref{NP}) which depends continuously on $f$ and $g$. Note that an integral part of this procedure is the establishment of certain nonlocal Poincaré-type inequalities which ascertain the ellipticity condition from the Lax-Milgram theorem to be satisfied. How our inequalities relate to previous versions as utilized by \cite{foghem2022general} or \cite{Frerick2022TheNN} and their classical local counterparts is examined in closer detail in \Cref{section_Inequalities}. A nonlocal weak maximum principle for the weak solutions can be found in \Cref{Sec_Maximum}.
    \item Finally, \Cref{Sec_Applications} is dedicated to a particular example. More specifically, our variational theory is aligned with the (well-known) analysis of the discrete Poisson problem (Dirichlet + Neumann boundary) where the latter is shown to be equivalent to the nonlocal problem (\ref{DP}) respectively (\ref{NP}) when $\mathcal{L}= \mathcal{L}_{d,h}$ is the $d$-dimensional stencil operator from (\ref{defi_stencil_operator}). Note that the discrete character of the problem is induced by a discrete measure $\mu$ with respect to which the variational formulation is set up. 
\end{itemize}

\noindent
\textit{Further remarks:} This work contains some of the main results of the Chapters 2-9 of the PhD thesis \cite{Huschens} of the second author Julia Huschens.

\section{A Notion Of Symmetry}\label{Sec_Symmetry}

Throughout this section, we set up the basic concept of our framework by specifying a notion of symmetry for the transition kernel $K$. Using an analogous approach as \cite{bezuglyi}, we therefore consider symmetric Borel measures $\mu$ on $\mathbb{R}^d \times \mathbb{R}^d$:

\begin{definition}{\textnormal{(Symmetric Borel Measures on $\mathbb{R}^d\times \mathbb{R}^d$)}}\label{Defi_symmetric_Borel_measure}
Let $\mu$ be a Borel measure on $\mathbb{R}^d \times \mathbb{R}^d$. We say that $\mu$ is \textbf{symmetric} if $\mu$ is invariant under the flip automorphism $T: (x,y)\mapsto(y,x)$, i.\,e. if
\[\mu = \mu^T \] 
where $\mu^T$ is the pushforward measure of $\mu$ by $T$. Especially, for all $A,B \in \mathcal{B}(\mathbb{R}^d)$, we have $\mu(A \times B) = \mu(B \times A)$.
\end{definition}

In the course of this section, we shall always assume that $\lambda$ is an arbitrary but fixed Borel measure on $\mathbb{R}^d$ and that $K$ is a transition kernel. Henceforth, we further let $\mathcal{M}^{(+)}$ denote the set of all (non-negative) Borel measurable functions $f: \mathbb{R}^d \times \mathbb{R}^d \rightarrow \mathbb{R}$ and let $\mathcal{B}(\mathbb{R}^d)$ signify the Borel-$\sigma$-algebra on $\mathbb{R}^d$.

\begin{definition}\label{non_degenerated_transitione_kernels}
    A transition kernel $K: \mathbb{R}^d \times \mathcal{B}(\mathbb{R}^d) \rightarrow [0,\infty]$ is called \textbf{non-degenerated} if the mapping
    \begin{equation}\label{Def_nondegenerated}
        \mathbb{R}^d \ni x \mapsto \int_{\mathbb{R}^d} f(x,y)\, K(x,\mathrm{d}y)
    \end{equation} 
    is Borel measurable for all $f \in \mathcal{M}^+$. Throughout the following, the set of all non-degenerated transition kernels shall be denoted by $\mathcal{K}$.
\end{definition}

\textbf{Remark:} Due to Tonelli's theorem, the kernel $K_\gamma$ depicted in (\ref{K_gamma}) is non-degenerated because the Lebesgue Borel measure $\lambda= \lambda_d$ on $\mathbb{R}^d$ is $\sigma$-finite. If $K$ is a finite transition kernel in the sense that $K(x,\cdot) \leq C$ for all $x \in \mathbb{R}^d$, the kernel $K$ is non-degenerated as well, see \cite[Satz 3.12]{wengenroth2008wahrscheinlichkeitstheorie}.\\

Let $\lambda$ be given as above and let $K \in \mathcal{K}$. Then, two measures $\lambda \otimes K$ and $K \otimes \lambda$ on the product space $\big(\mathbb{R}^d \times \mathbb{R}^d, \mathcal{B}(\mathbb{R}^d) \otimes \mathcal{B}(\mathbb{R}^d)\big)$ are given by
\begin{align}
    (\lambda \otimes K)(M) &:= \int_{\mathbb{R}^d}\int_{\mathbb{R}^d} \chi_M(x,y) \, K(x, \mathrm{d}y) \, \mathrm{d}\lambda(x), \label{Maß1}\\
    (K \otimes \lambda)(M) &:= \int_{\mathbb{R}^d}\int_{\mathbb{R}^d} \chi_M(x,y) \, K(y, \mathrm{d}x) \, \mathrm{d}\lambda(y), \label{Maß2}
\end{align}
and, we have $K \otimes \lambda = (\lambda \otimes K)^T$. In view of \Cref{Defi_symmetric_Borel_measure}, our notion of symmetry for $K \in \mathcal{K}$ is the following:

\begin{dar}\label{dar_Symmetry}
    Let $\lambda$ be a Borel measure on $\mathbb{R}^d$. Then, a transition kernel $K \in \mathcal{K}$ is called \textbf{$\mathbf{\lambda}$-symmetric} if $\lambda \otimes K$ is a symmetric Borel measure on $\mathbb{R}^d \times \mathbb{R}^d$, i.\,e., we have
    \[\lambda \otimes K = K \otimes \lambda.\] 
    If $\lambda= \lambda_d$ is the Lebesgue Borel measure on $\mathbb{R}^d$ and $K \in \mathcal{K}$ is such that $\lambda_d \otimes K$ is a $\sigma$-finite Borel measure on $\mathbb{R}^d \times \mathbb{R}^d$, our notion of symmetry is exactly \cite[Identity (1-2)]{dyda2020regularity}. 
\end{dar}

We summarize some important facts with respect to the measures $\lambda \otimes K$ and $K \otimes \lambda$:

\begin{Rem}\label{Rem_Produktmaße} Let $K \in \mathcal{K}$ be non-degenerated.  
\begin{itemize}
    \item Then, the non-degeneracy of $K$ ensures the double integrals appearing within (\ref{Maß1}) and (\ref{Maß2}) to be well-defined. 
    \item For $A,B \in \mathcal{B}(\mathbb{R}^d)$, it is valid that both
            \begin{align}
                (\lambda \otimes K)(A\times B) &= \int_A K(x,B) \, \mathrm{d}\lambda(x), \label{ref1}\\
                (K \otimes \lambda)(A \times B) &= \int_B K(x,A) \, \mathrm{d}\lambda(y). \label{ref2} 
            \end{align}
            If $\lambda$ is $\sigma$-finite and there exists a sequence $(B_n)_{n \in \mathbb{N}}$ of Borel measurable sets with $B_n \subseteq B_{n+1}$ and $\bigcup_{n\in \mathbb{N}} B_n = \mathbb{R}^d$ such that $\sup_{x \in \mathbb{R}^d}K(x, B_n) < \infty$, then $\lambda \otimes K$ and $K \otimes \lambda$ are both $\sigma$-finite and uniquely determined by (\ref{ref1}) respectively (\ref{ref2}) (cf. \cite[Theorem 3.13]{wengenroth2008wahrscheinlichkeitstheorie}).
    \item Moreover (cf. \cite[Theorem 3.14]       {wengenroth2008wahrscheinlichkeitstheorie}),
        \begin{align}
            \int_{\mathbb{R}^d \times \mathbb{R}^d} f \, \mathrm{d}(\lambda \otimes K) &= \int_{\mathbb{R}^d}\int_{\mathbb{R}^d} f(x,y) \,K(x,\mathrm{d}y)\, \mathrm{d}\lambda(x),\\
        \int_{\mathbb{R}^d \times \mathbb{R}^d} f \, \mathrm{d}(K \otimes \lambda) &= \int_{\mathbb{R}^d} \int_{\mathbb{R}^d} f(x,y) \, K(y, \mathrm{d}x) \, \mathrm{d}\lambda(y),
    \end{align}
    holds for all $f \in \mathcal{M}\big(\mathbb{R}^d \times \mathbb{R}^d, \mathcal{B}(\mathbb{R}^d) \otimes \mathcal{B}(\mathbb{R}^d)\big)$ which are quasi-integrable with respect to $\lambda \otimes K$ respectively $K \otimes \lambda$ (i.\,e., either the integral over the positive part or over the negative part is finite).
\end{itemize}
\end{Rem}

As shown next, the $\lambda$-symmetry of $K \in \mathcal{K}$ can be characterized by a Fubini-like integration identity. In \Cref{subsection_integration_by_parts}, the latter is of key interest for establishing a nonlocal integration by parts formula which contains the one from \cite[Theorem 2.4; $\gamma$ symmetric]{Frerick2022TheNN} for the nonlocal diffusion operator $\mathcal{L}_\gamma \in \mathscr{D}$ as a special case, cf. \Cref{subsection_integration_by_parts}.

\begin{lemma}\label{lemma_symmetry_Fubini}
    Let $\lambda$ be a Borel measure on $\mathbb{R}^d$ and let $K \in \mathcal{K}$. Then, $K$ is $\lambda$-symmetric if and only if for all $f \in \mathcal{M}$ quasi-integrable with respect to $\lambda \otimes K$, it is valid that
     \begin{equation}\label{generalisierung_fubini}
    \int_{\mathbb{R}^d} \int_{\mathbb{R}^d} f(x,y) \, K(x, \mathrm{d}y) \, \mathrm{d}\lambda(x) = \int_{\mathbb{R}^d} \int_{\mathbb{R}^d} f(x,y) \, K(y, \mathrm{d}x) \, \mathrm{d}\lambda(y).
    \end{equation}
\end{lemma}

\begin{proof}
    We only concentrate on the only-if part since the if part is straightforward from setting $f= \chi_M$, $M \in \mathcal{B}(\mathbb{R}^d) \otimes \mathcal{B}(\mathbb{R}^d)$: By algebraic induction, the identity in (\ref{generalisierung_fubini}) is obtained for all $f \in \mathcal{M}^+$; therefore, only the case of an arbitrary $f \in \mathcal{M}$ is elucidated in closer detail. Let $f \in \mathcal{M}$ be quasi-integrable with respect to $\lambda \otimes K$ and assume w.l.o.g. that $\int_{\mathbb{R}^d}\int_{\mathbb{R}^d} f^+(x,y)\, K(x,\mathrm{d}y)\, \mathrm{d}\lambda(x) < \infty$. Because (\ref{generalisierung_fubini}) is valid for all $f \in \mathcal{M}^+$, also $\int_{\mathbb{R}^d}\int_{\mathbb{R}^d} f^+(x,y)\, K(y,\mathrm{d}x) \, \mathrm{d}\lambda(y) < \infty$  holds and $f$ is quasi-integrable with respect to $K \otimes \lambda$ as well. Define $g_\pm, h_\pm: \mathbb{R}^d \rightarrow [0,\infty]$ by
\begin{align*}
    g_\pm(x) := \int_{\mathbb{R}^d} f^\pm(x,y)\, K(x,\mathrm{d}y), \quad h_\pm(y) := \int_{\mathbb{R}^d} f^\pm(x,y)\, K(y,\mathrm{d}x).
\end{align*}
Then, because $K$ is non-degenerated, the functions $g_\pm$ and $h_\pm$ are Borel measurable. Moreover, $\int_{\mathbb{R}^d} g_+(x)\, \mathrm{d}\lambda(x) = \int_{\mathbb{R}^d}\int_{\mathbb{R}^d} f^+(x,y) \, K(x,\mathrm{d}y) \, \mathrm{d}\lambda(x) < \infty$ is valid and, hence, the set $A:= \{x \in \mathbb{R}^d: g_+(x) = \infty\}$ is of zero measure, i.\,e., $\lambda(A) = 0$. Analogously, we see that the set $B:=\{y \in \mathbb{R}^d: h_+(y) = \infty\}$ is of zero measure as well which implies that the mapping $g(x) := g^+(x)-g^-(x)$ is defined for $x \notin A$ and that the mapping $h(y) := h_+(y)-h_-(y)$ is defined for $y \notin B$. As a consequence, 
\begin{align*}
 \int_{\mathbb{R}^d} \int_{\mathbb{R}^d} f(x,y) \, K(x, \mathrm{d}y) \, \mathrm{d}\lambda(x) &= \int_{\mathbb{R}^d \setminus A} g(x)\, \mathrm{d}\lambda(x) \\
 &= \int_{\mathbb{R}^d} g_+(x)\, \mathrm{d}\lambda(x) - \int_{\mathbb{R}^d} g_-(x)\, \mathrm{d}\lambda(x) \\
 &= \int_{\mathbb{R}^d}h_+(y)\, \mathrm{d}\lambda(y) - \int_{\mathbb{R}^d} h_-(y)\,\mathrm{d}\lambda(y)\\
 &= \int_{\mathbb{R}^d \setminus B} h(y)\, \mathrm{d}\lambda(y) = \int_{\mathbb{R}^d} \int_{\mathbb{R}^d} f(x,y) \, K(y, \mathrm{d}x) \, \mathrm{d}\lambda(y) 
\end{align*}
which yields the assertion.
\end{proof}

\noindent
Following our definition of $\lambda$-symmetric transition kernels, some explicit examples shall be depicted next. Therein, the transition kernels $K$ underlying the diffusion operator $\mathcal{L}_\gamma \in \mathscr{D}$, the stencil operator $\mathcal{L}_{d,h}$, and the $\epsilon$-sphere operator $\mathcal{L}_\epsilon$ are examined for symmetry regarding reasonable choices of Borel measures $\lambda$ on $\mathbb{R}^d$. A connection between the symmetry of $\gamma$ and our notion of symmetry for $K_\gamma$ is provided in \Cref{Bsp_symmetry_diffusion_operator}.

\begin{Bsp}\label{Bsp_symmetry_diffusion_operator}{\textnormal{\textbf{(The Nonlocal Diffusion Operator $\mathbf{\mathcal{L}_\gamma \in \mathscr{D}}$)}}}
Let $\lambda= \lambda_d$ be the Lebesgue Borel measure on $\mathbb{R}^d$ and let $\gamma: \mathbb{R}^d \times \mathbb{R}^d \rightarrow [0,\infty)$ be Borel measurable. From \Cref{lemma_symmetry_Fubini} and the Fubini theorem it is evident that $K_\gamma \in \mathcal{K}$ is $\lambda_d$-symmetric if and only if $\gamma$ is ($\lambda \otimes \lambda$-a.e.) symmetric on $\mathbb{R}^d \times \mathbb{R}^d$.  
\end{Bsp}

\begin{Bsp}\label{Bsp_Self_adjointness_Stencil}{\textnormal{\textbf{(The Stencil Operator $\mathbf{\mathcal{L}_{d,h}}$)}}}
 Let $d \in \mathbb{N}$ and $h>0$ be fixed and let $\delta_x$ denote the Dirac measure in the point $x \in \mathbb{R}^d$. Then, the transition kernel $K_{d,h}:\mathbb{R}^d \times \mathcal{B}(\mathbb{R}^d) \rightarrow [0,\infty]$ underlying the definition of $\mathcal{L}_{d,h}$ is given by
 \begin{equation}\label{kernel_stencil}
     K_{d,h}(x,A):= \frac{1}{h^2}\sum_{i=1}^d \big(\delta_{x+h}(A) + \delta_{x-h}(A) \big),
 \end{equation}
 i.\,e., $\mathcal{L}_{d,h}u(x) = \int_{\mathbb{R}^d} \big(u(x)-u(y)\big) \, K_{d,h}(x, \mathrm{d}y)$ for $x \in \mathbb{R}^d$.
 Let us examine the symmetry of $K_{d,h}$ with respect to the Lebesgue Borel measure $\lambda_d$ on $\mathbb{R}^d$ and the discrete measure $\mu_s$,
 \[\mu_s(A) := \sum_{x \in G_s} \delta_x\]
 where for $s=(s_1,...,s_d) \in \mathbb{R}^d$, $G_s:= s + h \mathbb{Z}^d$ is the equidistant grid of width $h$ in $\mathbb{R}^d$ which has a node at $s$. To this end, note that for all $x \in \mathbb{R}^d$, the transition kernel $K_{d,h}(x,\cdot)$ is finite and because both measures $\lambda_d$ and $\mu_s$ are $\sigma$-finite in $\mathbb{R}^d$, \Cref{Rem_Produktmaße} yields the $\sigma$-finiteness of $(\lambda_d \otimes K_{d,h})$ and $\mu_s \otimes K_{d,h}$ on $\big(\mathbb{R}^d \times \mathbb{R}^d, \mathcal{B}(\mathbb{R}^d) \otimes \mathcal{B}(\mathbb{R}^d)\big)$. Due to the translation invariance of $\lambda_d$, it follows for all $A, B \in \mathcal{B}(\mathbb{R}^d)$ that
 \begin{equation*}\label{Symmetry_stencil_rechnung}
 \begin{aligned}
    (K_{d,h} \otimes \lambda_d)(A \times B) &= \frac{1}{h^2}\int_B \sum_{i=1}^d \Big(\delta_{x+he_i}(A) + \delta_{x-he_i}(A)\Big) \, \mathrm{d}x \\
    &= \frac{1}{h^2} \sum_{i=1}^d \lambda_d\big(B \cap (A-he_i)\big) + \lambda_d \big(B \cap (A+he_i)\big)\\
    &= \frac{1}{h^2} \sum_{i=1}^d \lambda_d\big((B+he_i) \cap A\big) + \lambda_d\big((B-he_i) \cap A\big)\\
    &= \frac{1}{h^2} \int_A \sum_{i=1}^d \Big( \delta_{x+he_i}(B) + \delta_{x-he_i}(B) \Big)\, \mathrm{d}x= (\lambda_d \otimes K_{d,h})(A \times B) . 
\end{aligned}
\end{equation*}
Analogously, since $\mu_s(C) = \mu_s(C+he_i) = \mu_s(C-he_i)$ holds for all $C \in \mathcal{B}(\mathbb{R}^d)$ and $i =1,...,d$, we also have
\[(K_{d,h} \otimes \mu_s)(A \times B)= (\mu_s \otimes K_{d,h})(A \times B), \quad A,B \in \mathcal{B}(\mathbb{R}^d).\]
By the uniqueness of measures theorem, this proves that both $(K_{d,h} \otimes \lambda_d) = (\lambda_d \otimes K_{d,h})$ and $(K_{d,h} \otimes \mu_s)= (\mu_s \otimes K_{d,h})$ such that the kernel $K_{d,h} \in \mathcal{K}$ is symmetric with respect to the Lebesgue Borel measure $\lambda_d$ and the discrete measure $\mu_s$.
\end{Bsp}
    
\begin{Bsp}\label{Hausdorffkern_selbstadjungiert}{\textnormal{\textbf{(The $\mathbf{\epsilon}$-Sphere Operator $\mathbf{\mathcal{L}_\epsilon}$)}}}
Let $d \in \mathbb{N}$, $\epsilon >0$, and a norm $\norm{\cdot}$ on $\mathbb{R}^d$ be given. Then, we define
$K_\epsilon: \mathbb{R}^d \times \mathcal{B}(\mathbb{R}^d) \rightarrow [0,\infty]$ by
\begin{equation*}
    K_\epsilon(x,A) := H^{d-1} \left( \partial B_\epsilon(x) \cap A \right).
\end{equation*}
Because the translation invariance of the Hausdorff measure $H^{d-1}$ yields
\[K_\epsilon(x,A)= H^{d-1}\left( \partial B_\epsilon(0)\cap (A-x)\right) = \int_{\partial B_\epsilon(0)} \chi_A(x+y) \, \mathrm{d}H^{d-1}(y),\]
it is clear by the Tonelli theorem that $K_\epsilon$ defines a transition kernel; moreover, the measures $K_\epsilon(x, \cdot)$ are obviously finite for all $x \in \mathbb{R}^d$. In consequence, if $\lambda= \lambda_d$ is the Lebesgue Borel measure on $\mathbb{R}^d$, the measure $\lambda_d \otimes K_\epsilon$ is $\sigma$-finite by \Cref{Rem_Produktmaße}. We prove that $K_\epsilon$ is $\lambda_d$-symmetric and to this end, it once more suffices to verify the validity of
\[(K_\epsilon \otimes \lambda_d)(A \times B) = (\lambda_d \otimes K_\epsilon)(A \times B) \text{ for all } A,B \in \mathcal{B}(\mathbb{R}^d).\]
So let $A,B \in \mathcal{B}(\mathbb{R}^d)$ be given arbitrarily and let $H^{d-1}_{\partial B_\epsilon(0)} =H^{d-1}\big(\cdot \cap \partial B_\epsilon(0)\big)$ denote the trace measure with respect to $\partial B_\epsilon(0)$. By the translation invariance of the Hausdorff measure and Fubini's theorem, we have
     \begin{align*}
        &(\lambda_d \otimes   K_{\epsilon})(A \times B) = \int_A  K_{\epsilon}(x,B) \, \mathrm{d}x = \int_A H^{d-1}\big(\partial B_\epsilon(0) \cap (B-x)\big) \, \mathrm{d}x \\
        &= \int_A \int_{\mathbb{R}^d} \chi_B(x+y)\, \mathrm{d}H^{d-1}_{\partial B_\epsilon(0)}(y) \, \mathrm{d}x = \int_{\mathbb{R}^d}\int_{\mathbb{R}^d} \chi_A(x) \chi_{B-y}(x) \, \mathrm{d}x \, \mathrm{d}H^{d-1}_{\partial B_\epsilon(0)}(y) \\
        &= \int_{\mathbb{R}^d} \lambda_d\big(A \cap (B-y)\big) \, \mathrm{d}H^{d-1}_{\partial B_\epsilon(0)}(y)
    \end{align*}
     and the translation variance of the Lebesgue Borel measure $\lambda_d$ on $\mathbb{R}^d$ yields 
    \begin{align*}
        \int_{\mathbb{R}^d} \lambda_d\big(A \cap (B-y)\big) \, \mathrm{d}H^{d-1}_{\partial B_\epsilon(0)}(y)  &= \int_{\mathbb{R}^d} \lambda_d\big((A+y) \cap B\big)\, \mathrm{d}H^{d-1}_{\partial B_\epsilon(0)}(y).
    \end{align*}
    Thus, by again applying the Fubini theorem and the translation invariance of $H^{d-1}$, it follows that
    \begin{align*}
       &\int_{\mathbb{R}^d} \lambda_d\big((A+y) \cap B\big) \,  \mathrm{d}H^{d-1}_{\partial B_\epsilon(0)}(y) = \int_B \int_{\mathbb{R}^d}\chi_A(x-y) \, \mathrm{d}H^{d-1}_{\partial B_\epsilon(0)}(y)\, \mathrm{d}x\\
       &= \int_B \int_{\mathbb{R}^d} \chi_{-A+x}(y) \, \mathrm{d}H^{d-1}_{\partial B_\epsilon(0)}(y) \, \mathrm{d}x = \int_B H^{d-1}\big( \partial B_\epsilon(0) \cap (-A+x)\big) \, \mathrm{d}x \\
       &= \int_B H^{d-1}\big(\partial B_\epsilon(-x) \cap (-A)\big) \, \mathrm{d}x. 
    \end{align*}
    Now, let us define $f:\mathbb{R}^d \rightarrow \mathbb{R}^d$ by $x \mapsto -x$. Since $f$ is an isometry and it is valid that \linebreak $f\big(\partial B_\epsilon(-x) \cap(-A)\big) = \partial B_\epsilon(x) \cap A$, the invariance of the Hausdorff measure with respect to isometries yields
\begin{align*}
    \int_B H^{d-1}\big(\partial B_\epsilon(-x) \cap (-A)\big) \, \mathrm{d}x &=\int_B H^{d-1}\Big(f\big(\partial B_\epsilon(-x) \cap (-A)\big)\Big) \, \mathrm{d}x \\
    &= \int_B H^{d-1}(\partial B_\epsilon(x) \cap A)\, \mathrm{d}x \\
    &= \int_B  K_{\epsilon}(x,A)\, \mathrm{d}x = (K_{\epsilon} \otimes \lambda_d)(B \times A).
\end{align*}
From this, the $\lambda_d$-symmetry of $ K_{\epsilon}\in \mathcal{K}$ is obtained.   
\end{Bsp}

\section{Variational Analysis} \label{Sec_Variational_Analysis}

In this section, a weak respectively variational formulation of the nonlocal problems (\ref{DP}) and (\ref{NP}) is derived. Throughout the chapter, we shall always assume that $\emptyset \neq \Omega \subseteq \mathbb{R}^d$ is an open set, that $\lambda$ is an arbitrary Borel measure on $\mathbb{R}^d$, and that $K \in \mathcal{K}$ is a $\lambda$-symmetric transition kernel. Note that in many applications, $\lambda$ is typically chosen to be the Lebesgue Borel measure $\lambda_d$ on $\mathbb{R}^d$.

\subsection{Nonlocal Function Spaces}\label{subsec_Function_spaces}

To introduce the nonlocal function spaces $V_0=V_0(\Omega,K,\lambda)$ and ${V=V(\Omega,K,\lambda)}$ that are tailor-made for studying the nonlocal problems (\ref{DP}) and (\ref{NP}), we generalize the notion of the nonlocal function spaces introduced in \cite{foghem2022general} and \cite{Frerick2022TheNN} and set
\begin{align}
     \mathcal{V}:=\mathcal{V}(\Omega,K,\lambda) &:= \left\{v: \mathbb{R}^d \rightarrow \mathbb{R} \text{ Borel measurable:} \norm{v}_{\mathcal{V}} < \infty\right\},\\
     \mathcal{V}_0:=\mathcal{V}_0(\Omega,K,\lambda) &:= \left\{v \in \mathcal{V}(\Omega,K,\lambda): v \einschraenkung_{\mathbb{R}^d \setminus \Omega} =0 \right\},
\end{align}
where 
\begin{equation}\label{norm_function_space}
    \norm{v}^2_{\mathcal{V}} := \int_\Omega v(x)^2 \, \mathrm{d}\lambda(x) + \int_\Omega \int_{\mathbb{R}^d} \big(v(x)-v(y)\big)^2 \, K(x, \mathrm{d}y) \, \mathrm{d}\lambda(x).
\end{equation} 
Obviously, the seminorm $\norm{\cdot}_{\mathcal{V}}$ on $\mathcal{V}$ is induced by the semi-inner product \linebreak $\langle \cdot, \cdot \rangle_{\mathcal{V}}: \mathcal{V} \times \mathcal{V} \rightarrow \mathbb{R}$,
\begin{equation}\label{inner_product}
    \langle u, v\rangle_{\mathcal{V}} :=\int_\Omega u(x)v(x) \, \mathrm{d}\lambda(x) + \int_\Omega \int_{\mathbb{R}^d} \big(u(x)-u(y)\big)\big(v(x)-v(y)\big) \, K(x, \mathrm{d}y) \, \mathrm{d}\lambda(x)
\end{equation}
and by $\mathcal{B} = \mathcal{B}_K = \mathcal{B}_{K,\Omega}: \mathcal{V} \times \mathcal{V} \rightarrow \mathbb{R}$,
\begin{equation}\label{definition_B}
\begin{aligned}
    \mathcal{B}(u,v) &:= \frac{1}{2}\int_\Omega \int_\Omega \big(u(x)-u(y)\big)\big(v(x)-v(y)\big) \, K(x, \mathrm{d}y) \, \mathrm{d}\lambda(x) \\
    &\quad  \ \ + \int_\Omega \int_{\mathbb{R}^d \setminus \Omega} \big(u(x)-u(y)\big)\big(v(x)-v(y)\big) \, K(x, \mathrm{d}y) \, \mathrm{d}\lambda(x),
\end{aligned}
\end{equation}
a symmetric bilinear form is given on $\mathcal{V}$. Due to the Hölder inequality, the latter is easily seen to be well-defined and bounded; furthermore, the following basic observations hold:

\begin{lemma}\label{Lemma_Basic}
If $K \in \mathcal{K}$ is $\lambda$-symmetric, $\frac{1}{2} \norm{u}^2_{\mathcal{V}} \leq \int_\Omega u(x)^2 \, \mathrm{d}\lambda(x) + \mathcal{B}(u,u) \leq \norm{u}^2_{\mathcal{V}}$ holds for all $u \in \mathcal{V}$. Moreover, for $\lambda$-a.e. $x \in \Omega$, it is valid that $K(x, \mathbb{R}^d \setminus (\Omega \cup \Gamma)) = 0$. Consequently, if $u,v \in \mathcal{V}$, then
    \[\begin{array}{ll}
    \bullet \quad &\displaystyle{\int_\Omega \int_\Gamma \big(u(x)-u(y)\big)^2 \, K(x,\mathrm{d}y)\, \mathrm{d}\lambda(x)} = \displaystyle{\int_\Omega \int_{\mathbb{R}^d \setminus \Omega} \big(u(x)-u(y)\big)^2 \, K(x,\mathrm{d}y)\, \mathrm{d}\lambda(x);}\\
    & \\
    \bullet \quad&\mathcal{B}(u,v) = \displaystyle{\frac{1}{2}  \int_\Omega \int_\Omega \big(u(x)-u(y)\big)\big(v(x)-v(y)\big) \, K(x, \mathrm{d}y) \, \mathrm{d}\lambda(x)}\\
    &\\
    & \quad  \quad \quad \quad \quad \ + \displaystyle{\int_\Omega \int_{\Gamma}\big(u(x)-u(y)\big)\big(v(x)-v(y)\big) \, K(x, \mathrm{d}y) \, \mathrm{d}\lambda(x).}
    \end{array}\]
\end{lemma}

\begin{proof}
    The first assertion is clear and the second one is a consequence of the definition of $\Gamma$ and the $\lambda$-symmetry of $K \in \mathcal{K}$. Indeed,
    \begin{align*}
  \int_\Omega K(x,\mathbb{R}^d \setminus (\Omega \cup \Gamma)) \, \mathrm{d}\lambda(x) &= \int_\Omega \int_{\mathbb{R}^d\setminus(\Omega \cup \Gamma)} K(x,\mathrm{d}y)\, \mathrm{d}\lambda(x)\\
  &=\int_{\mathbb{R}^d\setminus(\Omega \cup \Gamma)} \int_\Omega K(y,\mathrm{d}x)\, \mathrm{d}\lambda(y) = \int_{\mathbb{R}^d\setminus(\Omega \cup \Gamma)} K(y,\Omega) \, \mathrm{d}\lambda(y) \\
  &= 0
\end{align*}
 and, thus, $K(x, \mathbb{R}^d \setminus (\Omega \cup \Gamma)) = 0$ for $\lambda$-a.e. $x \in \Omega$.
\end{proof}

\textnormal{We aim to establish a basic understanding of $\mathcal{V}_0$ and $\mathcal{V}$. To this end, we prove that both function spaces are complete. Note that our reasoning builds upon arguments given in \cite[Proposition 3.1]{dipierro2017nonlocal}, \cite[Lemma 2.3]{felsingerDirichlet}, and \cite[Theorem 3.46]{foghem2020l2} as well as generalizes the proof idea of \cite[Theorem 2.1]{Frerick2022TheNN}.}

\begin{theorem}\label{Completeness_Dirichlet}
    Let $\emptyset \neq \Omega \subseteq \mathbb{R}^d$ be open, let $\lambda$ be a Borel measure on $\mathbb{R}^d$, and let $K \in \mathcal{K}$ be $\lambda$-symmetric. Then, the subspace $\mathcal{V}_0$ of $\mathcal{V}$ is complete, i.\,e., every Cauchy sequence in $\mathcal{V}_0$ converges to an element in $\mathcal{V}_0$ with respect to $\norm{\cdot}_{\mathcal{V}}$. 
\end{theorem}

\begin{proof}
Let $(v_n)_{n \in \mathbb{N}}$ be a Cauchy sequence in $(\mathcal{V}_0,\norm{\cdot}_{\mathcal{V}})$. We show that an element $v \in \mathcal{V}_0$ exists for which $\norm{v_n-v}_{\mathcal{V}} \rightarrow 0$ as $n \rightarrow \infty$: At first, let us notice that because $(v_n)_{n \in \mathbb{N}}$ is Cauchy, we have both
\[\begin{array}{cll}
    \bullet & \displaystyle{\lim_{k,l \rightarrow \infty} \int_\Omega \big(v_k(x)-v_l(x)\big)^2 \, \mathrm{d}\lambda(x) = 0}; & \\
    \bullet & \displaystyle{\lim_{k,l \rightarrow \infty} \int_\Omega \int_{\mathbb{R}^d} \big(v_k(x)-v_k(y) - (v_l(x) - v_l(y))\big)^2 \,K(x,\mathrm{d}y)\, \mathrm{d}\lambda(x)} & \\
  & \quad \quad \quad \quad \quad =\displaystyle{\lim_{k,l \rightarrow \infty} \int_{\Omega \times \mathbb{R}^d} \big(v_k(x)-v_k(y) - (v_l(x) - v_l(y))\big)^2 \, \mathrm{d}(\lambda \otimes K)(x,y) = 0.} & \quad \quad
\end{array}\]
The spaces $\mathcal{L}^2(\Omega, \lambda)$ and $\mathcal{L}^2(\Omega \times \mathbb{R}^d, \lambda \otimes K)$ are complete which yields the existence of an element $v \in \mathcal{L}^2(\Omega, \lambda)$ and an element $w \in \mathcal{L}^2(\Omega \times \mathbb{R}^d, \lambda \otimes K)$ with
\[
\begin{array}{lll}
     \bullet &\displaystyle{\lim_{k \rightarrow \infty} \int_\Omega \big(v_k(x)-v(x)\big)^2 \,\mathrm{d}\lambda(x) = 0} & \\
    \bullet & \displaystyle{\lim_{k \rightarrow \infty} \int_{\Omega \times \mathbb{R}^d} \big(v_k(x)-v_k(y) - w(x,y)\big)^2\, \mathrm{d}(\lambda \otimes K)(x,y)}& \quad \quad \quad \quad \quad \quad \quad \\
     & \displaystyle{ \quad \quad =\lim_{k \rightarrow \infty} \int_\Omega \int_{\mathbb{R}^d} \big(v_k(x)-v_k(y) - w(x,y)\big)^2 \, K(x,\mathrm{d}y)\, \mathrm{d}\lambda(x) = 0.} &
\end{array}
\]
Now, let $(v_{n_l})_{l \in \mathbb{N}}$ be a subsequence of $(v_n)_{n \in \mathbb{N}}$ satisfying
\[\begin{array}{cll}
    \bullet & \lim_{l \rightarrow \infty}v_{n_l}(x)=v(x) &\text{for }x\in\Omega\setminus N , \, \lambda(N) = 0,\\
    &&\\
    \bullet & \lim_{l \rightarrow \infty}(v_{n_l}(x)-v_{n_l}(y))=w(x,y) &\text{for } (x,y)\in (\Omega \times \mathbb{R}^d)\setminus M, \, (\lambda \otimes K)(M) = 0.
\end{array}\]
Set $v\einschraenkung_{\mathbb{R}^d \setminus \Omega} \equiv 0$. Then, $\lim_{l \rightarrow \infty} v_{n_l}(y) = v(y)$ holds for all $y \in \mathbb{R}^d\setminus \Omega$, and because 
\[0= (\lambda \otimes K)(M), \quad 0 = (\lambda \otimes K)(N \times \mathbb{R}^d), \quad 0 = (\lambda \otimes K)(N \times \mathbb{R}^d)=(\lambda \otimes K)(\mathbb{R}^d \times N),\]
we also have
\begin{equation}\label{eq_gleichheit}
    w(x,y) = \lim_{l \rightarrow \infty} \big(v_{n_l}(x) -v_{n_l}(y)\big) = v(x)-v(y)
\end{equation}
for $(\lambda \otimes K)$-a.e. $(x,y) \in \Omega \times \mathbb{R}^d$. Thus, $v \in \mathcal{V}_0$ and ${\lim_{l \rightarrow \infty} \norm{v_{n_l} - v}_{\mathcal{V}} =0}$. Because $(v_n)_{n \in \mathbb{N}}$ is a Cauchy sequence with respect to $ \norm{\cdot}_{\mathcal{V}}$, also ${\lim_{n \rightarrow \infty} \norm{v_n - v}_{\mathcal{V}} =0}$ follows which proves that the space $\mathcal{V}_0$ is complete with respect to $\norm{\cdot}_{\mathcal{V}}$.
\end{proof}

\textnormal{To also obtain the completeness of $\mathcal{V}$, the same proof idea as for \Cref{Completeness_Dirichlet} can be used. However, it is no longer possible to simply extend $v$ by zero in $\mathbb{R}^d \setminus \Omega$ in order to ensure both
\begin{equation}\label{eq_star}
    v \in \mathcal{V} \quad \quad \text{ and }  \quad \quad v(x)-v(y) = w(x,y) \text{ for $(\lambda \otimes K)$-a.e. } (x,y) \in \Omega \times \mathbb{R}^d.
\end{equation}
That there nevertheless is a suitable extension is stated in \Cref{lemma_Fortsetzung} below:}

\begin{lemma}\label{lemma_Fortsetzung}
    Let $\emptyset \neq \Omega \subseteq \mathbb{R}^d$ be open, let $\lambda$ be a Borel measure on $\mathbb{R}^d$, and let $K \in \mathcal{K}$ be a $\lambda$-symmetric transition kernel. Further, assume that $v \in \mathcal{L}^2(\Omega,\lambda)$ and $w \in \mathcal{L}^2(\Omega \times \mathbb{R}^d, \lambda \otimes K)$. If $(v_n)_{n \in \mathbb{N}}$ is a Cauchy sequence in $(\mathcal{V}, \norm{\cdot}_{\mathcal{V}})$ with both
    \begin{equation}\label{eq_req}
        \begin{array}{cll}
            \bullet & \lim_{n \rightarrow \infty}v_{n}(x)=v(x) &\text{for $\lambda$-a.e. }x\in\Omega,\\
            &&\\
            \bullet & \lim_{n \rightarrow \infty}(v_{n}(x)-v_{n}(y))=w(x,y) &\text{for $(\lambda \otimes K)$-a.e. } (x,y)\in \Omega \times \mathbb{R}^d,
        \end{array}
    \end{equation}
    an extension of $v$ to $\mathbb{R}^d \setminus \Omega$ exists which satisfies (\ref{eq_star}).    
\end{lemma}

\newpage

\begin{proof}
    Assume that the null set corresponding to the first equation of (\ref{eq_req}) is given by $N$ and that the null set corresponding to the second one is given by $M$, i.\,e.,
    \[\begin{array}{cll}
    \bullet & \lim_{n \rightarrow \infty}v_{n}(x)=v(x) &\text{for }x\in\Omega\setminus N, \, \lambda(N) = 0,\\
    &&\\
    \bullet & \lim_{n \rightarrow \infty}(v_{n}(x)-v_{n}(y))=w(x,y) &\text{for } (x,y)\in (\Omega \times \mathbb{R}^d)\setminus M, \, (\lambda \otimes K)(M) = 0.
\end{array}\] 
Let $y \in \mathbb{R}^d$. Then, we set $M_y := \{x \in \Omega: (x,y) \in M\}$ and further define 
\[\tilde{\Omega}^y := \Omega\setminus (N \cup M_y).\]
Because $K\in \mathcal{K}$ is $\lambda$-symmetric, it follows from \Cref{lemma_symmetry_Fubini} that
\begin{align*}
    0 = (\lambda \otimes K)(M)  &= \int_{\mathbb{R}^d} \int_{\mathbb{R}^d} \chi_{M}(x,y)\, K(x,\mathrm{d}y)\, \mathrm{d}\lambda(x) \\
    &= \int_{\mathbb{R}^d} \int_{\mathbb{R}^d} \chi_{M}(x,y)\, K(y,\mathrm{d}x)\, \mathrm{d}\lambda(y) = \int_{\mathbb{R}^d} K(y,M_y)\, \mathrm{d}\lambda(y)
\end{align*}
and, further, since  $\lambda(N) = 0$, also $(\lambda \otimes K)(N \times \mathbb{R}^d) = 0$  holds which implies
\[0 = (\lambda \otimes K)(N \times \mathbb{R}^d) = (\lambda \otimes K)(\mathbb{R}^d \times N) = \int_{\mathbb{R}^d} K(y,N)\, \mathrm{d}\lambda(y).\]
Consequently, we have both $K(y,N)=0$ and $K(y,M_y)=0$ for $\lambda$-a.e. $y \in \mathbb{R}^d$ \textendash{} off a set $O \in \mathcal{B}(\mathbb{R}^d)$ with $\lambda(O)=0$. Thus, for $y \in \Gamma\setminus O$, it is valid that
\begin{align*}
    K(y,\tilde{\Omega}^y) = K(y, \Omega) - K\big(y, \Omega \cap (N \cup M_y)\big) &\geq K(y, \Omega) - K(y,N) - K(y,M_y)\\
    &= K(y,\Omega) > 0
\end{align*}
and due to the $\sigma$-finiteness of $K(y, \cdot)$, a measurable set $\Omega^y \subseteq \tilde{\Omega}^y$ exists with ${0 < K(y, \Omega^y) < \infty}$. Now define $v$ in $\mathbb{R}^d \setminus \Omega$ as follows:
\[v(y) := \begin{cases}
    \displaystyle{\frac{1}{K(y,\Omega^y)} \int_{\Omega^y} \big(v(x) - w(x,y)\big)\, K(y,\mathrm{d}x)} & \text{ if } y \in \Gamma \setminus O,\\
    0 & \text{ else.}
\end{cases}\]
We show that for all $y \in \Gamma \setminus O$, $v(y)$ is well-defined and satisfies $v(y) = \lim_{n \rightarrow \infty} v_{n}(y)$: To this end, note that for all $x \in \Omega^y$, we have 
\[v(x) = \lim_{n \rightarrow \infty} v_{n}(x), \quad w(x,y) = \lim_{n \rightarrow \infty}( v_{n}(x) - v_{n}(y)),\]
and consequently, because $w$ is real-valued,
\[|v_{n}(x) - (v_{n}(x)-v_{n}(y))| = |v_{n}(y)| \leq \sup_{n \in \mathbb{N}}|v_{n}(y)|< \infty\]
holds for all $n \in \mathbb{N}$ and $x \in \Omega^y$. Thus, for $y \in \Gamma \setminus O$, the dominated convergence theorem yields both
\begin{align*}
  v(y) &= \frac{1}{K(y,\Omega^y)} \int_{\Omega^y} \big(v(x) - w(x,y) \big) \, K(y,\mathrm{d}x)\\
  &= \frac{1}{K(y,\Omega^y)} \int_{\Omega^y} \lim_{l \rightarrow \infty} v_{n_l}(y)\, K(y,\mathrm{d}x)= \lim_{l \rightarrow \infty} \frac{1}{K(y,\Omega^y)} \int_{\Omega^y} v_{n_l}(y)\, K(y,\mathrm{d}x)\\
  &= \lim_{l \rightarrow \infty} v_{n_l}(y)
\end{align*}
and the well-definedness of $v(y)$. Clearly, $v: \mathbb{R}^d \rightarrow \mathbb{R}$ is now a measurable function and because
\begin{align*}
    &(\lambda \otimes K)(M) &=0,\\
    &(\lambda \otimes K)(N \times \mathbb{R}^d)&=0,\\
    &(\lambda \otimes K) \big(\mathbb{R}^d \times (N \cup O)\big) = (\lambda \otimes K^T)\big((N \cup O) \times \mathbb{R}^d\big) &= 0,
\end{align*}
we can conclude that 
\begin{equation}\label{eq_richtig}
    w(x,y) = \lim_{l \rightarrow \infty} \big(v_{n_l}(x) - v_{n_l}(y)\big) = v(x) - v(y)
\end{equation}
is valid for $(\lambda \otimes K)$-a.e. $(x,y) \in \Omega \times (\Omega \cup \Gamma)$. Finally, given that also \linebreak $(\lambda \otimes K)\big(\Omega \times \mathbb{R}^d\setminus(\Omega \cup \Gamma)\big)=0$ (cf. \Cref{Lemma_Basic}), the identity in (\ref{eq_richtig}) holds in fact for $(\lambda \otimes K)$-a.e. $(x,y) \in \Omega \times \mathbb{R}^d$ such that $v$ is as desired.
\end{proof}

\textnormal{Taking \Cref{lemma_Fortsetzung} into account, the same proof as for \Cref{Completeness_Dirichlet} yields: }

\begin{theorem}\label{Completeness_Robin_Neumann}
Let $\emptyset \neq \Omega \subseteq \mathbb{R}^d$ be open, let $\lambda$ be a Borel measure on $\mathbb{R}^d$, and let $K \in \mathcal{K}$ be a $\lambda$-symmetric transition kernel. Then, the semi-normed vector space $\mathcal{V}$ is complete.
\end{theorem}

\noindent
\textnormal{In order to transform $\mathcal{V}_0$ and $\mathcal{V}$ into complete and \textit{normed} spaces, we now identify the elements of $\mathcal{V}$ with their respective equivalent class:} \\

\noindent
\textnormal{Let $u \in \mathcal{V}$. By definition, $\norm{u}_{\mathcal{V}} = 0$ holds if and only if $u(x)=0$  for $\lambda$-a.e. $x \in \Omega$ and
\[\int_\Omega \int_{\mathbb{R}^d \setminus \Omega} u(y)^2 K(x,\mathrm{d}y)\, \mathrm{d}\lambda(x) = 0.\]
Here, the latter condition is equivalent to $u(y)=0$ for $\lambda$-a.e. $y \in \Gamma$ because, due to the $\lambda$-symmetry of $K \in \mathcal{K}$, \Cref{lemma_symmetry_Fubini} implies
\[\int_\Omega \int_\Gamma u(y)^2\, K(x,\mathrm{d}y)\, \mathrm{d}\lambda(x) = \int_\Gamma \int_\Omega u(y)^2\, K(y,\mathrm{d}x)\, \mathrm{d}\lambda(y) = \int_\Gamma u(y)^2 K(y,\Omega)\, \mathrm{d}\lambda(y)\]
and, by definition, $K(y,\Omega) >0$ is valid in $\Gamma$ and $K(y,\Omega)=0$ in $\mathbb{R}^d \setminus (\Omega \cup \Gamma)$. Hence, for
\[N:= \{v \in \mathcal{V}: v(x) = 0 \ \text{for $\lambda$-a.e. } x \in \Omega \cup \Gamma\},\]
it follows that $N = \ker{\norm{\cdot}_{\mathcal{V}}}$ such that by defining
\[[u]:= u + N =\{v \in \mathcal{V}: v(x)=u(x) \text{ for $\lambda$-a.e. } x \in \Omega \cup \Gamma\}, \ u \in \mathcal{V}\]
the quotient spaces of interest are given by 
\begin{align*}
    V=V(\Omega,K,\lambda) &:= \mathcal{V}/\ker{\norm{\cdot}}_{\mathcal{V}} = \mathcal{V}/N = \{[u]: u \in \mathcal{V}\}, \\
    V_0=V_0(\Omega,K,\lambda) &:= \mathcal{V}_0/\ker{\norm{\cdot}}_{\mathcal{V}} = \mathcal{V}_0/N = \{[u]: u \in \mathcal{V}_0\}.
\end{align*}
Since for $u,v \in \mathcal{V}$, we obtain both $\mathcal{B}(u_1,v_1) = \mathcal{B}(u_2,v_2)$ and $\langle u_1,v_1 \rangle_{\mathcal{V}} = \langle u_2,v_2 \rangle_{\mathcal{V}}$ for all $u_1,u_2 \in [u]$ and $v_1,v_2 \in [v]$, the mappings
\begin{align*}
 \mathcal{B}&: V \times V \rightarrow \mathbb{R}: \big([u],[v]\big) \mapsto \mathcal{B}(u,v),\\
    \langle \cdot,\cdot \rangle_{V}&: V \times V \rightarrow \mathbb{R}:\big([u],[v]\big) \mapsto \langle u,v \rangle_{\mathcal{V}},
\end{align*}
are well-defined. Moreover, $[v]=N$ if and only if $\langle [v],[v]\rangle_{V} = 0$ holds.} \\

\textnormal{Similarly as the elements of the Lebesgue spaces $L^p$ are treated as functions although being equivalence classes in fact, we want to consider the elements of $V$ to be functions which are defined on $\Omega \cup \Gamma$ a.e. with respect to $\lambda$. By definition $V=  \{u \in \mathcal{M}(\Omega \cup \Gamma): \norm{u}_{V} < \infty\}$ holds and two ``functions'' $f,g \in V$ are regarded identical if being  representatives of the same equivalent class. }

\begin{Bsp}
Let $\lambda = \lambda_d$ be the Lebesgue Borel measure on $\mathbb{R}^d$ and $\mu_s$, $s \in \mathbb{R}^d$, be the discrete measure from \Cref{Bsp_Self_adjointness_Stencil} whose associated grid shall be denoted by $G_s := s + h \mathbb{Z}^d$, $h>0$. Then, we have
\begin{itemize}
    \item $f\neq g$ in $V(\Omega,K,\lambda_d)$ if and only if $f \neq g$ on a set $\Lambda \subseteq \Omega \cup \Gamma$ with $\lambda_d(\Lambda)>0$;
    \item $f\neq g$ in $V(\Omega,K, \mu_s)$ if and only if $f(x) \neq g(x)$ for at least one $x \in G_s \cap (\Omega \cup \Gamma)$, i.\,e., $f$ and $g$ differ in at least one grid node in $\Omega \cup \Gamma$.
\end{itemize}
\end{Bsp}

\textnormal{With the help of \Cref{Completeness_Dirichlet} and \Cref{Completeness_Robin_Neumann}, we arrive at the core statement of this section:}

\begin{theorem}\label{Separability_Hilbert_Space}
    Let $\emptyset \neq \Omega \subseteq \mathbb{R}^d$ be an open set and let $\lambda$ be a Borel measure on $\mathbb{R}^d$. If $K \in \mathcal{K}$ is $\lambda$-symmetric, the space $V$ is a separable Hilbert space and the subspace $V_0$ is closed in $V$.
\end{theorem}

\begin{proof}
    The proof is analogous to \cite[Corollary 2.2]{Frerick2022TheNN} and only presented for the reader's convenience. 

\vspace{0.2cm}
\noindent
Due to \Cref{Completeness_Dirichlet}, \Cref{Completeness_Robin_Neumann}, and the fact that $\norm{\cdot}_{V}$ is induced by the inner product $\langle \cdot, \cdot \rangle_{V}$, only the separability is not proven yet. To obtain the latter, note that
\[\mathcal{I}: V \rightarrow L^2(\Omega,\lambda) \times L^2(\Omega \times \mathbb{R}^d, \lambda \otimes K), \ \ u \mapsto (u,u(\cdot)-u(\cdot \cdot))\]
is an isometric mapping by the definition of the norm $\norm{\cdot}_{V}$ in $V$. Because $V$ is complete with respect to $\norm{\cdot}_{V}$, it follows that $\mathcal{I}(V)$ is a (closed) subspace of $L^2(\Omega, \lambda) \times L^2(\Omega \times \mathbb{R}^d, \lambda \otimes K)$. Given that the latter product is separable, also the separability of $V$ is obtained.
\end{proof}

\subsection{A Variational Formulation}\label{subsection_integration_by_parts}

As in the local case, we are interested in providing a variational structure of the problems (\ref{DP}) and (\ref{NP}). For sufficiently regular $u,v \in V$, we thus prove that the bilinear form $\mathcal{B}: V \times V \rightarrow \mathbb{R}$ defined in (\ref{definition_B}) can be associated with the nonlocal operators $\mathcal{L}= \mathcal{L}_K$ and $\mathcal{N}= \mathcal{N}_{\Omega,K}$ via the identity
\begin{equation}\label{Aussage_nonlocal_integration_by_parts0}
        \int_\Omega \mathcal{L}u(x) v(x) \, \mathrm{d}\lambda(x) = \mathcal{B}(u,v) -  \int_\Gamma \mathcal{N}u(y) v(y) \, \mathrm{d}\lambda(y).
\end{equation}
As remarked below, the above equation resembles the first Gauss-Green identity from classical PDE theory and can thus justly be considered a \textit{nonlocal integration by parts formula}.

\begin{Rem}
    Let $\Omega \subseteq \mathbb{R}^d$ be an open and bounded set with $\partial \Omega \in C^1$ and let $u \in H^2(\Omega)$ and $v \in H^1(\Omega)$. Then, the first Gauss-Green identity (e.g. \cite[Appendix A.3]{triebel}) states that
    \begin{equation}\label{local_Gauss_Greeen_first}
    \int_\Omega (- \Delta)u(x)v(x) \, \mathrm{d}x  = \int_\Omega \nabla u(x) \nabla v(x) \, \mathrm{d}x - \int_{\partial \Omega} \partial_\eta u(y) v(y) \, \mathrm{d}\sigma(y) 
    \end{equation}
    where $\partial_\eta$ denotes normal derivative along the outward unit normal vector $\eta$ at $\partial \Omega$ and both $\partial_\eta u(y)$ and $v(y)$ are to be understood in the sense of traces. A direct comparison of (\ref{Aussage_nonlocal_integration_by_parts0}) with (\ref{local_Gauss_Greeen_first}) shows that while $\Gamma$ substitutes the topological boundary $\partial \Omega$, the normal derivative $\partial_\eta u$ is replaced by the Neumann operator $\mathcal{N} u$. Finally, $\mathcal{B}(u,v)$ serves as the nonlocal analog of the energy term $\int_\Omega \nabla u \nabla v \, \mathrm{d}x$. 
\end{Rem}

\begin{theorem}{\textnormal{\textbf{(Nonlocal Integration By Parts Formula)}}} \label{Theorem_Nonlocal_Integration_By_Parts_Formula} \index{Theorem ! nonlocal version of the first Gauss-Green identity}
    Let $\lambda$ be a Borel measure on $\mathbb{R}^d$, let $\emptyset \neq \Omega \subset \mathbb{R}^d$ be open and $\lambda$-bounded, i.\,e., $\lambda(\Omega) < \infty$, and let $K \in \mathcal{K}$ be $\lambda$-symmetric. If $u \in V$ satisfies the regularity condition
    \begin{equation}\label{(5)}
        \int_\Omega \Bigg(\int_{\mathbb{R}^d} \big|u(x)-u(y)\big| \, K(x, \mathrm{d}y) \Bigg)^2 \, \mathrm{d}\lambda(x) < \infty,
    \end{equation}
   then, for all $v \in V$, it is valid that
    \begin{equation}\label{Aussage_nonlocal_integration_by_parts}
        \int_\Omega \mathcal{L}u(x) v(x) \, \mathrm{d}\lambda(x) = \mathcal{B}(u,v) -  \int_\Gamma \mathcal{N}u(y) v(y) \, \mathrm{d}\lambda(y).
    \end{equation}
    In particular, setting $v:= \chi_{\mathbb{R}^d}$, we have $\int_\Omega - \mathcal{L} u(x) \, \mathrm{d}\lambda(x) = \int_\Gamma \mathcal{N}u(y) \, \mathrm{d}\lambda(y)$.
\end{theorem}

\textbf{Remark:} Let $\lambda= \lambda_d$ be the Lebesgue Borel measure and $\mathcal{K} \ni K=K_\gamma$ be determined by a symmetric and Borel measurable function $\gamma: \mathbb{R}^d \times \mathbb{R}^d \rightarrow \mathbb{R}$. Then, the nonlocal integration by parts formula (\ref{Aussage_nonlocal_integration_by_parts}) above coincides with the one in \cite[Theorem 2.4]{Frerick2022TheNN}. Let us highlight that while the latter can be proven by utilizing the Fubini theorem, our more generalized formulation relies heavily on our notion of symmetry and the generalized Fubini identity (\ref{generalisierung_fubini}) attached. 

\begin{proof}
     We split the proof into two steps. First, we show that all occurring expressions are in fact well-defined (Step 1) before the identity in (\ref{Aussage_nonlocal_integration_by_parts}) is verified (Step 2). \hfill 
\vspace{0.2cm} \linebreak
\textbf{Step 1:} At first, it should be noted that, by (\ref{(5)}) and due to $K \in \mathcal{K}$, it is valid that $\mathcal{L} u \in L^2(\Omega,\lambda)$. Consequently, Hölder's inequality yields
\[\int_\Omega \mathcal{L} u(x)v(x) \, \mathrm{d}\lambda(x) \leq \norm{\mathcal{L} u}_{L^2(\Omega, \lambda)} \norm{v}_{L^2(\Omega, \lambda)} < \infty\]
such that the above integral is well-defined and finite. To observe the analog result also for the second Lebesgue integral $\int_\Gamma \mathcal{N}u(y)v(y) \, \mathrm{d}\lambda(y)$, let us first register that the $\lambda$-symmetry of $K \in \mathcal{K}$ can be applied via \Cref{lemma_symmetry_Fubini} to obtain
\begin{align*}
    \int_\Gamma \int_\Omega |u(y)-u(x)| \, K(y, \mathrm{d}x) \, \mathrm{d}\lambda(y) & = \int_\Omega \int_\Gamma |u(y)-u(x)| \, K(x, \mathrm{d}y) \, \mathrm{d}\lambda(x)\\
    & \leq \int_\Omega \int_{\mathbb{R}^d} |u(x)-u(y)| \, K(x, \mathrm{d}y) \, \mathrm{d}\lambda(x) < \infty;
\end{align*}
especially, since $K \in \mathcal{K}$, the integral $\int_\Gamma \mathcal{N}u(y) v(y) \, \mathrm{d}\lambda(y)$ is well-defined. To see its boundedness, write
\begin{align}\label{IBP_Hilsformel1}
    \int_\Gamma |\mathcal{N}u(y) v(y)| \, \mathrm{d}\lambda(y) &\leq \int_\Gamma \int_\Omega \big|\big(u(y)-u(x)\big) v(y)\big| \, K(y, \mathrm{d}x) \, \mathrm{d}\lambda(y) \notag\\
    &= \int_\Gamma \int_\Omega \big|\big(u(y)-u(x)\big)\big(v(x)-v(x)-v(y)\big)\big| \, K(y, \mathrm{d}x) \, \mathrm{d}\lambda(y) \notag\\
    &\leq  \int_\Gamma \int_\Omega \big|\big(u(y)-u(x)\big)\big(v(x)-v(y)\big)\big| \, K(y, \mathrm{d}x) \, \mathrm{d}\lambda(y) \\
    & \quad +  \int_\Gamma \int_\Omega \big|\big(u(y)-u(x)\big)| \, |v(x)\big| \, K(y, \mathrm{d}x) \, \mathrm{d}\lambda(y).\notag
\end{align}
Then, because $u,v \in V(\Omega,K,\lambda)$, it is a consequence of \Cref{lemma_symmetry_Fubini} and Hölder's inequality that both
\begin{align*}
    &\int_\Gamma\int_\Omega  \big|\big(u(y)-u(x)\big)\big(v(x)-v(y)\big)\big| \, K(y, \mathrm{d}x) \, \mathrm{d}\lambda(y) \\
    &=\int_\Omega \int_\Gamma \big|\big(u(y)-u(x)\big)\big(v(x)-v(y)\big)\big| \, K(x, \mathrm{d}y) \, \mathrm{d}\lambda(x)\\
   &\leq \sqrt{\int_\Omega \int_{\mathbb{R}^d} \big(u(x)-u(y)\big)^2 \, K(x, \mathrm{d}y) \, \mathrm{d}\lambda(x)}\sqrt{\int_\Omega \int_{\mathbb{R}^d} \big(v(x)-v(y)\big)^2 \, K(x, \mathrm{d}y) \, \mathrm{d}\lambda(x)\Bigg)}\\
   &\leq \norm{u}_{V} \norm{v}_{V} <\infty
\end{align*}
and
\begin{align*}
    \int_\Gamma \int_\Omega  \big|\big(u(y)-u(x)\big)| \, & |v(x)\big| \, K(y, \mathrm{d}x) \, \mathrm{d}\lambda(y) \\
   &= \int_\Omega \int_\Gamma \big|\big(u(y)-u(x)\big)| \, |v(x)\big| \, K(x, \mathrm{d}y) \, \mathrm{d}\lambda(x) \\
   &\leq \int_\Omega \int_{\mathbb{R}^d} \big|\big(u(x)-u(y)\big)\big| \, K(x, \mathrm{d}y)\, \big|v(x)\big| \, \mathrm{d}\lambda(x)\\
     &\leq  \sqrt{\int_\Omega \Bigg(\int_{\mathbb{R}^d} \big|u(x)-u(y)\big| \, K(x, \mathrm{d}y) \Bigg)^2 \, \mathrm{d}\lambda(x)} \norm{v}_{L^2(\Omega,\lambda)} < \infty.
\end{align*}
Because finally by \Cref{Lemma_Basic}, $\mathcal{B}: V\times V \rightarrow \mathbb{R}$ defines a bounded bilinear form on $V$ all ingredients of (\ref{Aussage_nonlocal_integration_by_parts}) are thus well-defined and finite.\hfill 
\vspace{0.2cm} \linebreak
\textbf{Step 2:} To now observe the nonlocal integration by parts formula (\ref{Aussage_nonlocal_integration_by_parts}) asserted, we use the linearity of the integral to obtain
\begin{equation}\label{IBP_3}
\begin{aligned}
   \int_\Omega \mathcal{L} u(x) v(x) \, \mathrm{d}\lambda(x) &= \int_\Omega \int_{\mathbb{R}^d}\big(u(x)-u(y)\big)v(x) \, K(x, \mathrm{d}y) \, \mathrm{d}\lambda(x) \notag\\
   &= \int_\Omega \int_{\Omega}\big(u(x)-u(y)\big)v(x) \, K(x, \mathrm{d}y) \, \mathrm{d}\lambda(x)\\
   & \quad + \int_\Omega \int_{\mathbb{R}^d\setminus \Omega}\big(u(x)-u(y)\big)v(x) \, K(x, \mathrm{d}y) \, \mathrm{d}\lambda(x). \notag
\end{aligned}
\end{equation}
Due to $K\big(x, \mathbb{R}^d \setminus (\Omega \cup \Gamma)\big)=0$ for $\lambda$-a.e. $x \in \Omega$ (cf. \Cref{Lemma_Basic}), the second integral herein reduces to
\[ \int_\Omega \int_{\mathbb{R}^d\setminus \Omega}\big(u(x)-u(y)\big)v(x) \, K(x, \mathrm{d}y) \, \mathrm{d}\lambda(x) = \int_\Omega \int_{\Gamma}\big(u(x)-u(y)\big)v(x) \, K(x, \mathrm{d}y) \, \mathrm{d}\lambda(x).\]
The rest is a simple application of \Cref{lemma_symmetry_Fubini}: One the one hand, we have 
\begin{equation}\label{IBP_1}
    \begin{aligned}
        \int_\Omega \int_\Omega \big(u(x)-u(y)\big) & v(x) \, K(x, \mathrm{d}y) \, \mathrm{d}\lambda(x) \\
        &= \frac{1}{2} \int_\Omega \int_\Omega \big(u(x)-u(y)\big) v(x) \, K(x, \mathrm{d}y) \, \mathrm{d}\lambda(x) \\
        & \quad - \frac{1}{2}\int_\Omega \int_\Omega \big(u(y)-u(x)\big) v(x) \, K(y, \mathrm{d}x) \, \mathrm{d}\lambda(y)\\
        &= \frac{1}{2} \int_\Omega \int_\Omega \big(u(x)-u(y)\big) \big(v(x)-v(y)\big) \, K(x, \mathrm{d}y) \, \mathrm{d}\lambda(x) 
    \end{aligned}
\end{equation} 
and on the other hand, also
\begin{align}\label{IBP_2}
        \int_\Omega \int_{\mathbb{R}^d \setminus \Omega}&\big(u(x)-u(y)\big)v(x) \, K(x, \mathrm{d}y) \, \mathrm{d}\lambda(x)\\
        &= \int_\Omega \int_\Gamma \big(u(x)-u(y)\big)\big(v(x)-v(y)\big)\, K(x, \mathrm{d}y) \, \mathrm{d}\lambda(x) - \int_\Gamma \mathcal{N}u(y) v(y)\, \mathrm{d}\lambda(y) \notag
\end{align}
is valid. Conflating the observations from (\ref{IBP_3}), (\ref{IBP_1}), and (\ref{IBP_2}), the validity of (\ref{Aussage_nonlocal_integration_by_parts}) is obtained. The addendum is a direct consequence of $\lambda(\Omega) < \infty$ (then: $\chi_{\mathbb{R}^d} \in V$).
\end{proof}

As in the local theory of PDEs, the nonlocal integration by parts formula (\ref{Aussage_nonlocal_integration_by_parts}) can be used to motivate a reasonable notion of weak solutions for the nonlocal problems 
\\
\begin{minipage}{0.2\textwidth}
    \[ \hspace{5.3cm} \mathcal{L}u=f \text{ in } \Omega,\]
\end{minipage}
\begin{minipage}{0.8\textwidth}
    \begin{align*}
    u&=g \text{ in } \Gamma, \tag{DP} \label{DP}\\
    \mathcal{N}u &=g \text{ in } \Gamma, \tag{NP} \label{NP} 
    \end{align*}
\end{minipage} \\
\\
of interest where, henceforth, $\emptyset \neq \Omega \subseteq \mathbb{R}^d$ is always a bounded and open set:
\begin{itemize}
\item Let us assume that $u\in V$ is a solution of the nonlocal Dirichlet problem (\ref{DP}) which is sufficiently regular in the sense that (\ref{(5)}) is satisfied. Then, we obtain, for all $v \in V_0$:
\begin{align*}
    \int_\Omega \textcolor{blue}{f(x)} v(x) \, \mathrm{d}\lambda(x) &= \int_\Omega \textcolor{blue}{\mathcal{L}u(x)}v(x) \, \mathrm{d}\lambda(x) = \mathcal{B}(u,v).
\end{align*}
\item Analogously, if $u\in V$ is a solution of the nonlocal Neumann problem (\ref{NP}) and sufficiently regular in the sense of (\ref{(5)}), we have for all $v \in V$:
\begin{align*}
    \int_\Omega \textcolor{blue}{f(x)} v(x) \, \mathrm{d}\lambda(x) = \int_\Omega \textcolor{blue}{\mathcal{L}u(x)}v(x) \, \mathrm{d}\lambda(x) &= \mathcal{B}(u,v) -  \int_\Gamma \textcolor{red}{\mathcal{N}u(y)}v(y) \, \mathrm{d}\lambda(y)\\
    &= \mathcal{B}(u,v) - \int_\Gamma \textcolor{red}{g(y)}v(y) \, \mathrm{d}\lambda(y).
\end{align*}
\end{itemize}
Our notions of a weak solution are defined accordingly:

\begin{definition}\label{definition_weak_solutions}
     A function $u\in V$ is called a \textbf{weak} (or \textbf{variational}) \textbf{solution} of
     \begin{itemize}
         \item the nonlocal Dirichlet problem (\ref{DP}) if both
          \begin{equation}\label{weak_formulation_Dirichlet}
        u-g \in V_0 \quad \text{and} \quad \int_\Omega f(x)v(x) \, \mathrm{d}\lambda(x) = \mathcal{B}(u,v) \text{ holds for all } v \in V_0.
    \end{equation}
    We call the Dirichlet problem \textbf{homogeneous} if $g=0$. 
    \item the nonlocal Neumann problem (\ref{NP}) if
    \begin{equation}\label{weak_formulation_Neumann}
        \int_\Omega f(x)v(x) \, \mathrm{d}\lambda(x) + \int_\Gamma g(y) v(y) \, \mathrm{d}\lambda(y) = \mathcal{B}(u,v) \text{ holds for all } v \in V.
    \end{equation}
         In case that $g=0$, the Neumann problem is said to be \textbf{homogeneous}.
     \end{itemize}
\end{definition}

Let us depict the minimization problems corresponding to (\ref{weak_formulation_Dirichlet}) and (\ref{weak_formulation_Neumann}), thus justifying the term ``variational solution". Analogously as proven by Foghem in \cite[Proposition 4.15 and 4.21]{foghem2020l2}, one can show:

\begin{Prop}
    Set $V_g:= \{u \in V: u-g \in V_0\}$ and let $\mathcal{E}: V_g \rightarrow \mathbb{R}$ and $\mathcal{F}: V \rightarrow \mathbb{R}$ be given by
    \begin{align*}
        \mathcal{E}(v) &:= \mathcal{B}(v,v) - \int_\Omega f(x)v(x) \, \mathrm{d}\lambda(x);\\
        \mathcal{F}(v) &:= \mathcal{B}(v,v) - \left(\int_\Omega f(x)v(x) \, \mathrm{d}\lambda(x) + \int_\Gamma g(y)v(y) \, \mathrm{d}\lambda(y) \right).
    \end{align*}
    Then, a function $u \in V$ is a weak solution of the nonlocal Dirichlet problem (\ref{DP}) if and only if $\mathcal{E}(u) = \min_{v \in V_g} \mathcal{E}(v)$ and a weak solution of the nonlocal Neumann problem (\ref{NP}) if and only if $\mathcal{F}(u) = \min_{v \in V} \mathcal{F}(v)$.
\end{Prop}

Recall from the above that every solution of problem (\ref{DP}) respectively problem (\ref{NP}) which satisfies
\begin{equation}\label{condition_regularity}
        \int_\Omega \Bigg(\int_{\mathbb{R}^d} \big|u(x)-u(y)\big| \, K(x, \mathrm{d}y) \Bigg)^2 \, \mathrm{d}\lambda(x) < \infty,
\end{equation}
is a weak solution of the corresponding problem as well. \Cref{Strong_Weak_Solution} below provides sufficient conditions for the reverse implication to hold at least almost everywhere with respect to $\lambda$:

\begin{Rem}\label{Strong_Weak_Solution}
    Let $\emptyset \neq \Omega \subseteq \mathbb{R}^d$ be open and bounded, let $\lambda$ be a locally finite Borel measure on $\mathbb{R}^d$, and let $u \in V$ be a weak solution of problem (\ref{DP})/(\ref{NP})) which satisfies the condition (\ref{condition_regularity}) from above. Further, assume that $f \in L^2(\Omega,\lambda)$ and that $g \in L^1(\Gamma, \lambda)$.
    \begin{enumerate}
        \item If $V$ contains the space $C_c^\infty(\Omega \cup \Gamma)$ of smooth functions with compact support in $\Omega \cup \Gamma$, the nonlocal integration by parts formula yields, for all $v \in C_c^\infty(\Omega \cup \Gamma)$, that
        \begin{align*}
            \int_\Omega \mathcal{L}u(x) v(x) \, \mathrm{d}\lambda(x) &+ \int_\Gamma \mathcal{N}u(y)v(y) \, \mathrm{d}\lambda(y)  \\
            &= \mathcal{B}(u,v) = \int_\Omega f(x)v(x) \, \mathrm{d}\lambda(x) + \int_\Gamma g(y)v(y) \, \mathrm{d}\lambda(y).
        \end{align*}
        In particular,
        \begin{itemize}
             \item $\displaystyle{\int_\Omega \big( \mathcal{L}u(x) -f(x)\big)v(x) \, \mathrm{d}\lambda(x) = 0}$ holds for all $v \in C_c^\infty(\Omega)$ and
            \item $\displaystyle{\int_\Gamma \big(\mathcal{N}u(y) - g(y)\big)v(y) \, \mathrm{d}\lambda(y) =0}$ is valid for all $v \in C_c^\infty(\Gamma)$.
        \end{itemize}
        Thus, due to the regularity of the measure $\lambda$ and by utilizing standard techniques, it is easily seen that both
        \[\begin{array}{rcll}
         \mathcal{L}u & = & f & \lambda \text{-a.e. on } \Omega, \\
          (\mathcal{N})u & = & g &  \lambda \text{-a.e. on } \Gamma.
        \end{array}\]
        Accordingly, $u$ is almost everywhere (with respect to $\lambda$) a strong solution of problem (\ref{DP})/(\ref{NP}).
        \item If $\lambda$ is a regular Borel measure on $\mathbb{R}^d$ which is finite on compact sets and if for $\lambda$-a.e. $x \in \Omega \cup \Gamma^\alpha$, a number $R_x>0$ exists such that
        $\chi_{B_r(x)} \in V$ for all $0 < r < R_x$, this is likewise valid. Indeed, In this case, we can deduce from \cite[Theorem 8.4.6]{benedetto2009integration} that
        \begin{align*}
            \mathcal{L}u(x_0) &= \lim_{r \rightarrow 0} \frac{1}{\lambda(B_r(x_0))} \int_\Omega \mathcal{L}u(x) \chi_{B_r(x_0)}(x) \, \mathrm{d}\lambda(x) \\
            &=  \lim_{r \rightarrow 0} \frac{1}{\lambda(B_r(x_0))} \int_\Omega f(x) \chi_{B_r(x_0)}(x) \, \mathrm{d}\lambda(x) = f(x_0)\\
        \intertext{for $\lambda$-a.e. $x_0 \in \Omega$ and that}
            \mathcal{N}u(x_0) &=  \lim_{r \rightarrow 0} \frac{1}{\lambda(B_r(x_0))} \int_{\Gamma} \mathcal{N}u(x)\ \chi_{B_r(x_0)}(x) \, \mathrm{d}\lambda(x) \\
            &=  \lim_{r \rightarrow 0} \frac{1}{\lambda(B_r(x_0))} \int_{\Gamma} g(x)\chi_{B_r(x_0)}(x) \, \mathrm{d}\lambda(x) = g(x_0) 
        \end{align*}
        for $\lambda$-a.e. $x \in \Gamma$.
        \end{enumerate}    
\end{Rem}

\section{Nonlocal Poincaré-Type Inequalities}\label{section_Inequalities}

Poincaré-type inequalities play an important role for proving well-posedness of boundary vale problems and typically appear as an assumption, implicitely or explicitely, in the corresponding theorems. Throughout the section, both a nonlocal Poincaré inequality and a nonlocal Poincaré-Friedrichs inequality are established that are specifically geared to the general setting considered in this article. In what follows, we examine how both inequalities relate to their local counterparts and previous nonlocal versions as utilized by \cite{felsingerDirichlet}, \cite{foghem2022general}, or \cite{Frerick2022TheNN}. 

\begin{definition}{\textnormal{\textbf{(Nonlocal Inequalities)}}} \label{Defi_Nonlocal_Inequalities} 
Let $\emptyset \neq \Omega \subseteq \mathbb{R}^d$ be open, let $\lambda$ be a Borel measure on $\mathbb{R}^d$, and let $K\in \mathcal{K}$ be symmetric with respect to $\lambda$. Further, set \linebreak $\ker{\mathcal{B}}:= \{w \in V: \mathcal{B}(w,w)=0$\}. We say that
    \begin{enumerate}
        \item the \textbf{nonlocal Friedrichs inequality} holds in $V_0$ if a constant $C>0$ exists such that for all $v \in V_0$, we have
        \[ \norm{v}_{L^2(\Omega,\lambda)}^2 \leq C \, \mathcal{B}(v,v).\]
        Every constant $C>0$ for which the nonlocal Friedrichs inequality is satisfied is called a \textbf{Friedrichs constant}.
        \item the \textbf{nonlocal Poincaré inequality} holds in $V$ if a constant $C>0$ exists such that for all $v \in V$, we have
        \[\inf_{w \in \ker{\mathcal{B}}}\int_\Omega \big(v(x)-w(x)\big)^2 \, \mathrm{d}\lambda(x) \leq C\, \mathcal{B}(v,v).\]
        Every constant $C>0$ for which the nonlocal Poincaré inequality is satisfied is called a \textbf{Poincaré constant}.
    \end{enumerate} 
\end{definition}

\textbf{Remark:} In \cite[Chapter 6]{Huschens}, also a generalized Friedrichs inequality is introduced which is needed for tackling the nonlocal Robin problem. \\

In the setting above, we call $\ker{\mathcal{B}}$ the \textit{kernel} or \textit{nullspace} of $\mathcal{B}$ in $V$. Note that by the Cauchy-Schwarz inequality and the fact that 
\begin{equation}\label{Ungleichungskette_B}
    \mathcal{B}(z,z) \leq \int_\Omega \int_{\mathbb{R}^d} \big(z(x)-z(y)\big)^2 \, K(x, \mathrm{d}y) \, \mathrm{d}\lambda(x) \leq 2 \mathcal{B}(z,z) \quad \text{for all } z \in V,
\end{equation}
we obtain
\begin{align}\label{ker_B_über_energyform}
    \ker{\mathcal{B}} &= \bigcap_{u \in V} \{w \in V: \mathcal{B}(w,u) = 0\} \\
    &= \bigcap_{u \in V} \left\{w \in V: \int_\Omega \int_{\mathbb{R}^d} \big(w(x)-w(y)\big)\big(u(x)-u(y)\big) \, K(x, \mathrm{d}y) \, \mathrm{d}\lambda(x) = 0\right\}. \notag
\end{align}
Further, for all $u \in V$ and $w \in \ker{\mathcal{B}}$, we have $\mathcal{B}(u-w,u-w) = \mathcal{B}(u,u)$. \\

Because $V \ni w \mapsto \mathcal{B}(w,u)$ is a linear and continuous mapping for all $u \in V$ fixed, $\ker{\mathcal{B}}$ is a closed subspace of $V$. Let us utilize this fact next to show that the infimum appearing in the nonlocal Poincaré inequality is actually attained. As outlined below, this is a straightforward consequence of the Hilbert projection theorem \cite[Lemma 4.1]{stein2009real}:  

\begin{lemma}\label{Lemma_ker_B_orthogonal_L2_norm}
    Let $\lambda$ be a Borel measure on $\mathbb{R}^d$, let $\emptyset \neq \Omega \subseteq \mathbb{R}^d$ be open, and assume that $K \in \mathcal{K}$ is $\lambda$-symmetric. Then, for each $u \in V$, a unique element $P(u)\in \ker{\mathcal{B}}$ exists such that we have both $\langle u-P(u),w \rangle_{V} = \langle u-P(u),w \rangle_{L^2(\Omega,\lambda)} =0$ for all $w \in \ker{\mathcal{B}}$ and
    \begin{equation}\label{Ungl_Hilbert_projection}
        \int_\Omega \big(u(x)-P(u)(x)\big)^2 \, \mathrm{d}\lambda(x) = \inf_{w \in \ker{\mathcal{B}}}\int_\Omega \big(u(x)-w(x)\big)^2 \, \mathrm{d}\lambda(x).
    \end{equation} 
\end{lemma}

\begin{proof}
    By \Cref{Completeness_Robin_Neumann}, $\ker{\mathcal{B}}$ is a closed subspace of the Hilbert space $V$ such that the existence of the element $P(u) \in \ker{\mathcal{B}}$ satisfying \[\langle u-P(u),w \rangle_{V} = \langle u-P(u),w \rangle_{L^2(\Omega,\lambda)} =0 \quad \text{for all } w \in \ker{\mathcal{B}}\] is a straightforward consequence of \cite[Lemma 4.1]{stein2009real} and (\ref{ker_B_über_energyform}). \cite[Lemma 4.1]{stein2009real} also yields that $\inf_{w \in \ker{\mathcal{B}}} \norm{u-w}^2_{V}= \norm{u-P(u)}^2_{V}$ and because $w,P(u)\in \ker{\mathcal{B}}$, the validity of (\ref{Ungl_Hilbert_projection}) follows. 
\end{proof}

For the further course of this article, also the orthogonal complement of $\ker{\mathcal{B}}$ in $V$ is of interest:

\begin{dar}
Again, we assume that $\emptyset \neq \Omega \subseteq \mathbb{R}^d$ is an open set and that $K \in \mathcal{K}$ is symmetric with respect to a fixed Borel measure $\lambda$ on $\mathbb{R}^d$. The set
\begin{align*}
    (\ker{\mathcal{B}})^\perp :&= \big\{v \in V: \langle v,w \rangle_{V} = 0 \text{ for all } w \in \ker{\mathcal{B}}\big\}
\end{align*}
is called the \textbf{orthogonal complement of $\ker{\mathcal{B}}$ in $V$}. Utilizing (\ref{ker_B_über_energyform}), we have
\begin{align*}
    (\ker{\mathcal{B}})^\perp = \bigcap_{w \in \ker{\mathcal{B}}}\left\{v \in V: \int_\Omega v(x)w(x) \, \mathrm{d}\lambda(x) =0\right\}
\end{align*}
and because $\ker{\mathcal{B}}$ is a closed subspace of the Hilbert space $V$, the latter can be decomposed into $V = \ker{\mathcal{B}} \oplus (\ker{\mathcal{B}})^\perp$. Since moreover, the quotient map
\[q\einschraenkung_{(\ker{\mathcal{B}})^\perp}: (\ker{\mathcal{B}})^\perp \rightarrow V/(\ker{\mathcal{B}}), \quad  u \mapsto [u]:=\{u + \ker{\mathcal{B}}\}\]
is an isometric isomorphism, we also see that we can identify $\ker{\mathcal{B}}^\perp$ with the quotient space $V/\ker{\mathcal{B}}$.
\end{dar}

\subsection{The Nonlocal Friedrichs Inequality}\label{Subsection_friedrichs_inequality}

We want to put the nonlocal Friedrichs inequality from \Cref{Defi_Nonlocal_Inequalities} into context which demands a constant $C>0$ such that
\begin{equation}\label{Friedrichs_Inequality}
    \int_\Omega u(x)^2 \, \mathrm{d}\lambda(x) \leq C\, \mathcal{B}(u,u)
\end{equation}
holds for all $v \in V_0$. Because here, the $L^2$-norm of the function $u$ is estimated by its nonlocal energy $\mathcal{B}(u,u)$, the latter can justly be considered the nonlocal counterpart of the classical Poincaré-Friedrichs inequality which in case of a bounded set
$\Omega \subseteq \mathbb{R}^d$ and $1 \leq p < \infty$ states the existence of a constant $C>0$ such that
\[\norm{v}_{L^p(\Omega)} \leq C \norm{\nabla v}_{L^p(\Omega)}\]
is satisfied for all $v \in W_0^{1,p}(\Omega)$ (e.g., \cite[Theorem 13.19]{leoni2017first}). \\

Poincaré-Friedrich-type inequalities are typically exploited for proving well-\linebreak posedness of Dirichlet problems. Also in the literature centering around the nonlocal diffusion operator $\mathcal{L}_\gamma \in \mathscr{D}$ this is the case. Let us outline that our formulation (\ref{Friedrichs_Inequality}) is a straightforward generalization of the ones utilized by \cite{felsingerDirichlet}, \cite{foghem2022general}, and \cite{Frerick2022TheNN}:

\begin{Rem}
    Let $\gamma: \mathbb{R}^d \times \mathbb{R}^d \rightarrow [0,\infty)$ be both Borel measurable and symmetric. Then, the nonlocal Friedrichs inequality as utilized by \cite[Section 2.3, for $\Omega = \mathbb{R}^d$]{felsingerDirichlet}, \cite[Theorem 3.13 + Lemma 2.2]{foghem2022general}, and \cite[Section 4]{Frerick2022TheNN} requires the existence of a constant $C>0$ such that
    \begin{equation}\label{Friedrichs_ineq_alt}
        \int_\Omega v(x)^2 \, \mathrm{d}x \leq C \int_\Omega \int_{\mathbb{R}^d}  \big(v(x)-v(y)\big)^2 \gamma(x,y) \, \mathrm{d}y \,  \mathrm{d}x
    \end{equation}
    holds for all $v \in V_0(\Omega; \gamma):= V_0(\Omega,K_\gamma, \lambda_d)$. Because of (\ref{Ungleichungskette_B}), our nonlocal Friedrichs inequality (\ref{Friedrichs_Inequality}) perfectly renders (\ref{Friedrichs_ineq_alt}).    
\end{Rem}

The question arises under which specific circumstances the nonlocal Friedrichs inequality (\ref{Friedrichs_Inequality}) is satisfied in $V_0$. A sufficient criterion is given in \Cref{Prop_Friedrichs_validity} below.

\begin{Prop}\label{Prop_Friedrichs_validity}
  Let $\emptyset \neq \Omega \subseteq \mathbb{R}^d$ be a bounded domain and let $K \in \mathcal{K}$ be symmetric with respect to the Borel measure $\lambda$ on $\mathbb{R}^d$. If $\essinf_{x \in \Omega} K(x,\Gamma) >0$, the nonlocal Friedrichs inequality holds in $V_0$.
\end{Prop}

\textbf{Remark:} Above and in the following, the essential infimum/supremum is always understood with respect to the Borel measure $\lambda$ on $\mathbb{R}^d$, i.\,e., 
\begin{align*}
    \essinf_{x \in A} K(x,B) &:= \sup\left\{c \in \mathbb{R}: K(x,B) \geq c \text{ for $\lambda$-a.e. } x \in A\right\},\\
    \esssup_{x \in A} K(x,B) &:= \inf\left\{c \in \mathbb{R}: K(x,B) \leq c \text{ for $\lambda$-a.e. } x \in A\right\}.
\end{align*}

\begin{proof}
    Let $u \in V_0$. Then, the assertion is a direct consequence of the fact that
    \begin{align*}
        \mathcal{B}(u,u) & \geq \frac{1}{2} \int_\Omega \int_{\mathbb{R}^d} \big(u(x)-u(y)\big)^2 \, K(x, \mathrm{d}y) \, \mathrm{d}\lambda(x) \\
        &= \frac{1}{2}\int_\Omega \int_\Omega \big(u(x)-u(y)\big)^2 \, K(x, \mathrm{d}y) \, \mathrm{d}\lambda(x) + \frac{1}{2}\int_\Omega u(x)^2 \, K(x, \Gamma) \, \mathrm{d}\lambda(x) \\
        &\geq \frac{1}{2} \int_\Omega  u(x)^2 \, K(x, \Gamma) \, \mathrm{d}\lambda(x) \geq \frac{C}{2} \int_\Omega u(x)^2 \, \mathrm{d}\lambda(x)
    \end{align*}
    where $C:= \essinf_{x \in \Omega} K(x,\Gamma)$.
\end{proof}

If $K \in \mathcal{K}$ is essentially bounded in the sense that $\esssup_{x \in \Omega} K(x,\Omega) < \infty$, the assumption of $\essinf_{x \in \Omega} K(x, \Gamma) >0$ can be relaxed as follows:

\begin{theorem}\label{Theorem_Friedrichs_Ungleichung}
   Let $\emptyset \neq\Omega \subset \mathbb{R}^d$ be a bounded domain. Further, assume that $K \in \mathcal{K}$ is symmetric with respect to the Borel measure $\lambda$ on $\mathbb{R}^d$ and satsifies $\esssup_{x \in \Omega} K(x, \Omega)< \infty$. If a number $n \in \mathbb{N}$ exists and subsets $\Omega_i \subset \Omega$, \linebreak $i=1,...,n$, with $\lambda(\Omega_i)>0$ for all $i=1,...,n$ and $\Omega_i \cap \Omega_j = \emptyset$ for $i \neq j$ such that $\lambda(\Omega \setminus \bigcup_{i=1}^n \Omega_i) =0$ and both  
\begin{equation*}
    \begin{array}{lll}
     \alpha_1&:=\essinf_{x \in \Omega_1} K(x, \Gamma) > 0, & \\
     \alpha_i&:=\essinf_{x \in \Omega_i} K(x, \Omega_{i-1}) > 0 &\quad\text{for all } i=2,...,n,\\
    \end{array}
\end{equation*}
the nonlocal Friedrichs inequality holds in $V_0$. 
\end{theorem}

\begin{proof}
Let $v \in V_0$ be fixed. Because $\alpha_1 = \essinf_{x \in \Omega_1} K(x, \Gamma) > 0$, we can estimate
    \begin{equation} \label{C1}
    \begin{aligned}
     \mathcal{B}(u,u) &\geq \frac{1}{2}\int_\Omega \int_{\mathbb{R}^d} \big(v(x)-v(y)\big)^2 \, K(x,\mathrm{d}y) \, \mathrm{d}\lambda(x) \\
     &= \frac{1}{2} \int_\Omega \int_{\Omega} \big(v(x)-v(y)\big)^2 \, K(x,\mathrm{d}y)\, \mathrm{d}\lambda(x) + \frac{1}{2} \int_\Omega v(x)^2  K(x,\Gamma) \, \mathrm{d}\lambda(x)\\
     &\geq \frac{1}{2}\int_{\Omega_1} v(x)^2 K(x,\Gamma) \,\mathrm{d}\lambda(x) \geq \frac{\alpha_1}{2} \int_{\Omega_1} v(x)^2  \,\mathrm{d}\lambda(x).
    \end{aligned}  
    \end{equation}
Due to Jensen's inequality, the fact that $ \alpha := \esssup_{x \in \Omega} K(x, \Omega) < \infty$, and the symmetry of $K \in \mathcal{K}$ with respect to $\lambda$, it moreover follows that
\begin{equation}\label{C_2}
    \begin{aligned}
        \alpha_{2}\int_{\Omega_{2}} v(x)^2\, \mathrm{d}\lambda(x) &\leq \int_{\Omega_{2}} v(x)^2\, K(x,\Omega_1)\, \mathrm{d}\lambda(x) = \int_{\Omega_{2}}\int_{\Omega_{1}} v(x)^2 \, K(x,\mathrm{d}y) \, \mathrm{d}\lambda(x)\\
        &=  \int_{\Omega_{2}}\int_{\Omega_{1}} \Big(v(x)-v(y)+v(y)\Big)^2 \,K(x,\mathrm{d}y)\, \mathrm{d}\lambda(x)\\
        &\leq 2 \int_\Omega \int_{\mathbb{R}^d}\big(v(x)-v(y)\big)^2 \, K(x,\mathrm{d}y)\, \mathrm{d}\lambda(x) + 2\alpha \int_{\Omega_1} v(x)^2 \, \mathrm{d}\lambda(x)\\
        & \leq 4 \, \mathcal{B}(v,v) + 2 \alpha \int_{\Omega_1} v(x)^2 \, \mathrm{d}\lambda(x)
    \end{aligned}
\end{equation}
which implies
\[\int_{\Omega_{2}} v(x)^2 \, \mathrm{d}\lambda(x) \leq C_{2} \,  \mathcal{B}(v,v)\]
where the constant $C_2>0$ emerges from linking (\ref{C_2}) with (\ref{C1}). Substituting $\Omega_2$ by $\Omega_i$ and $\Omega_1$ by $\Omega_{i-1}$, it analogously also becomes apparent that for all $i=1,...,n$, there is a constant $C_i>0$ such that
    \[\int_{\Omega_i} v(x)^2 \, \mathrm{d}\lambda(x) \leq C_i  \, \mathcal{B}(v,v).\]
    Finally and as desired, this yields
    \[\int_\Omega v(x)^2 \, \mathrm{d}\lambda(x) = \sum_{i=1}^n \int_{\Omega_i} v(x)^2 \, \mathrm{d}\lambda(x) \leq nC \, \mathcal{B}(v,v)\]
    where $C:= \max\big\{\frac{2}{\alpha_1}, C_2,...,C_n\}$.
\end{proof}

Note that the nonlocal Friedrichs inequality is assumed to hold on the subspace $V_0$ of $V$ only. Indeed, we cannot expect to have an inequality of the form
\begin{equation}\label{Wunsch_Poincare}
    \int_\Omega u(x)^2 \, \mathrm{d}\lambda(x) \leq C \, \mathcal{B}(u,u), \quad u \in V,
\end{equation}
available because the nullspace of $\mathcal{B}$ may contain functions for which the left-hand side of (\ref{Wunsch_Poincare}) is not zero. \Cref{Example_non_zero_Kern} outlines two explicit examples where this is the case:

\begin{Bsp}\label{Example_non_zero_Kern}
Let $\emptyset \neq \Omega \subseteq \mathbb{R}^d$ be a bounded domain and let $\lambda= \lambda_d$ denote the Lebesgue Borel measure on $\mathbb{R}^d$. 
\begin{enumerate}
    \item If $\gamma: \mathbb{R}^d \times \mathbb{R}^d \rightarrow [0,\infty)$ is a Borel measurable and symmetric function and $\mathcal{K} \ni K = K_\gamma$, we have 
    \[\left\{u \in V(\Omega,K_\gamma,\lambda_d): u \equiv \text{const}\right\} \subseteq \ker{\mathcal{B}}.\]
    Especially, in the context of the nonlocal diffusion operator $\mathcal{L}_\gamma \in \mathscr{D}$, we cannot expect inequality (\ref{Wunsch_Poincare}) to be satisfied.
    \item For $d=1$ and $\mathcal{K} \ni K = K_{1,h}$ given as in (\ref{kernel_stencil}), it can be verified that
    \[\ker{\mathcal{B}} = \left\{u \in V(\Omega,K_{1,h},\lambda_1): \text{ $u$ is $h$-periodic in $\Omega \cup \Gamma$}\right\}.\]
    In particular, in the context of the three-point stencil $\mathcal{L}_{1,h}$ in $d=1$, we cannot expect (\ref{Wunsch_Poincare}) to be satisfied. 
\end{enumerate}   
\end{Bsp}

In order to adjust inequality (\ref{Wunsch_Poincare}) into a valid one, the nullspace $\ker{\mathcal{B}}$ must be taken into account. This is what motivates our choice of the nonlocal Poincaré inequality which is taking the center stage next. Note that therein the $L^2$-norm from the left-hand side of (\ref{Wunsch_Poincare}) is substituted by the $L^2$-norm of the projection onto $\ker{\mathcal{B}}$, see \Cref{Lemma_ker_B_orthogonal_L2_norm} below.

\subsection{The Nonlocal Poincaré Inequality} \label{Subsection_Poincare_inequality}

In what follows, the nonlocal Poincaré inequality shall be elucidated in closer detail which requires the existence of a constant $C>0$ such that
\begin{equation}\label{Poincare_Inequality}
    \inf_{w \in \ker{\mathcal{B}}}\int_\Omega \big(u(x)-w(x)\big)^2 \, \mathrm{d}\lambda(x) \leq C\, \mathcal{B}(u,u)
\end{equation}
holds for all $v \in V$. Utilizing \Cref{Lemma_ker_B_orthogonal_L2_norm} and mimicking \cite[Lemma 4.2]{Frerick2022TheNN}, we obtain:

\begin{lemma} \label{Equivalences_Poincaré}
Let $\Omega \subseteq \mathbb{R}^d$ be a nonempty and open set and assume that the transition kernel $K \in \mathcal{K}$ is symmetric with respect to the Borel measure $\lambda$ on $\mathbb{R}^d$. Then, the following statements are equivalent:
\begin{itemize}
    \item[(i)] The nonlocal Poincaré inequality holds in $V$.
    \item[(ii)] A constant $C>0$ exists such that $\norm{u}_{L^2(\Omega,\lambda)}^2 \leq C \mathcal{B}(u,u)$ is valid for all \linebreak $u \in (\ker{\mathcal{B}})^\perp$.
    \item[(iii)] A constant $C>0$ exists such  $\norm{u-P(u)}_{L^2(\Omega,\lambda)}^2 \leq C \mathcal{B}(u,u)$ holds for all $u \in V$ and $P(u)$ is the unique element in $\ker{\mathcal{B}}$ with $\langle u - P(u),w \rangle_{V} = 0$ for all $w \in \ker{\mathcal{B}}$.
\end{itemize}
\end{lemma}

In marked contrast to the nonlocal Friedrichs inequality which we have seen to be consistent with both the local formulation and previous nonlocal versions as utilized by \cite{felsingerDirichlet}, \cite{foghem2022general}, and \cite{Frerick2022TheNN}, we point out that our nonlocal Poincaré inequality not necessarily is. The details are explicated in \Cref{Rem_Strong_Poincare} next:

\begin{Rem}\label{Rem_Strong_Poincare}
    Let $\Omega \subseteq \mathbb{R}^d$ be open. In the local theory of PDEs, the classical Poincaré inequality (e.g., \cite[Theorem 13.27]{leoni2017first}) demands the existence of a constant $C>0$ such that
\[\norm{u-u_\Omega}_{L^p(\Omega,\lambda_d)} \leq C \norm{\nabla u}_{L^p(\Omega,\lambda_d)}, \quad u \in W^{1,p}(\Omega),\ 1 \leq p < \infty,\]
where $\lambda_d$ is the Lebesgue Borel measure on $\mathbb{R}^d$ and $u_\Omega := \left( \frac{1}{\lambda_d(\Omega)} \int_\Omega u(x) \, \mathrm{d}\lambda_d(x) \right) \chi_{\mathbb{R}^d}$.\\
When $ \lambda$ is an arbitrary Borel measure on $\mathbb{R}^d$, $K \in \mathcal{K}$ is $\lambda$-symmetric, and $\Omega$ satisfies $0 < \lambda(\Omega) < \infty$, the nonlocal counterpart of this inequality is given by
\begin{equation}\label{strongPoincare}
    \int_\Omega \big( u(x) - u_\Omega\big)^2 \, \mathrm{d}\lambda(x) \leq C \mathcal{B}(u,u), \quad u \in V.
\end{equation}
Because, by requirement, all constant functions are contained in $\ker{\mathcal{B}}$, one can easily show that inequality (\ref{strongPoincare}) is stronger than the nonlocal Poincaré inequality defined in \Cref{Defi_Nonlocal_Inequalities}. Analogously to \cite[Lemma 6.1.5]{Huschens}, more particularly the following equivalence holds:
\begin{itemize}
    \item The nonlocal Poincaré inequality is satisfied in $V$ and $\ker{\mathcal{B}}$ only contains the constant functions
    \item Inequality (\ref{strongPoincare}) holds in $V$.
\end{itemize}
Note that in general, one cannot expect inequality (\ref{strongPoincare}) to be met in V. Indeed, if the latter holds, then the argument from \cite[Lemma 4.2]{Frerick2022TheNN} yields a constant $\tilde{C}>0$ with
\begin{equation*}
    \int_\Omega \int_\Omega \big(u(x)-u(y)\big) \, \mathrm{d}y \, \mathrm{d}x \leq \Tilde{C}\, \mathcal{B}(u,u)
\end{equation*}
which in turn implies that $u$ is constant in $\Omega$, i.\,.e, $u \einschraenkung_\Omega \equiv \tilde{c}$ for a constant $\tilde{c} \in \mathbb{R}$. Since by \Cref{lemma_symmetry_Fubini} the $\lambda$-symmetry of $K \in \mathcal{K}$ implies
\begin{align*}
    0 = \int_\Omega \int_\Gamma \big(u(x)-u(y)\big)^2 \, K(x, \mathrm{d}y)\, \mathrm{d}x &= \int_\Omega \int_\Gamma (\tilde{c}-u(y))^2\, K(x, \mathrm{d}y)\, \mathrm{d}x \\
    &= \int_\Gamma \big(\tilde{c}-u(y)\big)^2 \Big(\int_\Omega K(y, \mathrm{d}x) \Big) \, \mathrm{d}y
\end{align*}
and, by definition, $\int_\Omega K(y, \mathrm{d}x) = K(y,\Omega) > 0$ holds for all $y \in \Gamma$, it follows that $\ker{\mathcal{B}} = \left\{u \in V: u \equiv \text{const}\right\}$. However, in \Cref{Example_non_zero_Kern}, we have seen that for the nonlocal function space $V(\Omega,K_{1,h},\lambda)$ associated with the stencil operator $\mathcal{L}_{1,h}$ this set equality does not hold true.
\end{Rem}

Unlike in \Cref{Subsection_friedrichs_inequality}, our consideration of the nonlocal Poincaré inequality is ended without giving a universal criterion for its validity in $V$. Instead, the interested reader is referred to
\begin{itemize}
    \item[(i)] \cite[Section 4]{Frerick2022TheNN} for sufficient conditions for the strong Poincaré inequality to be satisfied in $V(\Omega,K_\gamma,\lambda_d)$;
    \item[(ii)] \cite[Section 8.1.2]{Huschens} for a proof of the nonlocal Poincaré inequality in the nonlocal space $V(\Omega,K_{d,h}, \lambda_d)$ where $\emptyset \neq \Omega \subseteq \mathbb{R}^d$ is a bounded domain;
    \item[(iii)] \cite[Section 8.2]{Huschens} for a proof of the nonlocal Poincaré inequality in the nonlocal space $V(\Omega,K_\epsilon,\lambda_2)$ where $\emptyset \neq \Omega \subseteq \mathbb{R}^2$ is a bounded domain satisfying the geometric assumption from \cite[Definition 8.2.3]{Huschens}.
\end{itemize}
In \Cref{Sec_discret_Neumann} and for $\Omega \subseteq \mathbb{R}^d$ as in (ii), the validity of the nonlocal Poincaré inequality in $V(\Omega,K_{d,h},\mu_s)$ is shown where $\mu_s$, $s\in \mathbb{R}^d$, is the measure from \Cref{Bsp_Self_adjointness_Stencil}.

\section{The Nonlocal Dirichlet Problem}\label{Sec_Dirichlet_Problem}

Let $\lambda$ be an arbitrary Borel measure on $\mathbb{R}^d$, let $\emptyset \neq \Omega \subseteq \mathbb{R}^d$ be an open set satisfying $\lambda(\Omega) < \infty$, and let $K \in \mathcal{K}$ be $\lambda$-symmetric. We are concerned with the nonlocal Dirichlet problem 
\begin{equation}\label{Dirichlet_problem}\tag{DP}
\mathcal{L}u = f \ \text{ in } \Omega, \quad  u = g \ \text{ in } \Gamma,
\end{equation}
whose governing nonlocal operator $\mathcal{L}$ is given as in (\ref{Nonlocal_operator}), i.\,e.,
\[\mathcal{L}u(x):= \operatorname{PV}\int_{\mathbb{R}^d}\big(u(x)-u(y)\big) \, K(x, \mathrm{d}y), \quad x \in \mathbb{R}^d.\]
The aim of the section is to establish a well-posedness result for problem (\ref{DP}) which guarantees the existence of a unique weak solution depending continuously on $f$ and $g$. Because our test function space $V_0$ is Hilbert (cf. \Cref{Completeness_Dirichlet}), our approach is analogous to \cite[Section 4.3.2]{foghem2020l2} and based upon a standard tool from Hilbert space theory: the Lax-Milgram theorem \cite[Theorem 1 in Section 6.2.1]{evans10}. Recall from \Cref{definition_weak_solutions} that an element $u \in V$ is called a weak solution  of problem (\ref{DP}) if $u-g \in V_0$ and
\[\mathcal{B}(u,v) = \int_\Omega f(x)v(x) \, \mathrm{d}\lambda(x)\]
holds for all $v \in V_0$.

\begin{theorem}{\textnormal{(Homogeneous Case)}}\label{Well_posedness_Dirichlet}
    Let $f \in L^2(\Omega,\lambda)$ and $g=0$. If the nonlocal Friedrichs inequality is satisfied in $V_0$, the homogeneous Dirichlet problem (\ref{Dirichlet_problem}) has a unique weak solution $u$ and a constant $C>0$ exists such that
    \begin{equation}\label{regularity_estimate_dirichlet}
    \norm{u}_{V} \leq C \norm{f}_{L^2(\Omega, \lambda)}. 
    \end{equation}
    Especially, $u$ depends continuously on the Dirichlet data $f \in L^2(\Omega,\lambda)$.
\end{theorem}

\begin{proof}
    Due to the validity of the nonlocal Friedrichs inequality in $V_0$, a constant $C>0$ exists such that, for all $u \in V_0$, we have
    \begin{equation}\label{ellipticity_homogeneous_Dirichlet}
        \norm{u}^2_V \leq C \mathcal{B}(u,u).
    \end{equation}
    Since moreover $\Lambda_f: v \mapsto \int_\Omega f(x)v(x) \, \mathrm{d}\lambda(x)$ defines a linear and continuous functional on the Hilbert space $V_0$ and $\mathcal{B}$ is a bounded bilinear form on $V_0$, the existence of a unique weak solution $u \in V_0$ of the homogeneous problem (\ref{DP}) is a direct consequence of the Lax-Milgram theorem \cite[Theorem 1 in Section 6.2.1]{evans10}. Utilizing (\ref{ellipticity_homogeneous_Dirichlet}), one has
    \begin{align*}
        \norm{u}_{V}^2 \leq C \,  \mathcal{B}(u,u) &= C \int_\Omega f(x)u(x) \, \mathrm{d}\lambda(x) \leq C\, \norm{f}_{L^2(\Omega,\lambda)} \norm{u}_{V}
    \end{align*}
    which yields (\ref{regularity_estimate_dirichlet}) and because $f$ maps \textit{linearly} to the unique weak solution $u$, the continuous dependency of the solution $u$ on the input data follows.
\end{proof}

We next want to tackle the non-homogeneous case of $g \neq 0$. This necessitates to introduce the trace space $T=T(\Omega,K,\lambda)$ of $V=V(\Omega,K,\lambda)$ which consists of the restrictions $v \mapsto v\einschraenkung_\Gamma$ of the elements of $V$ onto the nonlocal boundary $\Gamma$, i.\,e., 
\begin{equation}
    T:=T(\Omega,K,\lambda) :=\left\{v:\Gamma \rightarrow \mathbb{R}: \text{ an element $u \in V$ exists with } u \einschraenkung_\Gamma = v\right\}.
\end{equation}
Hereinafter, this space shall always be equipped with its natural norm
\begin{equation}\label{defi_trace_space_norm}
    \norm{v}_{T} := \inf \left\{ \norm{u}_{V}: u \in V \text{ and } u \einschraenkung_{\Gamma} = v \right\}.
\end{equation}
Analogously to \cite[Theorem 4.22]{foghem2020l2} and \cite[Theorem 3.5]{felsingerDirichlet}, we verify that the well-posedness of the non-homogeneous Dirichlet problem (\ref{DP}) can be guaranteed if $g \in T$:

\begin{theorem}\label{Well_posedness_Dirichlet}
    Let $f \in L^2(\Omega,\lambda)$ and $g \in T$. If the nonlocal Friedrichs inequality holds in $V_0$, the Dirichlet problem (\ref{Dirichlet_problem}) has a unique weak solution $u\in V$ and a constant $C>0$ exists such that
\begin{equation}\label{regularity_estimate_dirichlet}
    \norm{u}_{V} \leq C \left(\norm{f}_{L^2(\Omega \cup \Gamma, \lambda)} + \norm{g}_{T}\right).
\end{equation}
  Especially, $u$ depends continuously on the Dirichlet data $(f,g) \in L^2(\Omega,\lambda) \times T$.
\end{theorem}

\begin{proof}
     Because $g \in T$, an extension $z \in V$ of $g$ exists with $z\einschraenkung_\Gamma = g$. Let the mapping $\Lambda_{f,z}: V_0 \rightarrow \mathbb{R}$ be given by
    \[\Lambda_{f,z}(v) := \int_\Omega f(x)v(x) \, \mathrm{d}\lambda(x) - \mathcal{B}(z,v), \ v \in V_0.\]
    Then, $\Lambda_{f,z}$ is a continuous linear functional. Since still $\mathcal{B}$ is a bounded bilinear form on $V_0$ for which a constant $\Tilde{C}>0 $ exists such that  we have
    \begin{equation}\label{ellipticity_Dirichlet_nonhomo}
        \norm{u}^2_V \leq \Tilde{C} \, \mathcal{B}(u,u) \quad \text{ for all } u \in V_0
    \end{equation}
    due to the validity of the nonlocal Friedrichs inequality in $V_0$, the Lax-Milgram theorem can again be applied to observe the existence of a unique element $\Tilde{u} \in V_0$ for which 
    \[\mathcal{B}(\Tilde{u},v) = \Lambda_{f,z}(v) = \int_\Omega f(x)v(x) \, \mathrm{d}\lambda(x) - \mathcal{B}(z,v)\]
    holds for all $v \in V_0$. Especially, for $u:=\Tilde{u}+z$ and all $v \in V_0$, it follows that
    \begin{align*}
        \mathcal{B}(u,v) = \mathcal{B}(\Tilde{u},v) + \mathcal{B}(z,v) &= \int_\Omega f(x)v(x) \, \mathrm{d}\lambda(x) - \mathcal{B}(z,v) + \mathcal{B}(z,v)\\
        &= \int_\Omega f(x)v(x) \, \mathrm{d}\lambda(x)
    \end{align*}
    which proves $u \in V$ to be a weak solution of problem (\ref{Dirichlet_problem}). It is also the only one given that for any other weak solution $\hat{u} \in V$ of (\ref{Dirichlet_problem}), we have $u-\hat{u}\in V_0$ such that (\ref{ellipticity_Dirichlet_nonhomo}) implies 
    \[0= \mathcal{B}(u - \hat{u}, u-\hat{u}) \geq \frac{1}{\Tilde{C}} \norm{u-\hat{u}}_V \geq 0.\]
    Hence, $u = \hat{u}$, as desired. The estimate (\ref{regularity_estimate_dirichlet}) remains to be verified: Because \linebreak $u-z = \Tilde{u}\in V_0$, inequality (\ref{ellipticity_Dirichlet_nonhomo}) is again used to obtain
    \begin{align*}
        \frac{1}{\Tilde{C}} \norm{\Tilde{u}}_{V}^2 \leq \mathcal{B}(\Tilde{u},\Tilde{u}) &= \int_\Omega f(x) \Tilde{u}(x) \, \mathrm{d}\lambda(x) - \mathcal{B}(z,\Tilde{u}) \\
        &\leq \norm{f}_{L^2(\Omega,\lambda)}\norm{\Tilde{u}}_{L^2(\Omega,\lambda)} + \norm{z}_{V} \norm{\Tilde{u}}_{V} \leq \norm{\Tilde{u}}_{V}\Big( \norm{f}_{L^2(\Omega,\lambda)} + \norm{z}_{V}\Big)
    \end{align*}
    such that, consequently,
    \begin{align*}
        \norm{u}_{V} &= \norm{\Tilde{u}+z}_{V} \leq \norm{\Tilde{u}}_{V}+ \norm{z}_{V} \leq (\Tilde{C}+1)\big(\norm{f}_{L^2(\Omega,\lambda)} + \norm{z}_{V}\big).
    \end{align*}
    Recalling that $\norm{g}_{T} = \inf\left\{\norm{z}_{V}:z \in V \text{ and } z\einschraenkung_\Gamma = g \right\}$ holds, the validity of inequality (\ref{regularity_estimate_dirichlet}) is obtained for $C:= \Tilde{C}+1$. From the latter, also the continuous dependency of $u$ on $f$ and $g$ follows because the Dirichlet data again maps \textit{linearly} to the unique weak solution $u$ of problem (\ref{Dirichlet_problem}). 
\end{proof}

At the moment, the condition of $g \in T$ is rather abstract and the question arises for which mappings $g:\Gamma \rightarrow \mathbb{R}$, it is actually satisfied. In order to given an answer, an alternative description of $T=T(\Omega,K,\lambda)$ is provided which is more informative in that regard. If $K \in \mathcal{K}$ is essentially bounded meaning that $\esssup_{x \in \Omega} K(x, \Gamma) < \infty$, the next theorem shows that $T$ equals a weighted Lebesgue space:

\begin{theorem}\label{Trace_Theorem_1}
Let $\lambda(\Omega) < \infty$ and let $K \in \mathcal{K}$ be $\lambda$-symmetric and satisfy \linebreak $\esssup_{x \in \Omega} K(x,\Gamma) < \infty$. If $\omega: \Gamma \rightarrow (0,\infty]$ is given by
\[\omega(y):= K(y,\Omega),\]
then the set $\{y \in \Gamma: w(y) = \infty\}$ is a set of zero measure (with respect to $\lambda$) and we have
\[T = \left\{v: \Gamma \rightarrow \mathbb{R}: \int_\Gamma v(y)^2 \, \omega(x) \, \mathrm{d}\lambda(y) < \infty \right\} = L^2(\Gamma, \omega \cdot \lambda)\]
where the norms $\norm{\cdot}_{T}$ and $\norm{\cdot}_{L^2(\Gamma, \omega \cdot \lambda)}$ are equivalent.
\end{theorem}

\begin{proof} The proof is inspired by \cite[Theorem 5.2]{Frerick2022TheNN}: Because $K$ is a transition kernel, the mapping $\omega$ is Borel measurable which justifies to consider $\omega$ as a $\lambda$-density in the following. By the definition of $\Gamma$, $\omega > 0$ holds and the $\lambda$-symmetry of $K \in \mathcal{K}$ yields
\begin{align*}\label{eq_w_boundedness}
        \int_\Gamma \omega(y) \, \mathrm{d}\lambda(y) = \int_\Gamma K(y,\Omega) \, \mathrm{d}\lambda(y) &= \int_\Omega K(x,\Gamma) \, \mathrm{d}\lambda(x) \\
        &\leq \esssup_{x \in \Omega} K(x,\Gamma) \lambda(\Omega) < \infty.
\end{align*}
Let now $v \in T$ be given arbitrarily. By definition, an element $u \in V$ exists with $ u \einschraenkung_\Gamma = v$. Hence, an application of \Cref{lemma_symmetry_Fubini} (in the second step) and Jensen's inequality (in the third step) gives
    \begin{equation}\label{eq_669_1}
    \begin{aligned}
        \norm{v}^2_{L^2(\Gamma, \omega \cdot \lambda)} &= \int_\Gamma u(y)^2\omega(y) \, \mathrm{d}\lambda(y) = \int_\Omega \int_\Gamma \Big(u(x)-u(y)+u(x)\Big)^2 \, K(x,\mathrm{d}y) \, \mathrm{d}\lambda(x)\\
        &\leq 2 \int_\Omega u(x)^2 K(x,\Gamma) \, \mathrm{d}\lambda(x) + 2 \int_\Omega \int_{\mathbb{R}^d} \Big(u(x)-u(y)\Big)^2 \, K(x,\mathrm{d}y)\, \mathrm{d}\lambda(x)\\
        &\leq 2 \max\left\{\esssup_{x \in \Omega}  K(x,\Gamma),1\right\}\norm{u}^2_{V}
    \end{aligned}
    \end{equation}
    which proves $T \subseteq L^2(\Gamma, \omega \cdot \lambda)$. For the other inclusion, pick $v \in L^2(\Gamma, \omega \cdot \lambda)$. Then, 
    \begin{equation}\label{eq_87}
    \begin{aligned}
        \norm{v\chi_\Gamma}^2_{V}= \int_\Omega \int_\Gamma v(y)^2 \, K(x, \mathrm{d}y) \, \mathrm{d}\lambda(x) &= \int_\Gamma v(y)^2 K(y,\Omega) \, \mathrm{d}\lambda(y) = \norm{v}^2_{L^2(\Gamma, \omega \cdot \lambda)} < \infty
        \end{aligned}
    \end{equation}
    which yields $v\chi_\Gamma \in V$ and, thus, $T = L^2(\Gamma, \omega \cdot \lambda)$. The equivalence of the norms $\norm{\cdot}_{T}$ and $\norm{\cdot}_{L^2(\Gamma, \omega \cdot \lambda)}$ is a consequence of (\ref{eq_669_1}) and (\ref{eq_87}).  
\end{proof}

The requirement of $\esssup_{x \in \Omega} K(x,\Gamma) < \infty$ appearing in \Cref{Trace_Theorem_1} is rather strong and, for instance, not satisfied for the fractional kernel
\[K_{\gamma_s}: \mathbb{R}^d \times \mathcal{B}(\mathbb{R}^d) \rightarrow [0,\infty], \quad (x,A) \mapsto \int_A \frac{1}{\norm{x-y}^{d+2s}}, \quad s \in (0,1).\]
We refer to \cite{MR4013823} for the proof that under certain conditions on $\Omega$ and for \linebreak $\rho_z:= \operatorname{dist}(z, \partial \Omega)$, $z \in \mathbb{R}^d$, we have
\[T(\Omega,K_{\gamma_s},\lambda_d) = \left\{v: \mathbb{R}^d\setminus \Omega \rightarrow \mathbb{R}: \int_{\mathbb{R}^d \setminus \Omega} \int_{\mathbb{R}^d \setminus \Omega} \frac{|v(x)-v(y)|^2}{\left(\norm{x-y}+\rho_x+\rho_y\right)^{d+2s}} \, \mathrm{d}y \, \mathrm{d}x < \infty \right\}\]
and also hint at \cite[Section 3.3.2/Proposition 3.37]{foghem2020l2} for a comparable result in the context of unimodal kernel functions $\gamma$ having full support in $\mathbb{R}^d$. In \Cref{Trace_Theorem_2} below and for $K \in \mathcal{K}$ satisfying $K(x,\Gamma) < \infty$ for $\lambda$-a.e. $x \in \Omega$, we confine ourselves to providing a necessary condition for $g: \Gamma \rightarrow \mathbb{R}$ to be contained in $T=T(\Omega,K,\lambda)$ only:

\begin{theorem}\label{Trace_Theorem_2}
Again, let $K \in \mathcal{K}$ be $\lambda$-symmetric; assume that $\lambda(\Omega) < \infty$ and $K(x,\Gamma) <\infty$ holds for $\lambda$-a.e. $x\in \Omega$. For $c \in (0,\infty)$, we define $\omega: \Gamma \rightarrow (0,\infty]$ by
\[\omega(y):= \int_\Omega \frac{1}{K(s,\Gamma)+c}K(y,ds).\]
Then, $\omega$ is a Borel measurable and positive function, $\{y \in \Gamma: \omega(y)=\infty\}$ is a null set (with respect to $\lambda$), and we have
\[T \subseteq \left\{v: \Gamma \rightarrow \mathbb{R}: \int_\Gamma v(y)^2 \, \omega(x) \, \mathrm{d}\lambda(y) < \infty \right\} = L^2(\Gamma, \omega \cdot \lambda).\]
\end{theorem}

\begin{proof} The proof generalizes the arguments from \cite[Theorem 5.1]{Frerick2022TheNN}: Because $K \in \mathcal{K}$, the mapping $y \mapsto \omega(y) = \int_\Omega \frac{1}{K(s,\Gamma)+c}K(y,\mathrm{d}s)$ is Borel measurable which warrants the use of $\omega$ as a $\lambda$-density in the following. By definition, $\omega(y)>0$ holds and using the $\lambda$-symmetry of $K \in \mathcal{K}$ via \Cref{lemma_symmetry_Fubini}, we obtain
\begin{align*}
    \int_\Gamma \omega(y) \, \mathrm{d}\lambda(y) &= \int_\Gamma \int_\Omega \frac{1}{K(x,\Gamma)+c} \, K(y,\mathrm{d}x) \, \mathrm{d}\lambda(y) \\
    &= \int_\Omega \int_\Gamma \frac{1}{K(x,\Gamma)+c} \, K(x,\mathrm{d}y) \, \mathrm{d}\lambda(x) \leq \lambda(\Omega)
\end{align*}
which yields $\lambda(\{y \in \Gamma: \omega(y) = \infty\})=0$. Let now $v \in T$ be given. Then, an element $u \in V$ exists with $u \einschraenkung_\Gamma = v$ and it is valid that
\begin{align*}
    \norm{v}_{L^2(\Gamma,\omega \cdot \lambda)}^2 &=\int_\Gamma u(y)^2 \omega(y) \, \mathrm{d}\lambda(y) = \int_\Gamma \int_\Omega u(y)^2 \frac{1}{K(x,\Gamma)+c} \, K(y,\mathrm{d}x) \, \mathrm{d}\lambda(y)\\
    &= \int_\Omega \int_\Gamma \Big(u(x)-u(y)-u(x)\Big)^2 \, \frac{K(x,\mathrm{d}y)}{K(x,\Gamma)+c} \, \mathrm{d}\lambda(x)\\
    &\leq 2 \int_\Omega \int_\Gamma \Big(u(x)-u(y)\Big)^2 \, \frac{K(x,\mathrm{d}y)}{K(x,\Gamma)+c} \, \mathrm{d}\lambda(x) + 2 \int_\Omega u(x)^2 \, \mathrm{d}\lambda(x)\\
    &\leq 2 \max\left\{\frac{1}{c},1\right\} \norm{u}_{V}^2.
\end{align*}
Especially, $T \subseteq L^2(\Gamma, \omega \cdot \lambda)$.
\end{proof}

\section{The Nonlocal Neumann Problem}\label{Sec_Neumann_Problem}

Let $\lambda$ be an arbitrary Borel measure on $\mathbb{R}^d$, let $\emptyset \neq \Omega \subseteq \mathbb{R}^d$ be an open set satisfying $\lambda(\Omega) < \infty$, and let $K \in \mathcal{K}$ be a $\lambda$-symmetric transition kernel. Throughout this section, we are concerned with the nonlocal Neumann problem 
\begin{equation}\label{Neumann_problem}\tag{NP}
    \mathcal{L}u = f \ \text{ in } \Omega, \quad \mathcal{N}u = g \ \text{ in } \Gamma,
\end{equation}
whose governing nonlocal operators $\mathcal{L}$ and $\mathcal{N}$ are given as in (\ref{Nonlocal_operator}) and (\ref{Neumann_operator}), respectively, i.\,e.,
\begin{align*}
    \mathcal{L}u(x) &:= \operatorname{PV} \int_{\mathbb{R}^d}\big(u(x)-u(y) \big) \, K(x, \mathrm{d}y), \ x \in \mathbb{R}^d, \\
    \mathcal{N}u(y) &:= \int_\Omega \big(u(y)-u(x)\big) \, K(y, \mathrm{d}x), \ y \in \mathbb{R}^d.
\end{align*}
Similar as for the nonlocal Dirichlet problem (\ref{DP}), we aim to establish a well-posedness result and provide the existence of a weak solution that depends continuously on $f$ and $g$. Recall that a function $u \in V$ is called a weak solution of problem (\ref{NP}) if
\[\mathcal{B}(u,v) = \int_\Omega f(x)v(x) \, \mathrm{d}\lambda(x) + \int_\Gamma g(y)v(y) \, \mathrm{d}\lambda(y)\]
holds for all $v \in V$. \\

Let us start with an important observation regarding the uniqueness of weak solutions: By definition, the bilinear form $\mathcal{B}$ annihilates all elements $w \in \ker{\mathcal{B}}$, i.\,e., if $u \in V$ and $w \in \ker{\mathcal{B}}$, we have 
  \[\mathcal{B}(u+w,v) = \mathcal{B}(u,v) \text{ for all } v\in V.\]
  Consequently, we cannot expect a weak solution of problem (\ref{NP}) to be unique and for the existence of such a solution, obviously also a compatibility condition needs to be satisfied which demands, for all $w \in \ker{\mathcal{B}}$ that
  \begin{equation}\label{compatibility}
      \Lambda_{f,g}(v) := \int_\Omega f(x)w(x) \, \mathrm{d}\lambda(x) + \int_\Gamma g(y)w(y) \, \mathrm{d}\lambda(y) =0.
  \end{equation}
In order to ensure the linear operator $\Lambda_{f,g}:V \rightarrow \mathbb{R}$ to exist and to be well-defined, we follow \cite[Section 3]{Frerick2022TheNN} and henceforth assume that $f \in L^2(\Omega,\lambda)$ and that $g: \Gamma \rightarrow \mathbb{R}$ satisfies the \textit{continuous functional condition}:

\begin{definition}
    Let $\emptyset \neq \Omega \subseteq \mathbb{R}^d$ be open and bounded and let $K \in \mathcal{K}$ be $\lambda$-symmetric with respect to the Borel measure $\lambda$ on $\mathbb{R}^d$. We say that a function $g: \Gamma \rightarrow \mathbb{R}$ satisfies the \textbf{continuous functional condition} if $g$ is Borel measurable and 
    \[v \mapsto \int_\Gamma g(y)v(y) \, \mathrm{d}y\]
    is a linear and continuous functional on $V$. The continuity constant shall henceforth be denoted by $\norm{g}_\Gamma$.
\end{definition}

Below, an existence result for a weak solution of problem (\ref{NP}) is presented which again utilizes the Lax-Milgram theorem as the main tool:

\begin{theorem}\label{Well_posedness_Neumann}
Let $f \in L^2(\Omega,\lambda)$ and assume that $g:\Gamma \rightarrow \mathbb{R}$ satisfies the continuous functional condition. If the compatibility condition (\ref{compatibility}) is satisfied and the nonlocal Poincaré inequality holds in $V$, a weak solution of the Neumann problem (\ref{Neumann_problem}) exists which is unique up to all $w \in \ker{\mathcal{B}}$. Again, the weak solutions $u \in V$ depend continuously on the Neumann data $(f,g)$ since a constant $C>0$ independent of $f$ and $g$ exists such that we have
\begin{equation}\label{regularity_estimate_Neumann}
\norm{u-P(u)}_{V} \leq C \big( \norm{f}_{L^2(\Omega,\lambda)} + \norm{g}_\Gamma\big)
\end{equation}
where $P(u)$ is the unique element in $\ker{\mathcal{B}}$ for which $\langle u- P(u),w \rangle_{V} = 0$ holds for all $w \in \ker{\mathcal{B}}$. 
\end{theorem}

\begin{proof}
    We consider the quotient space
    \[\hat{V}:= V/\ker{\mathcal{B}}= \{u + \ker{\mathcal{B}}: u \in V\} =\{[u]: u \in V\}\]
equipped with the quotient norm
\begin{align*}
    \norm{[u]}_{\hat{V}}&:=\inf_{w \in \ker{\mathcal{B}}}\norm{u-w}_{V}.
\end{align*}
Because $\ker{\mathcal{B}}$ is a closed subspace of $V$, $(\hat{V},\norm{\cdot}_{\hat{V}})$ is a Hilbert space with the inner product $\langle \cdot,\cdot \rangle_{\hat{V}}:\hat{V} \times \hat{V} \rightarrow \mathbb{R}$ being given by  
\[\langle [u],[v]\rangle_{\hat{V}}:= \inf_{w \in \ker{\mathcal{B}}} \langle u-w,v-w \rangle_{V}.\]
In (the proof of) \Cref{Lemma_ker_B_orthogonal_L2_norm}, we have seen that each $u \in V$ can be tied to a unique element $P(u) \in \ker{\mathcal{B}}$ such that
\[\norm{[u]}_{\hat{V}}=\inf_{w \in \ker{\mathcal{B}}}\norm{u-w}_{V} = \norm{u-P(u)}_{V}\]
and $\langle u- P(u), w \rangle_{V} = 0$ holds for all $w \in \ker{\mathcal{B}}$. Since we have $\mathcal{B}(u_1,v_1) = \mathcal{B}(u_2,v_2)$ for all $u_1,u_2 \in[u]$ and $v_1,v_2 \in [v]$, it is reasonable to define $\mathcal{B}_{\hat{V}}(\cdot,\cdot):\hat{V} \times \hat{V} \rightarrow \mathbb{R}$ by
\[\mathcal{B}_{\hat{V}}([u],[v]):= \mathcal{B}(u,v).\]
Then, $\mathcal{B}_{\hat{V}}$ is a bilinear form on  $(\hat{V},\norm{\cdot}_{\hat{V}})$ which is also continuous given that by the Cauchy Schwarz inequality and for all $[u],[v] \in \hat{V}$, we have 
\begin{align*}
    |\mathcal{B}_{\hat{V}}([u],[v])| = |\mathcal{B}(u,v)| = \mathcal{B}\big(u-P(u),v-P(v)\big) &\leq \norm{u-P(u)}_{V}\norm{v-P(v)}_{V} \\
    &= \norm{[u]}_{\hat{V}}\norm{[v]}_{\hat{V}}.
\end{align*}
By assumption, the Poincaré inequality holds; thus, from applying \Cref{Lemma_ker_B_orthogonal_L2_norm}, it follows that
\begin{align*}
     \norm{[u]}_{\hat{V}}^2 &= \norm{u-P(u)}^2_{V} \leq \int_\Omega \big(u(x)-P(u)(x)\big)^2 \, \mathrm{d}\lambda(x) + 2\mathcal{B}\big(u-P(u),u-P(u)\big)\\
     &= \inf_{w \in \ker{\mathcal{B}}} \int_\Omega \big(u(x)-w(x)\big)^2 \, \mathrm{d}\lambda(x) + 2\mathcal{B}(u,u) \leq (C+2) \, \mathcal{B}(u,u) \\
     &= (C+2) \, \mathcal{B}_{\hat{V}}([u],[u]), \quad [u] \in \hat{V},
\end{align*}
where $C>0$ is a constant. 
\newpage
\noindent
Let us now consider the linear functional $\Lambda_{f,g}: \hat{V} \rightarrow \mathbb{R}$,
\[\Lambda_{f,g}([v]):= \int_\Omega f(x)v(x) \, \mathrm{d}\lambda(x) + \int_\Gamma g(y) v(y) \, \mathrm{d}\lambda(y),\]
which is well-defined due to the compatibility condition. The latter also implies
\begin{align*}
    |\Lambda_{f,g}([v])| &= \left|\int_\Omega f(x)v(x) \mathrm{d}\lambda(x) + \int_\Gamma g(y) v(y) \, \mathrm{d}\lambda(y)\right| \\
    &= \left|\int_\Omega f(x)\Big(v(x)-P(v)(x)\Big) \mathrm{d}\lambda(x) + \int_\Gamma g(y) \big(v(y) - P(v)(y)\big) \, \mathrm{d}\lambda(y)\right| \\
    &\leq \big(\norm{f}_{L^2(\Omega,\lambda)} + \norm{g}_\Gamma\big) \norm{[v]}_{\hat{V}} ,
\end{align*}
i.\,e., $\Lambda_{f,g}$ is continuous on $(\hat{V},\norm{\cdot}_{\hat{V}})$. Consequently, the Lax-Milgram theorem yields the existence of a unique $u \in \hat{V}$ such that $\mathcal{B}_{\hat{V}}([u],[v])= \Lambda_{f,g}([v])$ holds for all $v \in \hat{V}$. Especially, if $u\in [u]$ and if $v \in V$ is arbitrary, it follows that
\[\mathcal{B}(u,v) = \mathcal{B}_{\hat{V}}([u],[v]) = \Lambda_{f,g}([v]) = \int_\Omega f(x)v(x)\, \mathrm{d}\lambda(x) + \int_\Gamma g(y) v(y) \, \mathrm{d}\lambda(y)\]
where the compatibility condition is again used in step 3. Thus, each $u \in [u]$ is a weak solution of the problem considered. Moreover, we have 
\begin{align*}
    \norm{u-P(u)}_{V}^2 &=\norm{[u]}_{\hat{V}}^2 \leq (C+2) \, \mathcal{B}_{\hat{V}} ([u],[u]) = (C+2) \, \Lambda_{f,g}([u]) \\
    &= (C+2) \, \Lambda_{f,g}([u-P(u)])\leq  (C+2) \, \big( \norm{f}_{L^2(\Omega, \lambda)} + \norm{g}_\Gamma\big)\norm{u-P(u)}_{V}
\end{align*}
which proves the validity of (\ref{regularity_estimate_Neumann}). From this, the continuous dependency of $u$ on $f$ and $g$ is straightforward from that fact that the Neumann data $(f,g)$ maps $\textit{linearly}$ to the weak solution of problem (\ref{NP}).
\end{proof}

We are interested in rendering the weak solution of the nonlocal Neumann problem (\ref{Neumann_problem}) unique. In order to do so, a common strategy is deployed in which the test function space $V$ is substituted by a specific subspace where a volume constraint is realized. In the setting of this article, this subspace is given by
\[(\ker{\mathcal{B}})^\perp = \bigcap_{w \in \ker{\mathcal{B}}}\left\{u \in V: \int_\Omega u(x)w(x) \, \mathrm{d}\lambda(x) = 0\right\}\]
which suggests itself from the fact that 
\[q: (\ker{\mathcal{B}})^\perp \rightarrow V/ \ker{\mathcal{B}}, \quad u \mapsto [u] = \{u + \ker{\mathcal{B}}\}\]
is an isometric isomorphism. Clearly, $(\ker{\mathcal{B}})^\perp$ endowed with the scalar product of $V$ is a Hilbert space as well. Henceforth, an element $u \in (\ker{\mathcal{B}})^\perp$ will be called a weak solution of problem (\ref{Neumann_problem}) on this subspace if
\[\mathcal{B}(u,v) = \int_\Omega f(x)v(x) \, \mathrm{d}\lambda(x) + \int_\Gamma g(y)v(y) \, \omega(y) \, \mathrm{d}\lambda(y) \]
holds for all $v \in (\ker{\mathcal{B}})^\perp$.

\begin{Cor}\label{Corollary_Existence_Weak_Solution_Unique_on_Subspace}
 Let $f \in L^2(\Omega,\lambda)$ and assume that $g:\Gamma \rightarrow \mathbb{R}$ satisfies the continuous functional condition. If the nonlocal Poincaré inequality is satisfied in $V$, a unique weak solution of the Neumann problem (\ref{Neumann_problem}) on $(\ker{\mathcal{B}})^\perp$ exists. Moreover, this solution depends continuously on the Neumann data $(f,g) \in L^2(\Omega,\lambda) \times L^2(\Gamma, \omega \cdot \lambda)$ given that there is constant $C>0$ independent of $f$ and $g$ such that
\begin{equation}\label{reg_estimate_Neumann2}
    \norm{u}_{V} \leq C \big( \norm{f}_{L^2(\Omega,\lambda)} + \norm{g}_\Gamma \big).
\end{equation}
If $f$ and $g$ furthermore satisfy the compatibility condition (\ref{compatibility}), each element in $\{u+w: w \in \ker{\mathcal{B}}\}$ is a weak solution of problem (\ref{Neumann_problem}). 
\end{Cor}

\begin{proof}
By \Cref{Equivalences_Poincaré}, the validity of the nonlocal Poincaré inequality in $V$ is equivalent to the existence of a constant $C>0$ such that 
\[\int_\Omega u(x)^2\, \mathrm{d}\lambda(x) \leq C \, \mathcal{B}(u,u) \leq C \int_\Omega \int_{\mathbb{R}^d} (u(x)-u(y))^2 \, K(x,\mathrm{d}y)\, \mathrm{d}\lambda(x)\]
holds for all $u \in (\ker{\mathcal{B}})^\perp$. Consequently, it is easy to see that $\mathcal{B}$ is not only a bounded but also an elliptic bilinear form on $(\ker{\mathcal{B}})^\perp$. Since
\[\Lambda_{f,g}: (\ker{\mathcal{B}})^\perp \rightarrow \mathbb{R}, \ u \mapsto \int_\Omega f(x)u(x)\, \mathrm{d}\lambda(x) + \int_\Gamma g(y) v(y) \, \omega(y) \, \mathrm{d}\lambda(y)\]
is  a linear and continuous functional, the existence of a unique weak solution of problem (\ref{NP}) on $(\ker{\mathcal{B}})^\perp$ is provided by the Lax-Milgram theorem. The validity of (\ref{reg_estimate_Neumann2}) follows from
\begin{align*}
    \norm{u}^2_{V} \leq C \mathcal{B}(u,u) = C \Lambda_{f,g}(u) \leq C \big(\norm{f}_{L^2(\Omega,\lambda)} + \norm{g}_{L^2(\Gamma, \omega \cdot \lambda)}\big)\norm{u}_{V}
\end{align*}
and also proves that $u$ depends continuously on the input data. Finally, let $v \in V$ be given. Because $V = \ker{\mathcal{B}} \oplus (\ker{\mathcal{B}})^\perp$, we have $v = q+z$ for $q \in \ker{\mathcal{B}}$ and $z \in (\ker{\mathcal{B}})^\perp$. Hence, if $w \in \ker{\mathcal{B}}$ is fixed and $f \in L^2(\Omega,\lambda)$ and $g \in L^2(\Gamma, \omega \cdot \lambda)$ satisfy the compatibility condition from \Cref{Well_posedness_Neumann}, it follows that
\begin{align*}
    \mathcal{B}(u+w,v) = \mathcal{B}(u+w, q+z) &= \mathcal{B}(u,z) \\
    &= \int_\Omega f(x)z(x) \, \mathrm{d}\lambda(x) + \int_\Gamma g(y)z(y) \, \omega(y) \, \mathrm{d}\lambda(y) \\
    &= \int_\Omega f(x)v(x) \, \mathrm{d}\lambda(x) + \int_\Gamma g(y)v(y) \, \omega(y) \, \mathrm{d}\lambda(y)
\end{align*}
which proves that $u +w \in V$ is a weak solution of problem (\ref{Neumann_problem}).
\end{proof}

As in \Cref{Sec_Dirichlet_Problem}, we want to provide sufficient conditions for a mapping $g: \Gamma \rightarrow \mathbb{R}$ to satisfy the continuous functional condition. Following \cite[Corollary 5.3]{Frerick2022TheNN}, we obtain:

\begin{Rem}\label{Rem_g_Neumann_problem}
    Let the assumptions of either \Cref{Trace_Theorem_1} or \Cref{Trace_Theorem_2} be met and set the density function $\omega: \Gamma \rightarrow (0,\infty]$ accordingly. Then, the trace operator
    \[\operatorname{Tr}: V \rightarrow L^2(\Gamma, \omega \cdot \lambda), \quad v \mapsto v\einschraenkung_\Gamma\]
    is linear and continuous and, as can be verified by the Hölder inequality, the \linebreak
    \newpage
    \noindent
    continuity of the linear functional $v \mapsto \int_\Gamma g(y)v(y) \, \mathrm{d}\lambda(y)$ is observed if either
    \begin{itemize}
        \item \quad $\displaystyle{\norm{\frac{g}{\sqrt{\omega}}}^2_{L^2(\Gamma,\lambda)} = \int_\Gamma \frac{g(y)^2}{\omega(y)} \, \mathrm{d}\lambda(y) < \infty}$ or
        \item \quad $ \essinf_{y \in \Gamma}\omega(y) > 0 \text{ and } g \in L^2(\Gamma)$.
    \end{itemize}
\end{Rem}

\section{A Weak Maximum Principle}\label{Sec_Maximum}

In \Cref{Sec_Dirichlet_Problem} and \Cref{Sec_Neumann_Problem}, we have discussed the well-posedness of the nonlocal problems (\ref{DP}) and (\ref{NP}) and stated sufficient criterions for a weak solution to exist. However, aside from the mere existence of such a solution, nothing is known about it so far. Throughout this section, we prove a weak maximum principle for the nonlocal operator $\mathcal{L}$ which allows to derive a first quantitative statement about the weak solution $u$. For analogous results in the specific context of the nonlocal diffusion operator $\mathcal{L}_\gamma$, we refer to \cite[Theorem 4.1]{felsingerDirichlet}, \cite[Theorem 5.1]{MR3738190}, \cite[Proposition 2.3]{jarohs2019strong} see also the discussion in \cite{ros2016nonlocal}).

\begin{lemma}\label{Lemma_Positivteil_Negativteil}
   Let $u\in V$ and $u^+:=\max\{u,0\}$ and $u^- := \max\{-u,0\}$. Then:
   \begin{enumerate}
       \item[\textnormal{(i)}] We have $u^+, u^- \in V$ and $\mathcal{B}_\Omega(u^\pm,u^\pm) \leq \mathcal{B}_\Omega(u,u)$ holds.
       \item[\textnormal{(ii)}] It is valid that $\mathcal{B}_\Omega(u,u^+) \geq \mathcal{B}_\Omega(u^+,u^+)$ and $\mathcal{B}_\Omega(u^-,u^-) \leq -\mathcal{B}_\Omega(u,u^-)$.
   \end{enumerate}
\end{lemma}
\begin{proof} The proof mimics the one of \cite[Lemma 3.2]{jarohs2019strong}. At first, it is clear that $u^+, u^- \in L^2(\Omega, \lambda)$. Moreover, we have 
\[\mathcal{B}(u,u) = \mathcal{B}(u^+-u^-,u^+-u^-) = \mathcal{B}(u^+,u^+) + \mathcal{B}(u^-,u^-) -2 \mathcal{B}(u^+,u^-)\]
where the last summand is non-negative since
    \begin{align*}\tag{8} \label{(8)}
        -2\Big(u^+(x)-u^+(y)\Big)\Big(u^-(x) - u^-(y)\Big) &= 2\Big(u^+(x)u^-(y)+u^-(x)u^+(y)\Big) \geq 0
    \end{align*}
    holds for all $x,y \in \mathbb{R}^d$. Consequently, it follows that $\mathcal{B}_\Omega(u,u) \geq \mathcal{B}_\Omega(u^\pm,u^\pm)$ which yields (i). To also obtain (ii) note that from (\ref{(8)}), it is evident that $\mathcal{B}_\Omega(u^+,u^-) \leq 0$. Hence,
    \begin{align*}
        \mathcal{B}_\Omega(u^+,u^+) &= \mathcal{B}_\Omega(u,u^+) + \mathcal{B}_\Omega(u^-,u^+) \leq \mathcal{B}_\Omega(u,u^+), \\
    -\mathcal{B}_\Omega(u^-,u^-) &= \mathcal{B}_\Omega(-u^-,u^-)= \mathcal{B}_\Omega(u,u^-) - \mathcal{B}_\Omega(u^+,u^-) \geq \mathcal{B}_\Omega(u,u^-),
    \end{align*}
    which gives the assertion.
\end{proof} 

Our main result is the following. Note that in comparison to \cite[Theorem 4.1]{felsingerDirichlet} and \cite[Proposition 4.1]{ros2016nonlocal}), the requirement on the validity of the nonlocal Friedrichs inequality in $V_0$ is comparably weak.

\begin{theorem}\label{Max_Principle}
    Let the nonlocal Friedrichs inequality be satisfied in $V_0$. If ${u \in V}$ is such that
    \begin{equation}\label{Vor_Max}
        \mathcal{B}(u,v) \leq 0 \quad \text{ for all } v \in V, \ v \geq 0,
    \end{equation}
    then $\esssup_{x \in \Omega} u(x) \leq \esssup_{x \in \Gamma} u(x)$.
\end{theorem}

\textbf{Remarks:} 
\begin{itemize}
    \item If applying the above theorem to $-u$ instead of $u$, we obtain that $\mathcal{B}(-u,v) \geq 0$ for all $v \in V_0$, $v \geq 0$, implies that $\essinf_{x \in \Omega} u(x) \geq \essinf_{x \in \Gamma} -u(x)$.
    \item If $u \in V$ is sufficiently regular in the sense of (\ref{(5)}), we have for all $v \in V_0$, $v \geq 0$ that $\int_\Omega \mathcal{L}u(x)v(x) \, \mathrm{d}\lambda(x) =\mathcal{B}(u,v) \leq 0$. Accordingly, assumption (\ref{Vor_Max}) can be interpreted as $\mathcal{L}$ being \textit{subharmonic} in the weak sense, i.\,e., with respect to $V_0$.
\end{itemize}

\begin{proof}
     Set $a:= \esssup_{x \in \Gamma} u(x)$ and assume w.l.o.g. that $a < \infty$. Then, \linebreak ${(u-a\chi_{\mathbb{R}^d})^+ \in V_0}$ holds. Moreover, by (\ref{Vor_Max}) and \Cref{Lemma_Positivteil_Negativteil}, we obtain
     \begin{align*}
         0 & \geq \mathcal{B}\left(u, (u-a \chi_{\mathbb{R}^d})^+\right) \geq \mathcal{B}\left((u-a \chi_{\mathbb{R}^d})^+, (u-a \chi_{\mathbb{R}^d})^+\right) \geq 0
     \end{align*}
     which yields $(u-a\chi_{\mathbb{R}^d})^+ \in V_0(\Omega)\cap \ker{\mathcal{B}}$. Because, by assumption, the nonlocal Friedrichs inequality holds in $V_0(\Omega)$, it follows that
       \[\int_\Omega \big((u-a\chi_{\mathbb{R}^d})^+(x)\big)^2 \, \mathrm{d}\lambda(x) \leq C \mathcal{B}\big((u-a\chi_{\mathbb{R}^d})^+,(u-a\chi_{\mathbb{R}^d})^+\big) = 0\]
    and, consequently, $u(x) \leq a$ is valid for $\lambda$-a.e. $x \in \Omega$.    
\end{proof}

\begin{Rem}\label{Rem_MP}
    Let $0 \leq c \in L^\infty(\Omega,\lambda)$. If the bilinear form $\mathcal{B}(u,v)$ in \Cref{Max_Principle} is replaced by
    \[B_c(u,v):= \mathcal{B}(u,v) + \int_\Omega c(x)u(x)v(x) \, \mathrm{d}\lambda(x)\]
    which is the one associated to the regularized operator $\mathcal{L}_c u(x) := \mathcal{L}u(x) + cu(x)$, the statement in (\ref{Vor_Max}) has to be weakened in the sense that $\mathcal{B}_c(u,v) \leq 0$ only gives $\esssup_{x \in \Omega} u(x) \leq \esssup_{x \in \Gamma} u^+(x)$. For a proof, we refer to \cite[Theorem 9.1]{Huschens}.
\end{Rem}

Because the weak maximum principle above is tied to the nonlocal operator $\mathcal{L}$ only, we obtain the following corollary for weak solutions of the nonlocal problems (\ref{DP}) and (\ref{NP}). Note that by \cite[Corollary 9.1.1]{Huschens} an analogous result can also be obtained for nonlocal Robin problems.

\begin{Cor}
    Let $u \in V$ be a weak solution of the nonlocal problem (\ref{DP}) (or (\ref{NP}). If $f \leq 0$ in $\Omega$ and the nonlocal Friedrichs inequality holds in $V_0$, then $\esssup_{x \in \Omega} u(x) \leq \esssup_{x \in \Gamma} u(x)$.
\end{Cor}

\begin{proof}
    Immediate because 
    \[0 \geq \int_\Omega f(x)v(x) \, \mathrm{d}\lambda(x) = \int_\Omega f(x)v(x) \, \mathrm{d}\lambda(x) + \int_\Gamma g(y)v(y) \, \mathrm{d}\lambda(y) = \mathcal{B}(u,v) \]
    is valid for all $v \in V_0$, $v \geq 0$.
\end{proof}

\section{Example \textendash{} The Discrete Poisson Problem}\label{Sec_Applications}

Nonlocal boundary value problems play a role not only in the context of the nonlocal diffusion operator $\mathcal{L}_\gamma$ but also in other branches of mathematics \textendash{} although they are not labeled as such. In this section, we give a demonstration of this fact by aligning our theory with the numerical analysis of the discrete Poisson problem whose solvability is well-established, e.g. see \cite[Chapter 4]{Hackbusch}. Our setting is the following:\\

\newpage

Let $f: \mathbb{R}^d \rightarrow \mathbb{R}$ be Borel measurable and let us consider the Possoin problem over the $d$-dimensional unit cube $\Omega=(0,1)^d$ (for arbitray bounded domains, see \cite[Section 8.1.1]{Huschens}),
\\
\begin{minipage}{0.22\textwidth}
    \[ \hspace{4.3cm} -\Delta u=f \text{ in } \Omega,\]
\end{minipage}
\begin{minipage}{0.78\textwidth}
    \begin{align}
    u&=0 \text{ in } \partial \Omega, \tag{PD} \label{PP1}\\
    \partial_\eta u &=0\text{ in } \partial \Omega, \tag{PN} \label{PP2}
    \end{align}
\end{minipage} \\
\\
which is referred to as problem (\ref{PP1}) if it is equipped with the homogeneous Dirichlet boundary or as problem (\ref{PP2}) if it is equipped with the homogeneous Neumann boundary ($\partial_\eta$ = normal derivative along the outward normal unit vector $\eta$ at $\partial \Omega$). In order to solve both problems numerically on a computer, one needs to discretize. In what follows, we turn to an equidistant grid $\Omega_h$ with grid size $h>0$ small and assume that an ordering (not necessarily driven by numerical considerations) of the grid points in $\Omega$ and $\partial \Omega \setminus E$, $E=$ set of vertices of $[0,1]^d$ in $\Omega$, has been obtained, i.\,e.,
\[\Omega_h:=\Omega \cap G = \left\{x_1,....,x_m\right\}, \quad (\partial \Omega)_h :=(\partial \Omega\setminus E)\cap G= \left\{x_{m+1},...,x_{m+l}\right\},\]
where $m,l \in \mathbb{N}$ and $n:= m+l$, see \Cref{fig:enter-label} for an illustration. Over this grid, the function $u$ is replaced by the discrete function $u_h: \Omega_h \rightarrow \mathbb{R}$ with $u_i := u(x_i)$, $i=1,...,n$. Note that the blue grid nodes are exactly the ``ghost points" of $\Omega$ in $\mathbb{R}^d$. 

\begin{figure}[H]
    \centering
\begin{tikzpicture}[scale=0.89]
        \draw[thick, black,->] (-6.75,0) -- (-3.25,0);
        \draw[thick, black, ->](-6,-0.75) -- (-6,2.75);
        \draw[thick, blue] (-6,0) rectangle (-4,2);
        \node[right,blue] at (-3.9,2) {$\Omega$};
        \node[below] at (-4,0) {$1$};
        \node[left] at (-6,2) {$1$};

        \draw[->] (-3, 1) -- (-1,1);

        \draw[thick, black,->] (-0.75,0) -- (2.75,0);
        \draw[thick, black, ->](0,-0.75) -- (0,2.75);
        \draw[thick, blue] (0,0) rectangle (2,2);
        \node[below] at (2,0) {$1$};
        \node[left] at (0,2) {$1$};
        \draw[black, opacity =0.25] (-0.5,-0.75) -- (-0.5,2.75);
        \draw[black, opacity =0.25] (0,-0.75) -- (0,2.75);
        \draw[black, opacity =0.25] (0.5,-0.75) -- (0.5,2.75);
        \draw[black, opacity =0.25] (1,-0.75) -- (1,2.75);
        \draw[black, opacity =0.25] (1.5,-0.75) -- (1.5,2.75);
        \draw[black, opacity =0.25] (2,-0.75) -- (2,2.75);
        \draw[black, opacity =0.25] (2.5,-0.75) -- (2.5,2.75);
        \draw[black, opacity =0.25] (-0.75,-0.5) -- (2.75,-0.5);
        \draw[black, opacity =0.25] (-0.75,0.5) -- (2.75,0.5);
        \draw[black, opacity =0.25] (-0.75,1) -- (2.75,1);
        \draw[black, opacity =0.25] (-0.75,1.5) -- (2.75,1.5);
        \draw[black, opacity =0.25] (-0.75,2) -- (2.75,2);
        \draw[black, opacity =0.25] (-0.75,2.5) -- (2.75,2.5);
        \node[circle,fill=blue,inner sep=2pt,minimum size=3pt] at (0,0.5){};
        \node[circle,fill=blue,inner sep=2pt,minimum size=3pt] at (0,1){};
        \node[circle,fill=blue,inner sep=2pt,minimum size=3pt] at (0,1.5){};
        \node[circle,fill=blue,inner sep=2pt,minimum size=3pt] at (2,0.5){};
        \node[circle,fill=blue,inner sep=2pt,minimum size=3pt] at (2,1){};
        \node[circle,fill=blue,inner sep=2pt,minimum size=3pt] at (2,1.5){};
        \node[circle,fill=blue,inner sep=2pt,minimum size=3pt] at (0.5,0){};
        \node[circle,fill=blue,inner sep=2pt,minimum size=3pt] at (1,0){};
        \node[circle,fill=blue,inner sep=2pt,minimum size=3pt] at (1.5,0){};
        \node[circle,fill=blue,inner sep=2pt,minimum size=3pt] at (0.5,2){};
        \node[circle,fill=blue,inner sep=2pt,minimum size=3pt] at (1,2){};
        \node[circle,fill=blue,inner sep=2pt,minimum size=3pt] at (1.5,2){};
        \node[circle,fill=red,inner sep=2pt,minimum size=3pt] at (0.5,0.5){};
        \node[circle,fill=red,inner sep=2pt,minimum size=3pt] at (1,0.5){};
        \node[circle,fill=red,inner sep=2pt,minimum size=3pt] at (1.5,0.5){};
        \node[circle,fill=red,inner sep=2pt,minimum size=3pt] at (0.5,1){};
        \node[circle,fill=red,inner sep=2pt,minimum size=3pt] at (1,1){};
        \node[circle,fill=red,inner sep=2pt,minimum size=3pt] at (1.5,1){};
        \node[circle,fill=red,inner sep=2pt,minimum size=3pt] at (0.5,1.5){};
        \node[circle,fill=red,inner sep=2pt,minimum size=3pt] at (1,1.5){};
        \node[circle,fill=red,inner sep=2pt,minimum size=3pt] at (1.5,1.5){};

        \node[circle,fill=blue,inner sep=2pt,minimum size=3pt] at (3.5,1.75){};
        \node[black, right] at (3.8,1.75) {Grid node in $(\partial\Omega)_h$}{};
        \node[circle,fill=red,inner sep=2pt,minimum size=3pt] at (3.5,1){};
        \node[black, right] at (3.8,1) {Grid node in $\Omega_h$}{};
        \node[circle, fill =blue, opacity =0.2, inner sep=2pt,minimum size=3pt] at (3.5,0.25){};
        \node[black, right] at (3.8,0.25) {Nonlocal boundary $\Gamma$}{};

        \draw[<->] (2.6,2) --(2.6,2.5);
        \draw[<->] (2,2.6) -- (2.5,2.6);
        \node[right] at (2.6,2.25) {$h$};
        \node[above] at (2.25,2.6) {$h$};

        \draw[fill = blue, opacity = 0.1] (-0.5,0) rectangle (0,2);
        \draw[fill = blue, opacity = 0.1] (0,2) rectangle (2,2.5);
        \draw[fill = blue, opacity = 0.1] (0,0) rectangle (2,-0.5);
        \draw[fill = blue, opacity = 0.1] (2,0) rectangle (2.5,2);
\end{tikzpicture}
    \caption{Illustration of the discretization of the unit cube $\overline{\Omega}= [0,1]^2$ in $d=2$. Note that the grid nodes in $(\partial \Omega)_h$ are contained in the nonlocal boundary $\Gamma$.}
    \label{fig:enter-label}
\end{figure}
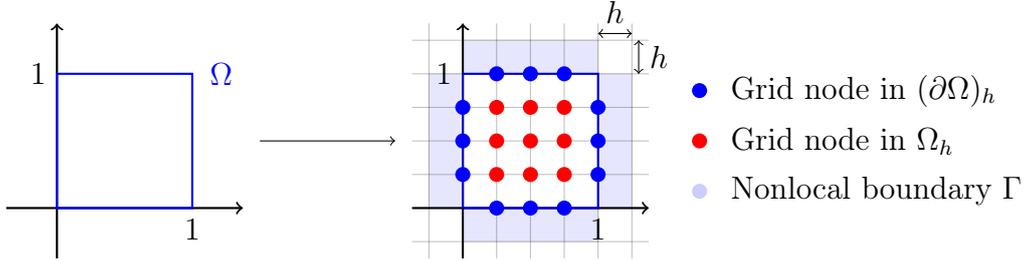

Via finite differences, the Laplacian $-\Delta u$ in $x \in \mathbb{R}^d$ can be rewritten as
\[-\Delta u(x) = \frac{1}{h^2}\left( 2d \, u(x) - \sum_{i=1}^d \left(u(x+he_i) - u(x-he_i) \right)\right) + \mathcal{O}(h^2) = \mathcal{L}_{d,h}u(x) + \mathcal{O}(h^2)\]
such that by dropping the $\mathcal{O}(h^2)$-term, a discrete approximation of $-\Delta u$ is obtained via the $(2d+1)$-point stencil operator $\mathcal{L}_{d,h}$. Likewise, a finite-difference discretization of the boundary condition in $\partial (\Omega)_h$ is realized by 
\begin{align}
    u&=0 \text{ in } (\partial \Omega)_h, \tag{Dirichlet boundary}\\
     \partial^-_\eta u := \frac{1}{h}\left(u(\cdot)-u(\cdot-\eta h)\right) &=0 \text{ in } (\partial \Omega)_h, \tag{Neumann boundary}
\end{align} 
where for $y \in (\partial \Omega)_h$ the backwards difference $\partial^-_\eta u(y)$, is a first order approximation of the normal derivative $\partial_\eta u(y)$. Because for all for all $y \in (\partial \Omega)_h$, it is valid that
\begin{align*}
    \mathcal{N}_{d,h}u(y) &:= \int_\Omega \big(u(y)-u(y)\, K_{d,h}(y, \mathrm{d}x) \notag\\
    & \ =  \frac{1}{h^2} \sum_{i=1}^d \chi_{\Omega}(y+he_i)\big(u(y)-u(y+he_i)\big) + \chi_\Omega(y-he_i)\big(u(y)-u(y-he_i)\big) \\
    &= \frac{1}{h} \partial_\eta^-u(y)
\end{align*}
the discrete Poisson problems (\ref{PP1}) and (\ref{PP2}) can be formulated the following way:
\begin{minipage}{0.4\textwidth}
    \[ \hspace{0.2cm} \mathcal{L}_{d,h}u(x_i) =f(x_i), \quad i=1,...,m \]
\end{minipage}
\begin{minipage}{0.6\textwidth}
    \begin{align}
    u(x_i)&=0, \quad i=m+1,...,m+l ,\tag{$\text{PD}_{\text{discr}}$}\label{DPP1}\\
    \mathcal{N}_{d,h}u(x_i)&=0, \quad   i=m+1,...,m+l. \tag{$\text{PN}_{\text{discr}}$}\label{DPP2}
    \end{align}
\end{minipage} \\
\\
\textnormal{In matrix-vector notation, the above problems are equivalent to the linear system of equations
\[\mathbf{Au} = \begin{bmatrix}
    \mathbf{f}\\
    \mathbf{0}
\end{bmatrix}\]
where $\mathbf{u} = \big(u(x_1),....,u(x_n)\big)^T \in \mathbb{R}^n$, $\mathbf{f}= \big(f(x_1),...,f(x_m)\big)^T \in \mathbb{R}^m$, and ${\mathbf{0} \in \mathbb{R}^{l}}$. Here, the stiffness matrices $\mathbf{A} = \mathbf{A}^\text{D} \in \mathbb{R}^{n \times n}$ for the Dirichlet problem and \linebreak ${\mathbf{A} = \mathbf{A}^\text{N} \in \mathbb{R}^{n \times n}}$ for the Neumann problem are given by
\begin{equation}\label{stiffness}
    \mathbf{A}^\text{D}:= \frac{1}{h^2} \begin{bmatrix}
        A^\Omega & A^\Gamma \\
        0 & I
    \end{bmatrix}, \quad \mathbf{A}^\text{N}:= \frac{1}{h^2} \begin{bmatrix}
        A^\Omega & A^\Gamma \\
        (A^\Gamma)^T & I
    \end{bmatrix}
\end{equation}
where $I \in \mathbb{R}^{l \times l}$ is the identity matrix, and the entries of the block matrices \linebreak $A^\Omega = (a^\Omega)_{j,k} \in \mathbb{R}^{m \times m}$ and $A^\Gamma = (a^\Gamma)_{j,k}\in \mathbb{R}^{m \times l}$ are set as follows:
\begin{align*}
    (a^\Omega)_{j,k} &= \left\{\begin{array}{rl}
    2d & \text{if } j=k\\
    -1 & \text{if $x_k$ is adjacent to $x_j$ (i.\,e., $x_k \pm he_i = x_j$ for $i \in \{1,...,n\})$}\\
    0 & \text{else}
\end{array}\right\}, \\
(a^\Gamma)_{j,k} &= \left\{ \begin{array}{rl}
    -1 & \text{if $x_{k+m}$ is adjacent to $x_j$}\\
    0 & \text{else}
\end{array}\right\}.
\end{align*}
By definition, $\mathbf{A}^\Omega \in \mathbb{R}^{m \times m}$ and $\mathbf{A}^\text{N} \in\mathbb{R}^{n \times n} $ are both symmetric matrices.}\\

We emphasize that although the problems (\ref{PP1}) and (\ref{PP2}) are clearly local, their finite difference discretizations (\ref{DPP1}) and (\ref{DPP2}) are both nonlocal problems. As will become evident in \Cref{lemma_Stencil_equivalence_poblems} below, a Borel measurable function $u: \mathbb{R}^d \rightarrow \mathbb{R}$ is a solution of problem (\ref{DPP1}) respectively (\ref{DPP2}) if and only if it is a weak solution of 
\\
\begin{minipage}{0.22\textwidth}
    \[ \hspace{4.3cm} \mathcal{L}_{d,h}u=f \text{ in } \Omega,\]
\end{minipage}
\begin{minipage}{0.78\textwidth}
    \begin{align}
    u&=0 \text{ in } \Gamma, \tag{BPD} \label{BP1}\\
    \mathcal{N}_{d,h}u &=0 \text{ in } \Gamma, \tag{BPN} \label{BP2}
    \end{align}
\end{minipage} \\
\\
where the measure $\lambda= \mu$ for the weak formulation is the following:

\begin{dar}
    Let $G:= h \mathbb{Z}^d$ be the uniform grid which divides $\mathbb{R}^d$ into cubes of length $h$ and has a node at $0 \in \mathbb{R}^d$. We define the measure ${\mu: \mathcal{B}(\mathbb{R}^d) \rightarrow [0,\infty]}$ by
    \begin{equation*}
        \mu:= \sum_{x \in G} \delta_x.
    \end{equation*}
    Then, $\mu$ is a $\sigma$-finite measure and for $A \in \mathcal{B}(\mathbb{R}^d)$, $\mu(A)$ displays the number of grid in $A$. By \Cref{Bsp_Self_adjointness_Stencil}, the transition kernel $K_{d,h} \in \mathcal{K}$ is  $\mu$-symmetric.
\end{dar}

Recall that an element $u \in V(\Omega,K_{d,h},\mu)$ is called a weak solution of problem (\ref{BP1})/(\ref{BP2}) if 
\[\begin{tabular}{cll}
   (*)  & (\ref{BP1}): &  $\displaystyle{\mathcal{B}(u,v) = \int_\Omega f(x)v(x) \, \mathrm{d}\mu(x)}$ holds for all $v \in V_0(\Omega,K_{d,h},\mu)$ \\
     & & \hspace{1.5cm} and $u \in V_0(\Omega,K_{d,h},\mu)$ \\
     (**) &  (\ref{BP2}): &  $\displaystyle{\mathcal{B}(u,v) = \int_\Omega f(x)v(x) \, \mathrm{d}\mu(x)}$ holds for all $v \in V(\Omega,K_{d,h},\mu)$
\end{tabular}\] 
where $\mathcal{B}: V(\Omega,K_{d,h},\mu) \times V(\Omega,K_{d,h},\mu) \rightarrow \mathbb{R} $ is given as in (\ref{definition_B}), i.\,e.,
    \begin{align*}
    \mathcal{B}(u,v) &= \frac{1}{2}\int_\Omega \int_\Omega \big(u(x)-u(x)\big)\big(v(x)-v(y)\big) \, K_{d,h}(x, \mathrm{d}y) \, \mathrm{d}\mu(x)\\
    &\hspace{3cm} + \int_\Omega \int_{\mathbb{R}^d \setminus \Omega} \big(u(x)-u(x)\big)\big(v(x)-v(y)\big) \, K_{d,h}(x, \mathrm{d}y) \, \mathrm{d}\mu(x).
    \end{align*}
We highlight that as outlined below and due to the discrete nature of the measure $\mu$ and the specific choice of $K= K_{d,h} \in \mathcal{K}$, the nonlocal function space $V(\Omega,K_{d,h},\mu)$ can be identified with the vector space $\mathbb{R}^n$ while the bilinearform $\mathcal{B}$ defined thereon can be represented via the matrix $\textbf{A}^\text{N} \in \mathbb{R}^n$:

\begin{dar}\label{identification_funktionenraum}
    Let $u: \mathbb{R}^d \rightarrow \mathbb{R}$ be Borel measurable. Because $\Omega \cap G_s = \{x_1,...,x_m\}$ is a finite set, $u \in V(\Omega,K_{d,h},\mu)$ holds. Let us consider the two vector space morphisms 
    \[\begin{tabular}[h]{lccl}
$F: \mathbb{R}^n \rightarrow V(\Omega,K_{d,h},\mu)$ & & &$\mathbf{v} \mapsto \sum_{i=1}^n \mathbf{v}_i \chi_{\{x_i\}}$, \\
$G: V(\Omega,K_{d,h},\mu) \rightarrow \mathbb{R}^n$ & & &$v \mapsto \big(v(x_1),v(x_2),...,v(x_{n-1}), v(x_n)\big)^\top$.
\end{tabular}\]
Since both $F\circ G = \operatorname{id}_{V(\Omega,K_{d,h},\mu)}$ and $G \circ F = \operatorname{id}_{\mathbb{R}^n}$, the spaces $V(\Omega,K_{d,h},\mu)$ and $\mathbb{R}^n$ are isomorphic; consequently, we can identify $v \in V(\Omega,K_{d,h},\mu)$ with a unique vector $\mathbf{v} \in \mathbb{R}^n$  and vice versa. In the following, we still write $v$ if we refer to the element $v \in V(\Omega,K_{d,h},\mu)$, whereas the associated vector
\[\mathbb{R}^n \ni \mathbf{v}:=\big((v(x_1),v(x_2),...,v(x_{n-1}), v(x_n)\big)^\top\]
will be marked by the bold symbol $\mathbf{v}$. Because for $v \in V_0(\Omega,K_{d,h},\mu)$, it is valid that $\mathbf{v}_i=v(x_i)=0$ for all $i=m+1,...,n$, we similarly identify $V_0(\Omega,K_{d,h},\lambda)$ with the (subspace) $\mathbb{R}^m$.
\end{dar} 

\begin{lemma}\label{Lemma_B_Identifikation}
    For all $u,v \in V(\Omega,K_{d,h},\mu)$, we have $\mathcal{B}(u,v) = \mathbf{v}^T \mathbf{A}^{N} \mathbf{u}$. Especially, $\mathbf{A}^\text{N} \in \mathbb{R}^{n \times n}$ is a positive semi-definite matrix.
\end{lemma}

\begin{proof}
    Let $u,v \in V(\Omega,K_{d,h},\mu)$ be fixed and for $1 \leq i \leq d$, define
    \begin{equation}\label{def_Gamma_i_pm}
    \Gamma_i^+ :=\{y \in \Gamma \cap G: y+ he_i \in \Omega\}, \quad \Gamma_i^- :=\{y \in \Gamma \cap G: y- he_i \in \Omega\},
    \end{equation}
    i.\,e., $\Gamma_i^\pm$ is the set of grid points in $\Gamma$ for which a $\pm h$-step in direction $i$ leads into $\Omega$. Then, for $x \in \mathbb{R}^d$, we have
    \begin{align*}
        \mathcal{L}_{d,h}u(x) &= \frac{1}{h^2} \Big(2d \, u(x) - \sum_{i=1}^d \big(u(x+he_i)+u(x-he_i)\big)\Big),\\
        \mathcal{N}_{d,h}u(x) &= \frac{1}{h^2}\sum_{i=1}^d \Bigg( \chi_{ \Gamma_i^+}(x) \big(u(x) - u(x+he_i)\big) + \chi_{\Gamma_i^-}(x) \big(u(x)-u(x-he_i)\big)\Bigg),
    \end{align*}
    and, consequently,
    \begin{align*}
        \int_\Omega \mathcal{L}_{d,h}u(x) v(x) \, \mathrm{d}\mu(x) &+ \int_\Gamma \mathcal{N}_{d,h}u(y) v(y) \, \mathrm{d}\mu(y)\\
        &= \sum_{j=1}^m \mathcal{L}_{d,h}u(x_j)v(x_j) + \sum_{j=m+1}^{m+l=n} \mathcal{N}_{d,h}u(x_j)v(x_j)\\  
        &= \sum_{j=1}^n (\mathbf{A}^\text{N}\mathbf{u})_j \mathbf{v}_j = \mathbf{v}^T \mathbf{A}^\text{N} \mathbf{u}.
    \end{align*}
   Because $C:= \sup_{x \in \Omega} K_{d,h}(x, \mathbb{R}^d) < \infty$ holds, the Jensen inequality yields for all ${u \in V(\Omega,K_{d,h},\mu_s)}$ that
    \begin{align*}\label{eq_K_dh_IBP}
        &\int_\Omega \Big(\int_{\mathbb{R}^d}|u(x)-u(y)| \, K_{d,h}(x, \mathrm{d}y)\Big)^2 \, \mathrm{d}\mu(x) \\
        &\leq C \int_\Omega \int_{\mathbb{R}^d} \big(u(x)-u(y)\big)^2 \, K_{d,h}(x, \mathrm{d}y) \, \mathrm{d}\mu(x) \\
        & \leq \frac{C}{h^2} \sum_{x \in \Omega \cap G} \sum_{i=1}^d \big(u(x)-u(x+he_i)\big)^2 + \big(u(x)-u(x-he_i)\big)^2 < \infty,
    \end{align*}
    an application of \Cref{Theorem_Nonlocal_Integration_By_Parts_Formula} gives
    \begin{align*}
        \mathcal{B}(u,v) &= \int_\Omega \mathcal{L}_{d,h}u(x) v(x) \, \mathrm{d}\mu(x) + \int_\Gamma \mathcal{N}_{d,h}u(y) v(y) \, \mathrm{d}\mu(y), 
    \end{align*}
    as desired. The positive semi-definiteness is clear from \Cref{identification_funktionenraum} and the fact that $\mathcal{B}(v,v) \geq 0$ holds for all $v \in V(\Omega,K_{d,h},\mu)$.
\end{proof}

Because for $z \in V_0(\Omega,K_{d,h},\mu)$, we have $\mathbf{z}_j =0$ for all $j=m+1,...,m+l=n$, the following corollary arises:

\begin{Cor}\label{Cor_B_Identifikation}
    For $u,v \in V_0(\Omega,K_{d,h},\mu)$ arbitrary, it is valid that 
    \[\mathcal{B}(u,v) = \mathbf{v}^T \left(\frac{1}{h^2}A^\Omega\right) \mathbf{u} = \begin{bmatrix}
        \mathbf{v}\\
        \mathbf{0}
    \end{bmatrix}^T \mathbf{A}^\text{D} \begin{bmatrix}
        \mathbf{u}\\
        \mathbf{0}
    \end{bmatrix}.\] Especially, the matrix $\frac{1}{h^2} A^\Omega \in \mathbb{R}^{m \times m}$ is positive semi-definite as well. 
\end{Cor}

Moreover, because $\Gamma \cap G = (\partial \Omega)_h$, (*) and (**) can be rewritten in: Find $\mathbf{u}\in \mathbb{R}^{n}$ such that
\[\begin{tabular}{cll}
   (*)  & (\ref{BP1}): &  $\begin{bmatrix}
            \mathbf{v}\\
            \mathbf{0}
        \end{bmatrix}^T \mathbf{A}^\text{D}\mathbf{u} = \begin{bmatrix}
            \mathbf{v} & \mathbf{0}
        \end{bmatrix} \begin{bmatrix}
            \mathbf{f} \\
            \mathbf{0}
        \end{bmatrix}$ holds for all $\mathbf{v} \in \mathbb{R}^m$, $\mathbf{u} = \begin{bmatrix}
            \mathbf{\tilde{u}}\\
            \mathbf{0}
        \end{bmatrix}$, $\mathbf{\tilde{u}} \in \mathbb{R}^m$, \\
        &&\\
     (**) &  (\ref{BP2}): &  $\mathbf{v}^T \mathbf{A}^{\text{N}} \mathbf{u}= \mathbf{v}^T \begin{bmatrix}
            \mathbf{f}\\
            \mathbf{0}
        \end{bmatrix}$ holds for all $\mathbf{v} \in \mathbb{R}^n$.
\end{tabular}\]
We have shown:

\begin{lemma}\label{lemma_Stencil_equivalence_poblems}
   Let $\Omega=(0,1)^d$. A Borel measurable $u: \mathbb{R}^d \rightarrow \mathbb{R}$ is a solution of problem (\ref{DPP1}) respectively (\ref{DPP2}) if and only if it is a weak solution of
\\
\begin{minipage}{0.22\textwidth}
    \[ \hspace{4.3cm} \mathcal{L}_{d,h}u=f \text{ in } \Omega,\]
\end{minipage}
\begin{minipage}{0.78\textwidth}
    \begin{align}
    u&=0 \text{ in } \Gamma, \tag{BPD} \label{BP1}\\
    \mathcal{N}_{d,h}u &=0 \text{ in } \Gamma, \tag{BPN} \label{BP2}
    \end{align}
\end{minipage} \\
\\
in $V(\Omega,K_{d,h},\mu)$. 
\end{lemma}

\begin{Rem}{\textnormal{(A Side Trip to Boundary Value Problems on Graphs)}}
    If $x,y \in  h\mathbb{Z}^d$, let us write $x \sim y$ if $x$ and $y$ are adjacent in $G$, i.\,e., if there is $i \in \{1,...,d\}$ such that either $x+he_i =y$ or $x-he_i=y$. Then, the stencil operator $\mathcal{L}_{d,h}$ operating on $h\mathbb{Z}^d$ can be written as
    \begin{equation}
        \mathcal{L}_{d,h}u(x) = \frac{4}{h^2}\left(\frac{1}{4}\sum_{x \sim y} \big(u(x)- u(x-y)\right)
    \end{equation}
    where the expression in the parenthesis is called the discrete Laplacian on the corresponding grid.\\
    Let now $(V, \mu)$ be a finite weighted graph without isolated points where the conductance function $\mu: V \times V \rightarrow [0,\infty)$ is such that both
    \begin{itemize}
        \item $\mu(x,y) := \mu_{x,y} = \mu_{y,x} =: \mu(y,x)$ for all $x,y \in V$;
        \item $\mu_{x,y} >0$ if and only if $x$ and $y$ are connected in $V$ (shortly: $x \sim y$).
    \end{itemize}
    Finally, let $\mu(x):=\sum_{x \sim y} \mu_{x,y}$ be the \textit{degree} of the vertex $x \in V$ and define the weighted Graph Laplacian $(- \Delta)_\mu u (x)$ on $(V, \mu)$ as follows:
    \begin{equation}\label{Graph_Laplacian}
    (- \Delta)_\mu u (x) := \frac{1}{\mu(x)} \sum_{y \sim x} \big(u(x)-u(y)) \, \mu_{xy}
    \end{equation}
    Then, we point out that in the same manner as the discrete Poisson problem with Dirichlet boundary can be considered a nonlocal boundary value problem, the Dirichlet problem
    \begin{equation}\label{Dirichlet_graph}
    (-\Delta)_\mu u = f \text{ on } \emptyset \neq\Omega \subset G, \quad u= 0 \text{ on } \Gamma(\Omega),
    \end{equation}
    $\Gamma(\Omega):=\{ y \sim x \text{ for at least one } x \in G\}$ like it is exemplarily considered in the context of chip firing games or hitting times of discrete Markov processes, can be considered a nonlocal boundary value problem as well. To see this, use an analogous reasoning as above and identify $V$ with a finite set in $\mathbb{R}^d$, let the underlying Borel measure be given by
    \[\lambda := \sum_{x \in V} \mu(x)\, \delta_x\]
    and determine the transition kernel $K$ constituting $(-\Delta)_\mu$ via
    \[K(x,A) := \frac{1}{\mu(x)}\sum_{ y \sim x} \delta_y(A) \, \mu_{x,y}.\]
    Because for arbitrary $A,B \in \mathcal{B}(\mathbb{R}^d)$, we have
    \begin{align*}
        (\lambda \otimes K)(A \times B) &= \int_A K(x,B) \, \mathrm{d}\lambda(x) = \int_A \frac{1}{\mu(x)} \sum_{ y\in B,\, y \sim x} \mu_{x,y} \, \mathrm{d}\lambda(x) = \sum_{x \in A} \sum_{y \in B, \, y \sim x} \mu_{x,y}\\
        &= \sum_{y \in B} \sum_{x \in A, \, x \sim y} \mu_{y,x} = \int_B \frac{1}{\mu(y)} \sum_{ x\in A,\, x \sim y} \mu_{y,x} \, \mathrm{d}\lambda(y) =  \int_B K(y,A) \, \mathrm{d}\lambda(y) \\
        &= (\lambda \otimes K)(B \times A),
    \end{align*}
    this transition kernel is symmetric with respect to the measure $\lambda$. The underlying matrix representation of problem (\ref{Dirichlet_graph}) is given by
    \begin{equation*}
        \mathbf{A}^\text{D} := \begin{bmatrix}
            A^\Omega & A^{\Gamma(\Omega)}\\
            0 & I
        \end{bmatrix} = \begin{bmatrix}
            \mathbf{f}\\
            \mathbf{0}
        \end{bmatrix}
    \end{equation*}
    where the matrices $A^\Omega$ and $A^{\Gamma(\Omega)}$ are set as follows:
    \begin{align*}
    (a^\Omega)_{j,k} = \left\{\begin{array}{rl}
    \mu(j) & \text{if } j=k\\
    -\mu_{jk} & \text{if $x_k \sim x_j$}\\
    0 & \text{else}
\end{array}\right\}, \quad (a^{\Gamma(\Omega)})_{j,k} = \left\{ \begin{array}{rl}
    -\mu_{jk} & \text{if $x_{k+m} \sim x_j$}\\
    0 & \text{else}
\end{array}\right\}.
\end{align*}
Again, $\mathbf{A}^\text{D}$ is a symmetric matrix.
\end{Rem}

Let us outline next how our nonlocal theory from Section \ref{Sec_Dirichlet_Problem} and \ref{Sec_Neumann_Problem} manifests itself in the setting of the discrete Poisson problem. For reasons of clarity, the Dirichlet problem and the Neumann problem are considered separately and one after the other. Subsequently, also the alignment of the weak maximum principle from \Cref{Max_Principle} with the so-called global discrete maximum principle (see \Cref{global_discrete_MP}) is illustrated.

\subsection{The Dirichlet Problem}

For our discussion, the following observations are of key importance:

\begin{lemma}
    The matrices $A^\Omega \in \mathbb{R}^{m \times m}$ and $\mathbf{A}^\text{D} \in \mathbb{R}^{n \times n}$ are both non-singular.
\end{lemma}

\begin{proof}
    With \cite[Definition 4.3.2]{Hackbusch}, it is easily seen that the matrix $A^\Omega \in \mathbb{R}^{m \times m}$ is irreducible. Because it is furthermore (weakly) diagonally dominant and
    \[|(a^\Omega)_{i,i}|  > \sum_{j=1, j \neq i}^m |(a^\Omega)_{i,j}|\]
    holds for at least one index $i \in \{1,...,m\}$ (namely if $(a^\text{D}_{i,j}) \neq 0$ for $j \in \{m+1,...,n\}$), \cite[Corollary 1.22]{varga1962iterative} yields the non-singularity of $A^\Omega \in \mathbb{R}^{m \times m}$. The inverse of the matrix $\mathbf{A}^\text{D}$ can be verified to be given by
     \begin{equation*}
    (\mathbf{A}^\text{D})^{-1} = h^2 \begin{bmatrix}
        (A^\Omega)^{-1} & -\frac{1}{h^2}(A^\Omega)^{-1}A^\Gamma \\
        0 & \frac{1}{h^2} I 
    \end{bmatrix}.
    \end{equation*}
\end{proof}

Because $\mathbf{A}^\text{D} \in \mathbb{R}^{n \times n}$ is non-singular, problem (\ref{DPP1}) has a unique solution which by \Cref{lemma_Stencil_equivalence_poblems} is also the unique weak solution of the nonlocal problem (\ref{BP1}). As remarked below, the invertibility of $\mathbf{A}^{\text{D}}$ is exactly the validity of the nonlocal Friedrichs inequality in $V(\Omega,K_{d,h},\mu)$ which turns out to be not only a sufficient but also necessary criterion for a weak solution of problem (\ref{BP1}) to exist:

\begin{Rem} \label{Rem_Friedrichs_Matrix}
    Recall that the nonlocal Friedrichs inequality in $V(\Omega,K_{d,h},\mu)$ demands: There is a constant $C>0$ such that for all $v \in V_0(\Omega,K_{d,h},\mu)$, it is valid that
    \begin{equation}\label{F_in}
        \int_\Omega v(x)^2 \, \mathrm{d}\mu(x) \leq C \, \mathcal{B}(v,v).
    \end{equation}
    With the identifications from \Cref{identification_funktionenraum} and \Cref{Lemma_B_Identifikation}, the inequality in (\ref{F_in}) can be equivalently rewritten into: Find a constant $C>0$ such that for all $\mathbf{v} \in \mathbb{R}^m$, it is valid that
    \begin{equation}
        \norm{v}^2_{\mathbb{R}^m} \leq C\, \mathbf{v}^T \left(\frac{1}{h^2}A^\Omega\right) \mathbf{v}. 
    \end{equation}
    Especially, the nonlocal Friedrichs inequality is satisfied in $V(\Omega,K_{d,h},\mu)$ if and only if the matrix $A^\Omega \in \mathbb{R}^{n \times n}$ is positive definite. However, because by \Cref{Cor_B_Identifikation}, $A^\Omega$ is positive semi-definite (and symmetric), the positive definiteness is equivalent to the invertibility which itself characterizes the unique solvability of problem (\ref{DPP1}) and, thus, of problem (\ref{BP1}).
\end{Rem}

\subsection{The Neumann Problem}\label{Sec_discret_Neumann}

We again outline some important observations with respect to the matrix \linebreak ${\mathbf{A}^\text{N} \in \mathbb{R}^{n \times n}}$ which are crucial for tackling the discrete Neumann problem (\ref{DPP2}). To this end, note that by \Cref{Lemma_B_Identifikation}, we can decompose the symmetric matrix $\mathbf{A}^\text{N} \in \mathbb{R}^{n \times n}$ into
\begin{equation}\label{eigendecomposition}
    \mathbf{A}^\text{N}= \mathbf{Q \Lambda Q}^\top
\end{equation}
where $\mathbf{\Lambda} = \operatorname{diag}(\lambda_1,...,\lambda_n)\in \mathbb{R}^{n \times n}$ is the diagonal matrix consisting of the real and non-negative eigenvalues of $\mathbf{A}^\text{N}$ and $\mathbf{Q} \in \mathbb{R}^{n \times n}$ is the orthogonal matrix whose columns are the real and orthonormal eigenvectors $\mathbf{q}_1,...,\mathbf{q}_{n}$ of $\mathbf{A}^\text{N} \in \mathbb{R}^{n \times n}$. The following statements hold:

\begin{lemma}\label{Lemma_A_N} \hfill
    \begin{enumerate}
        \item[(1.)] The kernel of $\mathbf{A}^{\text{N}}$ is the one-dimensional subspace of $\mathbb{R}^n$ which consists of all constant vectors, i.\,e., \[\ker{\mathbf{A}^\text{N}} =  \{\mathbf{v} \in \mathbb{R}^n: \mathbf{v} = c \mathbf{1}_n, c \in \mathbb{R}\}.\]
        \item[(2.)] The matrix $\mathbf{A}^\text{N}$ is positive definite on $(\ker{\mathbf{A}^\text{N})^{\perp_{\mathbb{R}^n}}}= \{\mathbf{v}\in \mathbb{R}^n: \sum_{i=1}^n \mathbf{v}_i =0\}$.
    \end{enumerate}
\end{lemma}

\begin{proof}
    \hfill
\begin{enumerate}
    \item[(1.)] Easy.
    \item[(2.)] We can w.l.o.g. assume that $\lambda_1=0$ and $\mathbf{q}_1 = \frac{1}{\sqrt{n}}(1,1,....,1,1)^T\in \mathbb{R}^n$. Then, $\mathbf{A}^\text{N} \in \mathbb{R}^{n \times n}$ is positive definite on $\operatorname{span}\{\mathbf{q}_2,...,\mathbf{q}_n\}$ because for $\mathbf{v} = \sum_{i=2}^{n} \rho_i \mathbf{q}_i$, we have
\begin{equation}\label{eq_Poincare3}
\begin{aligned}
\mathbf{v}^\top \mathbf{A}^\text{N} \mathbf{v} = \langle \mathbf{v}, \mathbf{A}^\text{N}\mathbf{v}\rangle_{\mathbb{R}^{n}} &= \Bigg\langle \sum_{i=2}^{n} \rho_i \mathbf{q}_i, \sum_{j=2}^{n} \lambda_j \rho_j \mathbf{q}_j \Bigg\rangle_{\mathbb{R}^{n}} \\
&= \sum_{i=2}^{n} \lambda_i \rho_i^2 \geq \lambda_{2,...,n}^{\text{min}} \sum_{i=2}^{n} \rho_i^2 =  \lambda_{2,...,n}^{\text{min}} \norm{\mathbf{v}}^2_{\mathbb{R}^{n}}
\end{aligned}
\end{equation}
where $\lambda_{2,...,n}^{\text{min}}:= \min\{\lambda_2,...,\lambda_n\}>0$. To conclude the proof, we show that $\operatorname{span}\{\mathbf{q}_2,...,\mathbf{q}_n\} = \{\mathbf{v}\in \mathbb{R}^n: \sum_{i=1}^n \mathbf{v}_i =0\}$: Let first $\mathbf{v}\in \mathbb{R}^n $ with $\sum_{i=1}^n \mathbf{v}_i =0$ be given. Because the set $\{\mathbf{q}_1,...,\mathbf{q}_n\}$ emerging from (\ref{eigendecomposition}) forms a basis of $\mathbb{R}^n$, we can write 
    \[\mathbf{v} = \sum_{i=1}^n \iota_i \mathbf{q}_i\]
    for $\iota_i \in \mathbb{R}$, $i=1,...,n$; furthermore, since for all $\mathbf{w} \in \ker{\mathbf{A}_\text{N}} \setminus \{0\}$, it is valid that $\mathbf{w} = \kappa \mathbf{q}_1 $ for a number $\kappa \neq 0$, the orthonormality of $\{\mathbf{q}_1,...,\mathbf{q}_n\}$ yields
    \begin{align*}
        0 &= \frac{\kappa}{\sqrt{n}}\sum_{i=1}^n \mathbf{v}_i = \sum_{i=1}^n \mathbf{v}_i \mathbf{w}_i  =\langle \mathbf{v}, \mathbf{w}\rangle_{\mathbb{R}^n} = \left\langle \sum_{i=1}^n \iota_i \mathbf{q}_i, \kappa\mathbf{q}_1\right\rangle_{\mathbb{R}^n} = \iota_1\kappa. 
    \end{align*}
    Especially, $\iota_1 = 0$ and $\mathbf{v} \in \operatorname{span}\{\mathbf{q}_2,...,\mathbf{q}_n\}$, as desired. For the other implication, let now $\mathbf{v} \in \operatorname{span}\{\mathbf{q}_2,...,\mathbf{q}_n\}$ be arbitrary, $\mathbf{v} = \sum_{i=2}^n \iota_i \mathbf{q}_i$. Then, for $\mathbf{w} = (1,1,...,1,1)^T \in \mathbb{R}^n$, we have $\mathbf{w} \in \ker{\mathbf{A}^\text{N}}$ such that the orthonormality of $\{\mathbf{q}_1,...,\mathbf{q}_n\}$ again implies
    \begin{align*}
        0 = \left\langle \sum_{i=2}^n \iota_i \mathbf{q}_i, \sqrt{n} \mathbf{q}_1\right\rangle_{\mathbb{R}^n} = \langle \mathbf{v}, \mathbf{w}\rangle_{\mathbb{R}^n} = \sum_{i=1}^n \mathbf{v}_i \mathbf{w}_i = \sum_{i=1}^n \mathbf{v}_i
    \end{align*}
    and, hence, $\sum_{i=1}^n \mathbf{v}_i =0$.
\end{enumerate}
\end{proof}

The positive definiteness of $\mathbf{A}^\text{N}$ on $(\ker{\mathbf{A}^\text{N})^{\perp_{\mathbb{R}^n}}}= \{\mathbf{v}\in \mathbb{R}^n: \sum_{i=1}^n \mathbf{v}_i =0\}$ yields that $\mathbf{A}^\text{N} \in \mathbb{R}^{n \times n}$ has closed range in $\mathbb{R}^n$, see \cite[Proposition 5.30]{hunterbook}. Because $\mathbf{A}^\text{N} \in \mathbb{R}^{n \times n}$ is symmetric, we hence obtain from \cite[Theorem 8.18]{hunterbook} that
\begin{equation}\label{cond_R_n}
    \mathbf{A}^\text{N} \mathbf{u} = \begin{bmatrix}
    \mathbf{f}\\
    \mathbf{0}
\end{bmatrix} \text{ is solvable} \quad \Leftrightarrow \quad \left\langle \mathbf{w}, \begin{bmatrix}
    \mathbf{f}\\
    \mathbf{0}
\end{bmatrix} \right\rangle_{\mathbb{R}^n} =0 \text{ for all } \mathbf{w} \in \ker{\mathbf{A}^\text{N}}
\end{equation}
where the right-hand side is equivalent to $\sum_{i=1}^m \mathbf{f}_i=0$, see also \cite[Theorem 4.7.3]{Hackbusch}. Due to \Cref{Lemma_A_N}, two solutions $\mathbf{u}_1$ and $\mathbf{u}_2$ of the left-hand side only differ by a constant, i.\,e., $\mathbf{u}_1-\mathbf{u}_2=c \mathbf{1}$, $c \in \mathbb{R}$. \\

It turns out that the above observations are in one-to-one correspondence to our nonlocal theory for problem (\ref{BP2}),
\begin{equation}\tag{BPN}\label{BP2}
    \mathcal{L}_{d,h}u = f \text{ in } \Omega, \quad \mathcal{N}_{d,h}u = 0 \text{ in } \Gamma.
\end{equation}
Recall from  \Cref{Sec_Neumann_Problem} and \Cref{Well_posedness_Neumann} in particular that in case of the validity of the nonlocal Poincaré inequality in $V(\Omega,K_{d,h},\mu)$, we have
\begin{equation}\label{cond_functionspace}
    \text{Problem (\ref{BP2}) has a weak solution} \ \ \Leftrightarrow  \ \ \text{$f \in L^2(\Omega,\mu)$ satisfies (\ref{compatibility}, $g=0$).} 
\end{equation}
If $u_1, u_2 \in V(\Omega,K_{d,h},\mu)$ are two solutions of problem (\ref{BP2}), then $u_1-u_2 \in \ker{\mathcal{B}}$. In view of \Cref{identification_funktionenraum}, \Cref{Lemma_B_Identifikation}, \Cref{lemma_Stencil_equivalence_poblems}, and \Cref{Lemma_A_N}, the following accordances hold:  

\begin{Rem}\label{corresp}
    \hfill
    \begin{enumerate}
        \item By \Cref{lemma_Stencil_equivalence_poblems}, the left-hand side of (\ref{cond_R_n}) is equivalent to the left-hand side of (\ref{cond_functionspace}).
        \item The right-hand side of (\ref{cond_R_n}) is equivalent to the right-hand side of (\ref{cond_functionspace}) because in the setting of this section, the compatibility condition (\ref{compatibility}, $g=0$) translates into
        \begin{equation}\label{comp}
            \int_\Omega f(x) \, \mathrm{d}\mu(x) = \sum_{i=1}^m f(x_i) =0.
        \end{equation}
        \item We have $u_1-u_2 \in \ker{\mathcal{B}}$ (nonlocal function space formulation) if and only if $\mathbf{u}_1 - \mathbf{u}_2 = c \mathbf{1}$ for a constant $c>0$ (discrete $\mathbb{R}^n$-formulation), see \Cref{identification_funktionenraum}, \Cref{Lemma_B_Identifikation}, and \Cref{Lemma_A_N}.
        \item By the same statements, we also see that (2.) from \Cref{Lemma_A_N} yields the existence of a constant $C>0$ such that for all $v \in V(\Omega,K_{d,h},\mu)$ satisfying $ \int_{\Omega \cup \Gamma} v(x) \, \mathrm{d}\mu(x) =0$, the inequality
\begin{equation}\label{Hilfs_Poincare}
    \norm{v}_{L^2(\Omega \cup \Gamma, \mu)}^2 \leq C \,  \mathcal{B}(v,v)
\end{equation}
holds. As shown in \Cref{theorem_equivalence_inequalities} below, the validity of (\ref{Hilfs_Poincare}) is exactly the validity of the nonlocal Poincaré inequality in $V(\Omega,K_{d,h},\mu)$.
\end{enumerate}
\end{Rem}

\begin{Prop}\label{Projection_L2_Omega_Gamma}
We have $V(\Omega,K_{d,h},\mu)= L^2(\Omega \cup \Gamma, \mu)$; especially, the norms $\norm{\cdot}_{V(\Omega,K_{d,h},\mu)}$ and $\norm{\cdot}_{L^2(\Omega \cup \Gamma, \mu)}$ are equivalent.
\end{Prop}

\begin{proof}
     Easy, for details see \cite[Lemma 4.2.8 respectively Lemma 8.1.23]{Huschens}.
\end{proof}

\begin{theorem}\label{theorem_equivalence_inequalities}
    The nonlocal Poincaré inequality holds in $V(\Omega,K_{d,h},\mu)$ if and only if a constant $C>0$ exists such that
    \begin{equation}\label{inequality_euivalent_3}
        \int_{\Omega \cup \Gamma} v(x)^2 \, \mathrm{d}\mu(x) \leq C \, \mathcal{B}(v,v)
    \end{equation}
    is valid for all $v \in V(\Omega,K_{d,h},\mu)$ satisfying $ \int_{\Omega \cup \Gamma} v(x) \, \mathrm{d}\mu(x) =0$.
\end{theorem}

\begin{proof}
     At first, assume (\ref{inequality_euivalent_3}) to hold and let $v \in V(\Omega,K_{d,h},\mu)= L^2(\Omega \cup \Gamma, \mu)$ be arbitrary. By \Cref{Projection_L2_Omega_Gamma} and the Hilbert projection theorem \cite[Lemma 4.1]{stein2009real}, a unique element $Q(v)\in \ker{\mathcal{B}}$ exists such that $\int_{\Omega \cup \Gamma} v(x)-Q(v)(x) \, \mathrm{d}\mu(x) =0$. Consequently, (\ref{inequality_euivalent_3}) yields
     \begin{align*}
         \int_\Omega \int_{\mathbb{R}^d} \big(v(x)-v(y)\big)^2 \, K_{d,h}(x, \mathrm{d}y) \, \mathrm{d}\mu(x) &\geq \mathcal{B}(v,v) = \mathcal{B}\big(v-Q(v),v-Q(v)\big) \notag \\
         &\geq \frac{1}{C} \int_{\Omega \cup \Gamma} \big(v(x)-Q(v)(x)\big)^2 \, \mathrm{d}\mu(x) \notag \\
         &\geq \frac{1}{C}\int_{\Omega} \big(v(x)-Q(v)(x)\big)^2 \, \mathrm{d}\mu(x)\notag \\
         &\geq \frac{1}{C} \inf_{w \in \ker{\mathcal{B}}} \int_\Omega \big(v(x)-w(x)\big)^2 \, \mathrm{d}\mu(x) \label{eq_Poincare1}
    \end{align*}
     which implies the nonlocal Poincaré inequality to be satisfied in $V(\Omega,K_{d,h},\mu)$. Vice versa, now assume the latter
     to hold with Poincaré constant $C>0$ and let \linebreak $v \in V(\Omega,K_{d,h},\mu)$ with 
     $\int_{\Omega \cup \Gamma} v(x) \, \mathrm{d}\mu(x)=0$ be given. Further, let $P(v)$ be the unique element in $\ker{\mathcal{B}}$ for which $v-P(v) \in (\ker{\mathcal{B}})^\perp$. Then, by \Cref{Equivalences_Poincaré}, we obtain
     \begin{align*}
         \mathcal{B}(v,v) &= \frac{1}{2}\mathcal{B}\big(v-P(v),v-P(v)\big) + \frac{1}{2}\mathcal{B}\big(v-P(v),v-P(v)\big) \\
         &\geq \frac{1}{2}\mathcal{B}\big(v-P(v),v-P(v)\big) + \frac{1}{2C} \int_\Omega \big(v(x)-P(v)(x)\big)^2 \, \mathrm{d}\mu(x)\\
         &\geq \frac{1}{2}\min\left\{1,\frac{1}{C}\right\} \norm{v-P(v)}^2_{V(\Omega,K_{d,h},\mu)}. 
     \end{align*}
     Because the norms $\norm{\cdot}_{V(\Omega,K_{d,h},\mu)}$ and $\norm{\cdot}_{L^2(\Omega \cup \Gamma, \mu)}$ are equivalent, i.\,e., there is a constant $\alpha>0$ such that $\norm{u}_{V(\Omega,K_{d,h},\mu)} \geq \alpha \norm{u}_{L^2(\Omega \cup \Gamma, \mu)}$ holds for all $u \in V(\Omega,K_{d,h},\mu)$, we can furthermore estimate 
     \begin{align*}
         \mathcal{B}(v,v) &\geq \frac{1}{2}\min\left\{1,\frac{1}{C}\right\} \norm{v-P(v)}^2_{V(\Omega,K_{d,h},\mu)}  \geq \frac{\alpha}{2}\min\left\{1,\frac{1}{C}\right\}\norm{v-P(v)}^2_{L^2(\Omega \cup \Gamma, \mu)} \\
         &=  \frac{\alpha}{2}\min\left\{1,\frac{1}{C}\right\}\int_{\Omega \cup \Gamma} \big(v(x)-P(v)(x)\big)^2 \, \mathrm{d}\mu(x)\\
         &=  \frac{\alpha}{2}\min\left\{1,\frac{1}{C}\right\} \int_{\Omega \cup \Gamma} v(x)^2 \, \mathrm{d}\mu(x) - \alpha \min\left\{1,\frac{1}{C}\right\} \int_{\Omega \cup \Gamma} v(x)P(v)(x) \, \mathrm{d}\mu(x) \\
         & \quad + \frac{\alpha}{2}\min\left\{1,\frac{1}{C}\right\} \int_{\Omega \cup \Gamma} P(v)(x)^2 \, \mathrm{d}\mu(x)\Big)
     \end{align*}
     where 
     \begin{align*}
         \int_{\Omega \cup \Gamma} v(x)P(v)(x) \, \mathrm{d}\mu(x) =c \int_{\Omega \cup \Gamma} v(x) \, \mathrm{d}\mu(x) =0
     \end{align*}
     because $P(v) \in \ker{\mathcal{B}}$, i.\,e., $c>0$ exists with $P(v)(x_1) = ... = P(v)(x_n) = c$ for all $i=1,...,n$. Thus,
     \begin{align*}
     \frac{\alpha}{2}\min\left\{1,\frac{1}{C}\right\} \Bigg(\int_{\Omega \cup \Gamma} v(x)^2 \, \mathrm{d}\mu(x) &+ \int_{\Omega \cup \Gamma} P(v)(x)^2 \, \mathrm{d}\mu(x)\Bigg) \\
     &\geq \frac{\alpha}{2}\min\left\{1,\frac{1}{C}\right\} \int_{\Omega \cup \Gamma} v(x)^2 \, \mathrm{d}\mu(x),
     \end{align*}
    as desired. 
\end{proof}

\subsection{The Discrete Maximum Principle}

Finally, also the alignment of the weak maximum principle from \Cref{Sec_Maximum} with an existing discrete maximum principle for finite difference operators shall be discussed. More precisely, we show that in the context of this section, \Cref{Max_Principle} is fully consistent with what is commonly called the global discrete maximum principle (\Cref{global_discrete_MP} below).

\begin{definition}
    A matrix $\mathbf{A}\in \mathbb{R}^{m \times n}$, $m,n \in \mathbb{N}$, is said to be of \textbf{non-negative type} if
    \begin{align}
        a_{ij} &\leq 0 \quad \quad  \text{for all } 1 \leq i \leq m, \ 1 \leq j \leq n, \ i \neq j,\\
        \sum_{j=1}^n a_{ij} & \geq 0 \quad \quad  \text{for all } 1 \leq i \leq m. \label{nonnegative2}
    \end{align}
\end{definition}

Matrices of non-negative type must not be confused with non-negative matrices as studied e.g. in \cite[Chapter 2]{varga1962iterative}. By definition, the matrix
\[\mathbf{A}^{\text{red}} := \frac{1}{h^2} \begin{bmatrix}
    A^\Omega & A^\Gamma
\end{bmatrix} \in \mathbb{R}^{m \times n}\]
which consists of the upper block matrices from $\mathbf{A}^\text{D}$ and $\mathbf{A}^\text{N}$, respectively, is of non-negative type and $\sum_{j=1}^n a_{ij}=0$ holds. Recall from \Cref{Rem_Friedrichs_Matrix} that the matrix $A^\Omega$ is non-singular if and only if the nonlocal Friedrichs inequality holds in $V(\Omega,K_{d,h},\mu)$.\\

Let $u \in V(\Omega,K_{d,h},\mu)$ and $v \in V_0(\Omega,K_{d,h},\mu)$ with $v \geq 0$ be given. Then, by \Cref{identification_funktionenraum}, we have $\mathbf{v}= \begin{bmatrix}
    \mathbf{\Tilde{v}} & 0
\end{bmatrix}^T$ for $\mathbf{\Tilde{v}}\in \mathbb{R}^m$ and \Cref{Lemma_B_Identifikation} yields that $\mathcal{B}(u,v) \leq 0$ is equivalent to
\begin{align*}
    \mathbf{v}^\top \mathbf{A}^\text{N} \mathbf{u} = \mathbf{\Tilde{v}}^T \left(\frac{1}{h^2} \begin{bmatrix}
        A^\Omega & A^\Gamma
    \end{bmatrix}\right) \mathbf{u}.
\end{align*}
If inserting $\mathbf{\Tilde{v}}= \mathbf{\tilde{e}_i}$ where $\mathbf{\Tilde{e}_i}$ is the $i$-th standard unit vector in $\mathbb{R}^m$, it becomes apparent that 
\begin{equation*}
    \mathcal{B}(u,v) \leq 0 \text{ for all } v \in V(\Omega,K_{d,h},\mu), \ v \geq 0
 \ \ \Leftrightarrow \ \ \sum_{j=1}^n a_{ij}\mathbf{u}_j \leq 0 \text{ for all } i=1,...,m. 
 \end{equation*}
Summming up all these observations and choosing $\mathbf{A}= \mathbf{A}^{\text{red}}$, \Cref{Max_Principle} is the nonlocal reformulation of \Cref{global_discrete_MP} below  which can be considered a generalization of \cite[Theorem 3]{ciarlet1970discrete}.

\begin{theorem}{\textnormal{ \textbf{(Global Discrete Maximum Principle, e.g. \cite[Theorem 3.5]{Barr})}}}\label{global_discrete_MP} 
    Let the matrix $\mathbf{A} = (a_{ij})_{j=1,...,n}^{i=1,...,m} \in \mathbb{R}^{m \times n}$ with $m < n$ be of non-negative type and assume that the matrix $A^I =(a_{ij})_{j=1,...,m}^{i=1,...,m}$ is non-singular. If $\sum_{j=1}^n a_{ij}=0$, we have
    \begin{equation}\label{gdMP_aussage}
        \sum_{j=1}^n a_{ij}\mathbf{u}_j \leq 0, \quad i=1,...,m \quad \Rightarrow \quad \max_{i=1,...,n} \mathbf{u}_i \leq \max_{i=m+1,...,n} \mathbf{u}_i.
    \end{equation}   
\end{theorem}

\textbf{Remark:} In case that the bilinear form $\mathcal{B}$ stems from the regularized problem (\ref{BP1}) respectively (\ref{BP2}) where $\mathcal{L}_{d,h}u$ is replaced by its regularized version \linebreak $\mathcal{L}_{d,h}u +cu$, $c \in L^\infty(\Omega, \mu)$, condition (\ref{nonnegative2}) is no longer satisfied with equality. In this case, (\ref{gdMP_aussage}) has to be weakened to
\begin{equation}
    \sum_{j=1}^n a_{ij}\mathbf{u}_j \leq 0, \quad i=1,...,m \quad \Rightarrow \quad \max_{i=1,...,n} \mathbf{u}_i \leq \max_{i=m+1,...,n} \mathbf{u}_i^+
\end{equation}
where $\mathbf{u}_i^+:= \max\{\mathbf{u}_i,0\}$, in accordance with \Cref{Rem_MP}, see \cite[Theorem 3.5]{Barr}.

\section{Summary}

In this article, we introduced a theory for tackling nonlocal equations of the form $\mathcal{L}u = f$ on $\Omega$ which are equipped with either a Dirichlet- or a Neumann-type boundary condition and whose governing operator is determined by a symmetric transition kernel $K$. While in \Cref{Sec_Symmetry}, an adequate notion of symmetry was implemented in order to establish a weak formulation of the problem that is facilitated by a nonlocal integration by parts formula, sufficient conditions for its well-posedness were derived in \Cref{Sec_Dirichlet_Problem} and \Cref{Sec_Neumann_Problem} by building upon the nonlocal inequalities from \Cref{section_Inequalities} and the Lax-Milgram theorem. In \Cref{Sec_Maximum}, the corresponding weak solutions were then shown to obey a weak maximum principle before, finally, our theory was made consistent with well-known results centering around the discrete Poisson problem.

\bibliographystyle{plain} 
\bibliography{references} 
\end{document}